\documentclass[12pt]{amsart}

\usepackage{macros_trianguline_springer}

\calclayout

\numberwithin{equation}{section}

\title[Tangent spaces on the trianguline variety]{The trianguline variety, tangent spaces and the Grothendieck-Springer resolution}

\author{Seginus Mowlavi}
\address{Université Paris-Saclay, CNRS, Laboratoire de mathématiques d’Orsay, 91405, Orsay, France}

\begin{document}

\begin{abstract}
By the work of Breuil-Hellmann-Schraen, we know that the trianguline variety contains crystalline companion points which are parametrised by pairs $w\preceq w_\mathrm{sat}$ of permutations where $\preceq$ is the Bruhat order. We first define and study a certain combinatorial property of a pair $w'\preceq w$ of permutations in the context of Weyl groups of root systems. We call good pairs the pairs satisfying this property (which is the vast majority of pairs) and bad pairs the other ones. We then give an exact formula for the dimension of the tangent space to the trianguline variety at (generic) crystalline companion points such that $w\preceq w_\mathrm{sat}$ is a good pair. The method (due to Breuil-Hellmann-Schraen) is to first compute an analogous dimension for a local model of the trianguline variety built out of Grothendieck's simultaneous resolution. To achieve this, we prove a conjecture of Breuil-Hellmann-Schraen, describing the intersection of the closure of a Schubert cell with another Schubert cell on this local model (in the context of an arbitrary split reductive group), when this pair of cells is parametrised by a good pair of permutations. We give counter-examples to this conjecture for an infinite family of pairs of cells (associated to bad pairs).
\end{abstract}

\maketitle

\tableofcontents

\section{Introduction} \label{secIntro}

Let $k$ be a field of characteristic $0$. In their study \cite{bhs3} of the trianguline variety $\Xtri$ of Hellmann \cite{hellmann}, Breuil-Hellmann-Schraen introduced a variety $X$ over $k$ associated to a split reductive group $G$ over $k$. This variety admits a stratification $X=\coprod_{w\in W}V_w$ indexed by the Weyl group $W$ of $G$, similar to the Schubert decomposition of the flag variety. Denoting by $X_w$ the Zariski-closure in $X$ of the locally closed subvarieties $V_w$, the authors of \cite{bhs3} have linked the dimension of the tangent spaces of $\Xtri$ at certain points of interest to the dimension of $T_{X_w,x}$ at particular points $x\in X_w\inter V_{w'}$ for $w,w'\in W$ such that $w'\preceq w$ where $\preceq$ is the Bruhat order. They also formulate a conjecture describing $X_w\inter V_{w'}$, which if true allows to compute the dimension of $T_{X_w,x}$. This conjecture is the object of the present paper.

A special case of a variant of this conjecture was already known by the work of Ginsburg \cite{ginsburg} using the Beilinson-Bernstein correspondence. By applying this special case to certain Levi subgroups of $G$, we are able to expand this result into a proof of many cases of the conjecture.

These cases are characterised by a combinatorial condition on the pair $(w',w)$ that we call being a ``good pair''. We develop the theory of good pairs in the context of Weyl groups of root systems and establish along the way different criteria for good pairs. We also show that the set of good pairs is related to pattern avoidance.

Finally, by spelling out explicit equations defining $V_w$ as a quasiprojective variety in the case $G=\GL_{n/k}$, we give a family of counter-examples to the conjecture.

We now explain our results in more detail.

\subsubsection*{Good pairs in a Weyl group}

Let $\Phi$ be an abstract root system; we write $W\coloneqq W(\Phi)$ for its Weyl group. For $w\in W$, let $\Gamma_w\coloneqq\sum_{\alpha\in\Phi}\Z(w(\alpha)-\alpha)$, let $E_w\coloneqq\Gamma_w\tens_{\Z}\Q$ and finally let $\Phi_w\coloneqq\Phi\inter E_w$. We call $\Phi_w$ the \emph{minimal generating root subsystem of $w$}. It has a clear interpretation in terms of reductive groups:

\begin{prop}[Proposition \ref{propCycleDecompositionCentraliser}] \label{propIntroCycleDecompositionCentraliser}
Let $(G,T)$ be a connected split reductive algebraic group over a field $k$ of characteristic 0 whose root system is isomorphic to $\Phi$. Then the centraliser $\left(C_G((T^w)^\circ),T\right)$ is a connected split reductive group whose root system is isomorphic to $\Phi_w$.
\end{prop}

Fixing a basis $I$ of $\Phi$ gives a length function $\lg$ and a Bruhat order $\prec$ on $W$. The central notion we study in \S\ref{secGoodPairs} and use in \S\ref{secProofGoodPairs} is the following:

\begin{dfn}[Definition \ref{dfnGoodPair}] \label{dfnIntroGoodPair}
Let $(w_1,w_2)\in W$ such that $w_1\preceq w_2$. If there exist roots $\alpha_1,\ldots,\alpha_r\in\Phi_{w_1w_2^{-1}}$ which give rise to an ascending sequence $w_1\prec w_1s_{\alpha_1}\prec\ldots\prec w_1s_{\alpha_1}\ldots s_{\alpha_r}=w_2$, then $(w_1,w_2)$ is called a \emph{good pair}; otherwise, it is called a \emph{bad pair}.
\end{dfn}

In \S\ref{ssecMinGenSys}, we establish a few combinatorial results about minimal generating subsystems $\Phi_w$. In \S\ref{ssecBonnePaireDef} we prove a criterion for good pairs which is motivated by the fact that, for any $w\in W$, the subsystem $\Phi_w$ is conjugated to a standard subsystem, \emph{i.e.\ }a root subsystem $\Phi_J\subseteq\Phi$ generated by some subset $J\subseteq I$ of simple roots. More precisely, for $J\subseteq I$, let $W_J$ be the subgroup of $W$ generated by $\enstq{s_\alpha}{\alpha\in J}$, and let $W^J$ be the set of representatives of minimal length for $W_J\backslash W$. Then one can write $\Phi_w=u^J(\Phi_J)$ for some $J\subseteq I$ and $u^J\in W^J$. The criterion is:

\begin{prop}[Proposition \ref{propGoodPairsMinipermutations}] \label{propIntroGoodPairsMinipermutations}
Let $(w_1,w_2)\in W^2$. Write $\Phi_{w_2w_1^{-1}}=u^J(\Phi_J)$ for some $J\subseteq I$ and $u^J\in W^J$. Then there exist $v^J\in W^J$ and $w_{J,1},w_{J,2}\in W_J$ such that
\[
	w_1=u^Jw_{J,1}(v^J)^{-1} \,, \quad w_2=u^Jw_{J,2}(v^J)^{-1} \,,
\]
and the pair $(w_1,w_2)$ is good if and only if $w_{J,1}\preceq w_{J,2}$.
\end{prop}

We conclude in \S\ref{ssecPatternAvoidance} with an application of this last criterion in the special case where $\Phi$ is of type $A_{n-1}$ for some $n\in\Z_{>1}$, \emph{i.e.\ }$W=\Scal_n$ is the symmetric group. It is stated in terms of what is called \emph{pattern avoidance}. Here we say that a permutation $w\in\Scal_n$ has a pattern $f\in\Scal_m$ when there exist $1\leq i_1<\ldots<i_m\leq n$ such that the $m$-uples $(w(i_1),\ldots,w(i_m))$ and $(f(1),\ldots,f(m))$ are in the same order; see Definition \ref{dfnPattern}. This can in fact be stated in the general setting of an abstract root system, and many properties such as smoothness of Schubert varieties have been found to be instances of pattern avoidance; see for example \cite[\S2, \S4]{billeyPostnikov}. We prove that, at least for root systems of type $A$, being a good pair is also an instance of pattern avoidance:

\begin{thm}[Theorem \ref{thmBadPairPatterns}] \label{thmIntroBadPairPatterns}
Let $w\in\Scal_n$.
\begin{enumerate}[label=\emph{(\arabic*)},ref=(\arabic*)]
\item
The pair $(w',w)$ is good for any $w'\in\Scal_n$ if and only if $w$ avoids the four patterns $[4231]$, $[42513]$, $[35142]$ and $[351624]$.
\item
The pair $(w,w'')$ is good for any $w''\in\Scal_n$ if and only if $w$ avoids the four patterns $[1324]$, $[24153]$, $[31524]$ and $[426153]$.
\end{enumerate}
\end{thm}

Linking these patterns with those deciding smoothness of Schubert varieties, we obtain that, if $(w,w')\in(\Scal_n)^2$ is a bad pair, then the Schubert variety $\overline{BwB/B}$ in the flag variety $G/B$ is singular (but not necessarily at the point $w'B\in\overline{BwB/B}$); see Corollary \ref{corBadPairSingularSchubert}.

\subsubsection*{The conjecture of \cite{bhs3}}

Let $G$ be a connected split reductive group over $k$, $T$ a split maximal torus inside $G$, $B$ a Borel subgroup of $G$ containing $T$ and $U$ the unipotent subgroup of $B$. We write $\gfrak$, $\tfrak$, $\bfrak$ and $\ufrak$ for their respective Lie algebras, and $W\coloneqq N_G(T)/T$ for the Weyl group of $(G,T)$. For $w\in W$, we choose a lifting of $w$ in $N_G(T)$ which we still call $w$. We have the Schubert decomposition $G/B=\coprod_{w\in W}BwB/B$ of $G/B$ into locally closed strata $BwB/B$, where a cell $Bw'B/B$ intersects the Zariski-closure $\overline{BwB/B}$ of $BwB/B$ if and only if $w'\preceq w$ for the Bruhat order on $W$, in which case $Bw'B/B\subseteq\overline{BwB/B}$.

Let $G\times^B\gfrak$ be the quotient of $G\times\gfrak$ by the right-action of $B$ given by $(g,\psi)b=(gb,\Ad(b^{-1})\psi)$. We define the morphisms $q\colon G\times^B\gfrak\to\gfrak$ and $\widetilde{\pi}\colon G\times^B\gfrak\to G/B$ by $q(g,\psi)=\Ad(g)\psi$ and $\widetilde{\pi}(g,\psi)=gB$ ($q$ is called Grothendieck's simultaneous resolution of singularities). We then define the algebraic variety $\widetilde{X}\coloneqq q^{-1}(\bfrak)$. For $w\in W$, we consider the locally closed subvariety $\widetilde{V}_w\coloneqq\widetilde{X}\inter\widetilde{\pi}^{-1}(BwB/B)$ of $\widetilde{X}$ as well as its Zariski-closure $\widetilde{X}_w$ in $\widetilde{X}$; they give a cell decomposition of $\widetilde{X}$.

When $L$ is another split reductive group over $k$, we add the subscript $L$ to $q,\widetilde{\pi},\widetilde{X},\widetilde{V}_w,\widetilde{X}_w$ (for $w\in W(L)$) to mean the same definition with the reductive group $L$ in place of the reductive group $G$. When nothing is specified, we mean $G$.

Similarly to Schubert cells in $G/B$, the intersection $\widetilde{V}_{w'}\inter\widetilde{X}_w$ is non-empty if and only if $w'\preceq w$. The difference is that in this case $\widetilde{V}_{w'}$ is not contained in $\widetilde{X}_w$. More precisely, it is not difficult to show that $\widetilde{V}_{w'}\inter\widetilde{X}_w\subseteq\widetilde{V}_{w'}\inter q^{-1}(\tfrak^{ww'^{-1}}+\ufrak)$ (\cite[Lemma 2.3.4]{bhs3}) where $\tfrak^{ww'^{-1}}\subseteq \tfrak$ is the closed subvariety of fixed points by $ww'^{-1}$ and Breuil-Hellmann-Schraen conjecture that this is an equality:

\begin{conj}[Conj.\ 2.3.7 of \cite{bhs3}] \label{conjIntroLocalModelSimple}
Let $w,w'\in W$. If $w'\preceq w$, then
\[
	\widetilde{V}_{w'}\inter\widetilde{X}_w = \widetilde{V}_{w'}\inter q^{-1}\left(\tfrak^{ww'^{-1}}+\ufrak\right) \,.
\]
\end{conj}

Conjecture \ref{conjIntroLocalModelSimple} first appeared in the arithmetic context of the study of the rigid analytic \emph{trianguline variety} $\Xtri$ introduced by Hellmann \cite{hellmann} (here $\bar{r}$ is a mod $p$ representation of $\Gal(\overline K/K)$ where $p$ is a prime number and $K$ a finite extension of $\Q_p$). More precisely, Breuil-Hellmann-Schraen proved in \cite{bhs3} that the varieties $\widetilde{X}_w$ are local models of $\Xtri$ at \emph{crystalline regular generic points}, which allow them to derive a number of consequences about the local geometry of $\Xtri$ at these points.

We show in \S\ref{secEquations} that Conj.\ \ref{conjIntroLocalModelSimple} is actually \emph{false} in general! In \S\ref{secProofGoodPairs}, we nonetheless prove it in the case of good pairs:

\begin{thm}[Theorem \ref{thmConjIsTrue}] \label{thmIntroConjIsTrue}
Let $w,w'\in W$ such that $w'\preceq w$ and $(w',w)$ is a good pair. Then, for all $t\in\tfrak$ fixed by $ww'^{-1}$, we have
\begin{equation} \label{eqIntroConjLocalModel}
	\widetilde{V}_{w'}\inter q^{-1}(t+\ufrak) \subseteq \widetilde{X}_w \,.
\end{equation}
Hence Conj.\ \ref{conjIntroLocalModelSimple} for $(w',w)$ is true.
\end{thm}

Going back to the motivation of \cite{bhs3}, Theorem \ref{thmIntroConjIsTrue} has the following application:

\begin{cor}[Theorem \ref{thmTrianguline}] \label{corTriangulineIntro}
Let $x\in\Xtri$ be a crystalline regular generic point and $T_{\Xtri,x}$ the tangent space to $\Xtri$ at $x$. Let $w_x, w$ be as in \cite[\S4.1]{bhs3} ($w_x$ and $w$ are elements of a fixed product of Weyl groups of type $A$) and assume that $(w_x,w)$ is a good pair. Then
\begin{equation} \label{eqDimensionIntro}	
\dim_{k(x)}T_{\Xtri,x} = \dim\Xtri - d_{ww_x^{-1}} + \dim_{\Q_p}T_{\overline{(BwB/B)},w_xB} - \length(w_x)
\end{equation}
where $k(x)$ is the residue field of $\Xtri$ at $x$, $\overline{BwB/B}$ is the Schubert variety associated to $w$ and where we refer to Definition \ref{dfnMinGenSys} for the integer $d_{ww_x^{-1}}$.
\end{cor}

In \cite{chap3}, we proved Corollary \ref{corTriangulineIntro} under the assumption of modularity lifting conjectures; here we do not need this assumption. Interestingly, while they use fundamentally different methods, both works prove the exact same cases:  the notions of good pair are stated differently (here as in Proposition \ref{propGoodPairsMinipermutations}, and in \cite{chap3} as in Definition \ref{dfnGoodPair}) but in equivalent ways. Note that it was already known that the right hand side in \eqref{eqDimensionIntro} is an upper bound for $\dim_{k(x)}T_{\Xtri,x}$ (even if $(w_x,w)$ is a bad pair), see \cite[Proposition 4.1.5(ii)]{bhs3}. We do not know if this upper bound becomes strict for bad pairs (see however Proposition \ref{propSmoothnessXtri}).

Before outlining the proof of Theorem \ref{thmIntroConjIsTrue}, we remark that it suffices to prove \eqref{eqIntroConjLocalModel} for all ``sufficiently generic'' (closed) points $t\in\tfrak^{ww'^{-1}}$: see Proposition \ref{propGenericTorus} and Lemma \ref{lemDenseIsEnough}. Without loss of generality, we can also assume $k$ algebraically closed; see \S\ref{ssecExtensionLevi}. The proof of Theorem \ref{thmIntroConjIsTrue} then proceeds in three steps.

Firstly, as $(w',w)$ is a good pair, Proposition \ref{propIntroGoodPairsMinipermutations} gives a subset $J\subseteq I$ of simple roots of $\Phi(G,T)$ and elements $u^J,v^J\in W^J$, $w_J,w'_J\in W_J$, such that $w=u^Jw_J(v^J)^{-1}$, $w'=u^Jw'_J(v^J)^{-1}$ and $w'_J\preceq w_J$ in $W_J$. Let $L$ be the Levi subgroup of the standard parabolic subgroup $P_J$ of $G$, then $L=C_G\big(\big((u^J)^{-1}T^{ww'^{-1}}u^J\big)^\circ\big)$ using Proposition \ref{propIntroCycleDecompositionCentraliser}.

It is proved in \cite[\S2.4]{bhs3} using the Beilinson-Bernstein localisation that \eqref{eqIntroConjLocalModel} holds for $t=0$ for all $w'\preceq w$ (even in the case of bad pairs). Since $w'_J\preceq w_J$, we can apply this to the connected split reductive group $L$ with split maximal torus $T$, Borel subgroup $B_L\coloneqq B\inter L$ and unipotent subgroup $U\coloneqq U\inter L$. For any $t\in\tfrak_J=\Ad(u^J)^{-1}\tfrak^{ww'^{-1}}$, we deduce by translating by $t$ that $\widetilde{V}_{L,w'_J}\inter q^{-1}(t+\ufrak_L) \subset \widetilde{X}_{L,w_J}$.

Secondly, by carefully decomposing the adjoint action of $G$ on some elements of $\lfrak\inter\bfrak$ (where $\lfrak$ is the Lie algebra of $L$), we show:

\begin{prop}[Proposition \ref{propLeviToG}] \label{propIntroLeviToG}
Let $u^J,v^J\in W^J$. One can define a morphism $\iota_{t,u^J,v^J}\colon L\times^{L\inter B}\lfrak\inter\bfrak\to G\times^B\bfrak$ by $\iota_{t,u^J,v^J}(l,\psi_L))=\left(u^Jl(v^J)^{-1},\Ad(v^J)\psi_L\right)$. For any $t\in\Ad(u^J)^{-1}\tfrak^{ww'^{-1}}$ sufficiently generic and any $w_J\in W_J$, this morphism $\iota_{t,u^J,v^J}$ induces an isomorphism of algebraic varieties
\[
	\iota_{t,u^J,w_J,v^J} \colon \widetilde{V}_{L,w_J}\inter q_L^{-1}(t+\ufrak_L) \isoto \widetilde{V}_{u^Jw_J(v^J)^{-1}}\inter q^{-1}\left(\Ad(u^J)(t+\ufrak_L)\right) \,.
\]
\end{prop}

Thirdly, using standard results about Jordan decompositions in $\gfrak$ (that we show in Appendix \ref{appendixJordan} as we couldn't totally find them in the literature), we prove:

\begin{lem}[Lemma \ref{lemJordan}] \label{lemIntroJordan}
Let $t\in\Ad(u^J)^{-1}\tfrak^{ww'^{-1}}$ be sufficiently generic, let $u^J\in W^J$. Then there is a surjective morphism given by
\begin{equation*}
	f \colon
	\begin{array}{ccc}
		U\times\left(\Ad(u^J)(t+\ufrak_L)\right) &\surj &\Ad(u^J)t+\ufrak \\
		(u,x) &\longmapsto &\Ad(u)x \,.
	\end{array}
\end{equation*}
\end{lem}

The surjective morphism $f$ of Lemma \ref{lemIntroJordan} combined with the isomorphism of Proposition \ref{propIntroLeviToG} induces two morphisms of varieties $U\times\left(\widetilde{V}_{L,w_J}\inter q_L^{-1}(t+\ufrak_L)\right)\surj\widetilde{V}_w\inter q^{-1}\left(\Ad(u^J)t+\ufrak\right)$ and $U\times\widetilde{X}_{L,w_J}\to\widetilde{X}_w$, which fit into a commutative diagram
\[
	\begin{tikzcd}
		U\times\left(\widetilde{V}_{L,w'_J}\inter q_L^{-1}(t+\ufrak_L)\right) \arrow[r,"\subseteq"] \arrow[dd,two heads] & U\times\widetilde{X}_{L,w_J} \arrow[d] \\
		 & \widetilde{X}_w \arrow[d,hook,"\subseteq"] \\
		\widetilde{V}_{w'}\inter q^{-1}(\Ad(u^J)t+\ufrak) \arrow[r,"\subseteq"] & \widetilde{X}
	\end{tikzcd}
\]
where the top horizontal arrow is given by the first step. The surjectivity of the left vertical arrow then proves that the image of the bottom horizontal arrow is contained in $\widetilde{X}_w$, which is the statement \eqref{eqIntroConjLocalModel}.

\subsubsection*{Counter-examples}

In the last \S\ref{secEquations}, we give explicit equations and inequations for the projective varieties $\widetilde{X}$ and $\widetilde{V}_w$ in the case $G=\GL_{n/k}$. Even though we do not have such an explicit description for the varieties $\widetilde{X}_w$, this is enough to find a family of counter-examples to Conj.\ \ref{conjIntroLocalModelSimple}.

We use that $\widetilde{X}$ can be seen as the subvariety $\enstq{(gB,\psi)\in G/B\!\times\!\bfrak}{\Ad(g^{-1})\psi\in\bfrak}$ of $G/B\times\bfrak$, see \eqref{defXtilde}, which is easier to handle than a quotient variety. Indeed $\bfrak$ is an affine space, and the flag variety $G/B$ is very well-understood. Specifically, $G/B$ is a subvariety of a product of Grassmannians defined by incidence relations; and there are well-known equations defining the Grassmannian as a subvariety of a projective space \emph{via} the Plücker embedding. Therefore $G/B\times\bfrak$ is well-understood as a subvariety of $\prod_{d=1}^{n-1}\Proj(\bigwedge^d k^n)\times\mathbb{A}^{n(n+1)/2}$. Details are spelled out in \S\ref{ssecFlagEquations}.

We then work out a set of equations defining $\widetilde{X}$ inside $G/B\times\bfrak$; this is Lemma \ref{lemEqStableFlag}. The next step is to use the (in)equations defining each Schubert cell $BwB/B$ inside $G/B$ in order to find a set $\{F_i=0,G_j\neq0\}_{i\in I, j\in J}$ of (in)equations defining $\widetilde{V}_w$ inside $\prod_{d=1}^{n-1}\Proj(\bigwedge^d k^n)\times\mathbb{A}^{n(n+1)/2}$. Doing so by brute force is not too difficult; however the challenge is to have enough equations so that, hopefully, the set of equations $\{F_i=0\}_{i\in I}$ describes the Zariski-closure $\widetilde{X}_w$ of $\widetilde{V}_w$ inside $\prod_{d=1}^{n-1}\Proj(\bigwedge^d k^n)\times\mathbb{A}^{n(n+1)/2}$. The result we achieve in Proposition \ref{allEquations} is quite satisfying, as Conj.\ \ref{conjIntroLocalModelSimple} is equivalent to $\widetilde{X}_w$ being the closed subset $\{F_i=0\}_{i\in I}$; see Proposition \ref{conjImpliesConj1}.

On the other hand, through technical simplifications in \S\ref{ssecZariskiClosureEquations}, we improve on Proposition \ref{allEquations} by finding unexpected additional equations satisfied on $\widetilde{V}_w$. As a consequence, we get for a certain family of couples $(w,w')\in(\Scal_n)^2$, a beautifully simple equation satisfied on $\widetilde{X}_w\inter \widetilde{V}_{w'}$ but not on $\widetilde{V}_{w'}\inter q^{-1}(\tfrak^{ww'^{-1}}+\ufrak)$, hence the following result:

\begin{thm}[Theorem \ref{thmConjIsFalse}, Remark \ref{remConjIsFalseDim}, Remark \ref{remConcreteCounterExamples}] \label{thmIntroConjIsFalse}
For any $n\in\Z_{\geq4}$, there exists a bad pair $(w,w')$ in $\Scal_n$ such that
\[
	\dim\left(\widetilde{X}_w\inter\widetilde{V}_{w'}\right) < \dim\left(\widetilde{V}_{w'}\inter q^{-1}(\tfrak^{ww'^{-1}}+\ufrak)\right) \,.
\]
In particular, Conj.\ \ref{conjIntroLocalModelSimple} is false.
\end{thm}

To our knowledge, the question of deciding whether or not Conj.\ \ref{conjIntroLocalModelSimple} is false for \emph{all} bad pairs remains open (we know that this holds at least for $\GL_4$ and $\GL_5$). Another area of further investigation is to find accurate estimates for the difference $\dim\big(\widetilde{V}_{w'}\inter q^{-1}(\tfrak^{ww'^{-1}}+\ufrak)\big)-\dim\big(\widetilde{X}_w\inter\widetilde{V}_{w'}\big)$.

\subsubsection*{Notation} We list here some of the main notation of the paper.

Throughout the paper, $k$ is a field of characteristic $0$ with algebraic closure $k^\mathrm{a}$. We denote by $(G,T)$ a connected split reductive algebraic group over $k$, $B$ a Borel subgroup of $G$ containing $T$ and $U\subset B$ the subgroup of unipotent elements in $B$.

We write $\gfrak$ for the Lie algebra of $G$, by which we mean the affine scheme over $k$ such that $\gfrak(k)$, as a finite-dimensional vector space over $k$, is naturally isomorphic to the tangent space $T_eG$ of $G$ at the neutral element $e\in G$. Similarly, $\tfrak$, $\bfrak$, $\ufrak$, $\ufrak_\alpha$ are the Lie algebras of $T$, $B$, $U$, $U_\alpha$ respectively ($\ufrak_\alpha$ is sometimes also denoted $\gfrak_\alpha$). We write $\Ad\colon G\to\GL(\gfrak)$ and $\ad\colon\gfrak\to\gl(\gfrak)$ for the adjoint representations.

We denote by $X(T)\coloneqq\Hom(T,\Gm)$ the group of algebraic characters of $T$ viewed as a $\Z$-module, and by $\Phi\subset X(T)$ the set of roots of $(G,T)$, \emph{i.e.}\ the set of $\alpha\in X(T)\setminus\{0\}$ such that the eigenspace $\gfrak_\alpha\coloneqq\enstq{x\in\gfrak}{\Ad(t)x=\alpha(t)x\;\forall t\in T}$ is non-zero. For any root $\alpha\in\Phi$, we have $\dim\gfrak_\alpha=1$ (see \cite[Thm.\ 21.11(b)]{milne}). We write $U_\alpha\subset G$ for the root subgroup of $\alpha$ (see \cite[21.10, 21.11]{milne}); its Lie algebra $\ufrak_\alpha$ satisfies $\ufrak_\alpha(k)=\gfrak_\alpha$.

Next, we write $\Gamma\subseteq X(T)$ for the $\Z$-submodule generated by $\Phi$, and $E\coloneqq\Gamma\tens_\Z\Q$. We then write $I$ for the basis of the root system $\Phi$ defined by the Borel $B$ and $\Phi_+$ for the set of positive roots defined by $I$.

We write $W\coloneqq W(\Phi)$ for the Weyl group of $\Phi$. Recall that $W(G,T)\coloneqq N_G(T)/T$ is a finite group with a natural faithful action on $T$ and hence on $X(T)$, which makes $W(G,T)$ equal to $W$ as groups of automorphisms of $X(T)$ (see \cite[Cor.\ 21.38]{milne}). We therefore make the identification $W=N_G(T)/T$. For $w\in W$, we write $\Ad(w)$ for the conjugation morphism on $T$ by any $\dot{w}\in N_G(T)$ which lifts $w$, \emph{i.e.\ }$\Ad(w)t=\dot{w}t\dot{w}^{-1}$ for $t\in T$. Then the action of $W$ on $\Phi$ can be written as $(w\alpha)(t) = \alpha(\Ad(w^{-1}){t})$ for $\alpha\in\Phi$, $w\in W$ and $t\in T$. Finally, $T^w\coloneqq\enstq{t\in T}{\Ad(w)t=t}$ denotes the subgroup of $T$ of invariants under $\Ad(w)$ and $(T^w)^\circ$ denotes its identity component.

In \S\ref{secGoodPairs} (except \S\ref{ssecMinGenSysGroups}), $E$ is simply defined as a vector space over $\Q$ and $\Phi$ as an abstract root system in $E$, independently of any split reductive group $G$. In \S\ref{secEquations}, we restrict to the case $G=\GL_{n/k}$ for some $n\in\Z_{>0}$; then $T$ and $B$ are respectively taken to be the subgroups of diagonal and upper triangular matrices.

From \S\ref{secConjs} onwards, we deal with algebraic varieties over $k$, \emph{i.e.\ }reduced separated schemes of finite type over $k$. Morphisms of varieties are then uniquely determined by their value on $\kalg$-points, and their ``image'' (more precisely, the surjectivity onto a subvariety of the target) can be checked on closed points (see \cite[Ex.\ 7.4.C]{vakil}). Given two varieties $X_1$ and $X_2$ equipped with immersions into a third variety $X$, we write
\[
	X_1\inter X_2 \coloneqq \left(X_1\times_XX_2\right)^\red
\]
when there is no ambiguity. By base change, the natural map $X_1\inter X_2\to X_1$ is a closed (resp.\ locally closed) immersion when $X_2\inj X$ is closed (resp.\ locally closed), and similarly for $X_1\inter X_2\to X_2$.  In the same way, given a morphism of varieties $f\colon X\to Y$ and a closed (resp.\ locally closed) immersion $Y'\inj Y$, we write
\[
	f^{-1}(Y') \coloneqq \left(X\times_YY'\right)^\red \,,
\]
and the natural map $f^{-1}(Y')\to X$ is a closed (resp.\ locally closed) immersion.

\subsubsection*{Acknowledgements}

The results of this paper were obtained during the author's PhD at the Laboratoire de Mathématiques d'Orsay. Much credit is due to his PhD advisor Christophe Breuil for his supervision and for providing key ideas for the proof of Theorem \ref{thmIntroConjIsTrue}. The author would like to thank Ariane Mézard for numerous exchanges which have enlightened the search of counter-examples to Conjecture \ref{conjIntroLocalModelSimple}. The author would also like to thank Benjamin Schraen and Anne Moreau for their help, as well as Gabriel Dospinescu and Ruochuan Liu for their careful review of the author's PhD dissertation.

\numberwithin{thm}{subsection}
\numberwithin{equation}{section}

\section{Some combinatorics in Weyl groups} \label{secGoodPairs}

In this section, we study particular subsystems of a root system $\Phi$ associated with an element of the Weyl group $W(\Phi)$. We then define the notions of \emph{good pairs} and \emph{bad pairs} in $W(\Phi)$. Finally, we explicitly describe the case of systems of type $A$.

\subsection{Minimal generating root subsystems} \label{ssecMinGenSys}

Let $E$ be a finite dimensional a vector space over $\Q$ and let $\Phi$ be a root system inside $E$ (see \cite[Chapter VI, \S1]{bourbakiLie456}). Write the (finite) Weyl group of $\Phi$ as $W\subset\End_\Q(V)$. For a root $\alpha\in\Phi$, we denote by $s_\alpha\in W$ its associated reflection.

\begin{dfn} \label{dfnMinGenSys}
Let $w\in W$, $\Gamma_w$ the $\Z$-submodule of $E$ generated by $w(\alpha)-\alpha$ for $\alpha\in\Phi$, $E_w\coloneqq\Gamma_w\tens_\Z\Q=\im(w-\id_E)$ and $\Phi_w\coloneqq \Phi\inter E_w$. We call $\Phi_w$ the \emph{minimal generating (root) subsystem} of $w$. We also define $d_w\coloneqq\rk_\Z\Gamma_w=\dim_\Q E_w$.
\end{dfn}

Being the intersection of $\Phi$ with a subspace $E_w$ of $E$, the set $\Phi_w$ is a root subsystem in $\Phi$, \emph{i.e.}\ it is a root system in the subspace of $E$ that it generates (see the Corollary to Proposition 4 of \cite[Chapter VI, \S1.1]{bourbakiLie456}). Note that this latter subspace may be strictly smaller than $E_w$; however we will see later in Corollary \ref{corMinGenSysReflections} that this never occurs. It therefore makes sense to consider the Weyl group $W(\Phi_w)$ of $\Phi_w$ as a subgroup of $W$, with the convention $W(\emptyset)=\{1\}$.

\begin{prop} \label{propSelfInWeylSubgroup}
Let $w\in W$. Then $w$ is in the Weyl group $W(\Phi_w)$ of its minimal generating subsystem.
\end{prop}

\begin{proof}
Any root system comes from some connected split reductive group $G$ over $\Q$ (see for instance the discussion in \cite[\S33.6]{humphreys}). Hence the statement follows from Corollary \ref{corSelfInWeylSubgroup} below.
\end{proof}

We equip $E$ with a non-degenerate symmetric bilinear form $(\,\cdot\mid\cdot\,)$ invariant by $W$ (see Proposition 3 of \cite[Chapter VI, \S1.1]{bourbakiLie456}). Then every root $\alpha\in\Phi$ satisfies $(\alpha\mid\alpha)\neq 0$, and $s_\alpha(x)=x-2\frac{(x\mid\alpha)}{(\alpha\mid\alpha)}\alpha$ for all $x\in E$. For any $x\in E$, we also note $x^\perp\coloneqq\ker(\,\cdot\mid x)$ the orthogonal subspace of $x$.

\begin{lem} \label{lemWhenInMinGenSys}
Let $w\in W$ and $x\in E$. Then $x\in E_w$ if and only if $\ker(w-\id_E)\subseteq x^\perp$.
\end{lem}

\begin{proof}
Suppose $x\in E_w$; take $y\in E$ such that $x=w(y)-y$. Then, for all $z\in\ker(w-\id)=\ker(w^{-1}-\id)$, we have
\[
	(z\mid x) = (z\mid w(y)-y) = (z\mid w(y))-(z\mid y)=(w^{-1}(z)-z\mid y)=0
\]
hence $z\in x^\perp$.

Conversely, suppose that $\ker(w-\id)\subseteq x^\perp$. Then the linear form $(\,\cdot\mid x)$ factors through a linear form $l$ on $E/\ker(w-\id)$. The map $w-\id$ also factors through an isomorphism $f\colon E/\ker(w-\id)\isoto\im(w-\id)$. The linear form $l\circ f^{-1}$ on $\im(w-\id)$ can be extended arbitrarily to a linear form $\widetilde{l}$ on $E$. Since $(\,\cdot\mid\cdot\,)$ is non-degenerate, there exists $y\in E$ such that $\widetilde{l}=(\,\cdot\mid y)$. Then, for all $z\in E$,
\[
	(z\mid x) = (w(z)-z\mid y) = (w(z)\mid y)-(z\mid y) = (z\mid w^{-1}(y)-y)
\]
thus, by non-degeneracy, $x=w^{-1}(y)-y=(w-\id)(-w^{-1}(y))\in E_w$.
\end{proof}

\begin{prop} \label{propMinGenSysIncrement}
Let $w\in W$ and $\alpha\in\Phi$. Let $w'$ be either $s_\alpha w$ or $ws_\alpha$.
\begin{enumerate}[label=\emph{(\arabic*)},ref=(\arabic*)]
\item \label{itemRootNotIn}
If $\alpha\notin\Phi_w$, then $E_{w'}=\Q\alpha\oplus E_w$. In particular, $d_{w'}=d_w+1$.
\item \label{itemRootIn}
If $\alpha\in\Phi_w$, then $E_w=\Q\alpha\oplus E_{w'}$. In particular, $d_{w'}=d_w-1$.
\end{enumerate}
\end{prop}

\begin{proof}
Assume that $w'=ws_\alpha$. Using the decomposition $E=\Q\alpha\oplus\alpha^\perp$, one can write
\begin{equation} \label{eqEw}
	E_w = \Q(w(\alpha)-\alpha) + (w-\id)\alpha^\perp
\end{equation}
and $E_w=(w-\id)\alpha^\perp$ if and only if $\ker(w-\id)\not\subseteq\alpha^\perp$. From Lemma \ref{lemWhenInMinGenSys}, this happens if and only if $\alpha\notin E_w$.
Similarly, because $w'(\alpha)=-w(\alpha)$ and $w'|_{\alpha^\perp}=w|_{\alpha^\perp}$,
\begin{equation} \label{eqEw'}
	E_{w'} = \Q(w(\alpha)+\alpha) + (w-\id)\alpha^\perp
\end{equation}
and $E_{w'}=(w-\id)\alpha^\perp$ if and only if $\alpha\notin E_{w'}$.

Suppose first that $\alpha\notin\Phi_w$, or equivalently $\alpha\notin E_w$. Then
\[
	w(\alpha)-\alpha\in E_w=(w-\id)\alpha^\perp\subseteq E_{w'} \,.
\]
Since $w(\alpha)+\alpha\in E_{w'}$, we get $\alpha=\frac{1}{2}(w(\alpha)+\alpha)-
\frac{1}{2}(w(\alpha)-\alpha)\in E_{w'}$. We also know from \eqref{eqEw'} that $E_w=(w-\id)\alpha^\perp\subseteq E_{w'}$ and that $\dim{E_{w'}}\leq\dim{E_w}+1$. Altogether, this yields $E_{w'}=\Q\alpha\oplus E_w$, so \ref{itemRootNotIn} is proved.

Now suppose that $\alpha\in\Phi_w$, or equivalently $\alpha\in E_w$. Take $x\in E$ such that $\alpha=w(x)-x$. Then
\[
	(w(x)+x\mid\alpha) = (w(x)+x\mid w(x)-x) = (w(x)\mid w(x))-(x\mid x) = 0
\]
so $w(x)+x\in\alpha^\perp$, hence $x=\frac{1}{2}(-\alpha+w(x)+x)\in-\frac{1}{2}\alpha+\alpha^\perp$. Therefore
\[
	2\alpha = 2(w(x)-x) \in \alpha-w(\alpha)+(w-\id)\alpha^\perp
\]
which can be rewritten as $w(\alpha)+\alpha\in(w-\id)\alpha^\perp$. From \eqref{eqEw'}, this means $E_{w'}=(w-\id)\alpha^\perp$, so $\alpha\notin E_{w'}$. As before, this also implies, using \eqref{eqEw}, that $E_{w'}\subseteq E_w$ and $\dim{E_w}\leq\dim{E_{w'}}+1$. We deduce that $E_w=\Q\alpha\oplus E_{w'}$, so \ref{itemRootIn} is proved.

If $w'=s_\alpha w$, then one can also write $w'=ww^{-1}s_\alpha w=w s_{w^{-1}(\alpha)}$. Since $\alpha\in\Phi_w$ if and only if $w^{-1}(\alpha)\in\Phi_w$, and similarly for $\Phi_w'$, the problem reduces to the previous case.
\end{proof}

Let $R\coloneqq\enstq{s_\alpha}{\alpha\in\Phi}$ be the set of reflections of $W$, which generates $W$ as a group. We can then view any element $w\in W$ as a word on the alphabet $R$ (in a non-unique way); we call \emph{length of $w$ on $R$} the minimal length of such a word, and \emph{reduced decomposition of $w$ on $R$} any such word of minimal length. The following corollary to Proposition \ref{propMinGenSysIncrement} interprets minimal generating subsystems using reduced decompositions on $R$.

\begin{cor} \label{corMinGenSysReflections}
Let $w\in W$.
\begin{enumerate}[label=\emph{(\arabic*)},ref=(\arabic*)]
\item \label{itemDwEqualsLength}
The number $d_w$ is the length of $w$ on the alphabet $R$.
\item \label{itemReducedFactors}
The subsystem $\Phi_w$ is the set of roots $\alpha\in\Phi$ such that $s_\alpha$ appears in a reduced decomposition of $w$ on the alphabet $R$.
\item \label{itemMinGenSysSpan}
The subsystem $\Phi_w$ spans $E_w$ as a vector space.
\end{enumerate}
\end{cor}

\begin{proof}
If $w=1$, then all claims are true because $E_w=\im(\id_E-\id_E)=\{0\}$. We now suppose that $w\neq 1$.

Let $w=s_{\alpha_r}\ldots s_{\alpha_1}$ be a decomposition of $w$ on $R$, with $\alpha_i\in\Phi$ for all $1\leq i\leq r$. By Proposition \ref{propMinGenSysIncrement}, we know that $d_{s_{\alpha_i}\ldots s_{\alpha_1}}\leq d_{s_{\alpha_{i-1}}\ldots s_{\alpha_1}}+1$ for all $1\leq i\leq d$, with equality if and only if $E_{s_{\alpha_i}\ldots s_{\alpha_1}}=\Q\alpha_i\oplus E_{s_{\alpha_{i-1}}\ldots s_{\alpha_1}}$. An induction on $i$ gives $d_w=d_{s_{\alpha_r}\ldots s_{\alpha_1}}\leq r$, with equality if and only if
\begin{equation} \label{eqReducedFactorsSpan}
	E_w = \Q\alpha_r\oplus\ldots\oplus\Q\alpha_1 \,.
\end{equation}
Therefore the three claims will be proved if we can show that $\Phi_w$ is non-empty and that for any $\alpha\in\Phi_w$, there is a decomposition of $w$ on $R$ of length $d_w$ containing $s_\alpha$.

The non-emptiness of $\Phi_w$ is a consequence of Proposition \ref{propSelfInWeylSubgroup}. Now take $\alpha\in\Phi_w$. The wanted decomposition can be obtained by induction on $d_w$: Proposition \ref{propMinGenSysIncrement} implies that $d_{s_\alpha w}=d_w-1$, and if $s_\alpha w=s_{\alpha_r}\ldots s_{\alpha_1}$ is a decomposition of length $r\coloneqq d_{s_\alpha w}$, then $w=s_\alpha s_{\alpha_r}\ldots s_{\alpha_1}$ is a decomposition of length $d_w$.
\end{proof}

The following corollary explains the terminology of minimal generating subsystem introduced in Definition \ref{dfnMinGenSys}.

\begin{cor} \label{corMinGenSysEtymology}
Let $w\in W$ and let $E'$ be a subspace of $E$. Let $\Phi'\coloneqq\Phi\inter E'$, which is a root system in $E'$ with Weyl group $W(\Phi')\subseteq W$. Then $w\in W(\Phi')$ if and only if $\Phi_w\subseteq\Phi'$, if and only if $E_w\subseteq E'$.
\end{cor}

\begin{proof}
The implication $E_w\subseteq E'\implies\Phi_w\subseteq\Phi'$ is a mere consequence of the definitions, and the implication $\Phi_w\subseteq\Phi'\implies w\in W(\Phi')$ is Proposition \ref{propSelfInWeylSubgroup}.

Now suppose that $w\in W(\Phi')$. Let $\Phi'_w$ be the minimal generating subsystem of $w$ in $\Phi'$ and let $E'_w\coloneqq\im(w|_{E'}-\id_{E'})\subseteq E'\inter E_w$. From Corollary \ref{corMinGenSysReflections}\ref{itemDwEqualsLength} applied to $\Phi'$, $w$ has a reduced decomposition on $R'\coloneqq\enstq{s_\alpha}{\alpha\in\Phi'}$ of length $\dim{E'_w}$. Since $R'\subseteq R$, the same result applied to $\Phi$ shows that $\dim{E'_w}\geq\dim{E_w}$. Therefore, the chain of inclusions $E'_w\subseteq E'\inter E_w\subseteq E_w$ is a chain of equalities. In particular, $E'\inter E_w=E_w$; or equivalently $E_w\subseteq E'$.
\end{proof}

Fix $I\subset\Phi$ a basis of simple roots of $\Phi$, and write $S\coloneqq\enstq{s_\alpha}{\alpha\in I}\subseteq W$ for the set of simple reflections corresponding to the simple roots. These define a length function $\lg$ on $W$. The following result is a generalisation of \cite[Lemma 2.7]{bhs2} where it was proved for root systems $\Phi$ of type $A$.

\begin{cor} \label{corDwLeqLg}
Let $w\in W$. Then
\[
	d_w \leq \lg(w)
\]
with equality if and only if $w$ is a product of distinct simple reflections.
\end{cor}

\begin{proof}
Let $w=s_{\alpha_{\lg(w)}}\ldots s_{\alpha_1}$ be a reduced decomposition of $w$ on $S$, with $\alpha_i\in I$ for all $1\leq i\leq\lg(w)$. Corollary \ref{corMinGenSysReflections}\ref{itemDwEqualsLength} and its proof give $d_w\leq\lg(w)$, with equality if and only if $E_w=\Q\alpha_{\lg(w)}\oplus\ldots\oplus\Q\alpha_1$ (see \eqref{eqReducedFactorsSpan}). Therefore, in the equality case, the $\alpha_i$ must be distinct.

Conversely, suppose that $w$ is a product of distinct simple reflections. Write $w=s_{\alpha_r}\ldots s_{\alpha_1}$ with $\alpha_i$ for $1\leq i\leq r$ being distinct elements of $I$; then $r\geq\lg(w)$ by definition. Also, since the elements of $I$ are linearly independent, the sum $\Q\alpha_r+\ldots+\Q\alpha_1$ is direct. The proof of Corollary \ref{corMinGenSysReflections} implies that $r=d_w$; hence $d_w\geq\lg(w)$, so we have equality.
\end{proof}

\subsection{Interpretation for a split reductive group} \label{ssecMinGenSysGroups}

For $\Phi$ a root system and $(G,T)$ a connected split reductive group giving rise to $\Phi$, we give an interpretation of the minimal generating subsystem $\Phi_w\subseteq\Phi$ for $w\in W(\Phi)$ in terms of Levi groups. We use the notation of the introduction.

We start by recalling a classical lemma.

\begin{lem}
Let $N\geq 1$ be an integer and $M$ be a $\Z[\frac{1}{N}]$-module equipped with a linear action of the group $\Z/N\Z$. Then $H^1(\Z/N\Z,M)=0$.
\end{lem}

\begin{proof} \label{lemGroupCohomology}
The $\Z[\frac{1}{N}]$-module $H^1(\Z/N\Z,M)$ is killed by $N$ (see \cite[\S VII.7, Proposition 6]{serre}), hence it must be $0$.
\end{proof}

\begin{prop} \label{propCycleDecompositionCentraliser}
For $w\in W$, the centraliser $\left(C_G((T^w)^\circ),T\right)$ is a connected split reductive group whose root system is the minimal generating subsystem $\Phi_w$ of $w$.
\end{prop}

\begin{proof}
$(T^w)^\circ$ is a subtorus of $T$ (see the discussion in \cite[Notation 12.29]{milne}; note that $T^w$ is always reduced in characteristic $0$), hence $C_G((T^w)^\circ)$ is connected and split reductive by \cite[Corollary 17.59]{milne}. Note that it always contains $T$, so $\left(C_G((T^w)^\circ),T\right)$ is split reductive.

Let $\alpha\in \Phi$ be a root of $(G,T)$. The eigenspace $\gfrak_\alpha$ has dimension $1$, so $\alpha$ is a root of $\left(C_G((T^w)^\circ),T\right)$ if and only if $\gfrak_\alpha$ is in the Lie group of $C_G((T^w)^\circ)$. By \cite[Thm.\ 21.11(c)]{milne}, we thus have
\[
	\Phi'\coloneqq\Phi\left(C_G((T^w)^\circ),T\right) = \enstq{\alpha\in\Phi}{\alpha|_{(T^w)^\circ}=1} \,.
\]

Let $N$ be the order of $w$ in $W$. The morphism of algebraic varieties $T\to T,\, t\longmapsto\prod_{k=0}^{N-1}\Ad(w^{-k})t$ has a connected image (because $T$ is connected) included in $T^w$. Hence, $\prod_{k=0}^{N-1}\Ad(w^{-k})t\in(T^w)^\circ$ for all $t\in T$. Then, for $\alpha\in \Phi'$ and $t\in T$, we have
\[
	\left(\sum_{k=0}^{N-1}w^k\alpha\right)(t)=\alpha\left(\prod_{k=0}^{N-1}\Ad(w^{-k})t\right)=1
\]
hence $(\sum_{k=0}^{N-1}w^k)\alpha=0$. Conversely, if $\alpha\in\ker(\sum_{k=0}^{N-1}w^k)$, then for all $t\in(T^w)^\circ$,
\[
	\alpha(t^N)=\alpha\left(\prod_{k=0}^{N-1}\Ad(w^{-k})t\right)=1
\]
so, since the group of characters on the torus $(T^w)^\circ$ has no torsion, we necessarily have $\alpha(t)=1$. Hence $\alpha|_{(T^w)^\circ}=1$. Therefore $\Phi'=\Phi\inter\ker(\sum_{k=0}^{N-1}w^k)$. Since $H^1(\langle w\rangle,E)=0$ by Lemma \ref{lemGroupCohomology}, we have $\Phi'=\Phi\inter\im(w-\id_E)=\Phi_w$.
\end{proof}

\begin{cor} \label{corSelfInWeylSubgroup}
Let $w\in W$. Then $w\in W(\Phi_w)$.
\end{cor}

\begin{proof}
Seeing $w$ as an element of $N_G(T)/T$, we have $w\in C_G((T^w)^\circ)/T$, so that actually $w\in N_{C_G((T^w)^\circ)}(T)/T=W(C_G((T^w)^\circ),T)=W(\Phi_w)$ where the last equality comes from Proposition \ref{propCycleDecompositionCentraliser}.
\end{proof}

\subsection{Pairs of elements in the Weyl group} \label{ssecBonnePaireDef}

Fix a basis of simple roots $I\subset\Phi$ and write $\Phi_+$ for the subset of positive roots that it generates. Write $S\coloneqq\enstq{s_\alpha}{\alpha\in I}\subseteq W$ for the set of simple reflections corresponding to the simple roots. They define a length function $\lg$ and a Bruhat order $\preceq$ on $W$.

\begin{dfn} \label{dfnGoodPair}
Let $(w_1,w_2)\in W^2$ such that $w_1\preceq w_2$. Let $\Phi_{w_1w_2^{-1}}\subseteq\Phi$ be the minimal generating subsystem of $w_1w_2^{-1}$. We say that $(w_1,w_2)$ is a \emph{good pair} if there exists a sequence $(\alpha_1,\ldots,\alpha_r)$ of roots in $\Phi_{w_1w_2^{-1}}$ such that $w_2=s_{\alpha_r}s_{\alpha_{r-1}}\ldots s_{\alpha_1}w_1$ and that for each $1\leq i\leq r$, the relation $s_{\alpha_i}\ldots s_{\alpha_1}w_1\succ s_{\alpha_{i-1}}\ldots s_{\alpha_1}w_1$ holds. If such a sequence does not exist, we say that $(w_1,w_2)$ is a \emph{bad pair}.
\end{dfn}

\begin{rem}
The set of reflections $R=\enstq{s_\alpha}{\alpha\in\Phi}$ of $W$ can be rewritten as $R=\bigcup_{w\in W}wSw^{-1}$. On the other hand, using \ref{itemReducedFactors} of Corollary \ref{corMinGenSysReflections}, one sees that the subset $\enstq{s_\alpha}{\alpha\in\Phi_{w_1w_2^{-1}}}$ of $R$ is the set of reflections that appear in some reduced decomposition of $w$ on the alphabet $R$. Hence the notion of good pair is intrinsic to the Coxeter system $(W,S)$.
\end{rem}

Let $w_0\in W$ be the maximal element for the Bruhat order. Using that $w\mapsto ww_0$ and $w\mapsto w_0w$ are antiautomorphisms of $W$ for the Bruhat order, one easily proves the following result.

\begin{prop} \label{propGoodPairsInverted}
Let $(w_1,w_2)\in W^2$. Then $(w_1,w_2)$ is a good pair (resp.\ bad pair) if and only if $(w_2w_0,w_1w_0)$ is a good pair (resp.\ bad pair). Similarly, $(w_1,w_2)$ is a good pair (resp.\ bad pair) if and only if $(w_0w_2,w_0w_1)$ is a good pair (resp.\ bad pair).
\end{prop}

\begin{prop} \label{propEqualityImpliesGoodPair}
Let $w_1,w_2\in W$ with $w_1\preceq w_2$. Then
\[
	d_{w_2w_1^{-1}} \leq \lg(w_2)-\lg(w_1)
\]
and if there is equality, then the pair $(w_1,w_2)$ is good.
\end{prop}

\begin{proof}
By the chain property in $W$ (see \cite[Thm.\ 2.2.6]{coxeter}), there is a sequence of size $r=\lg(w_2)-\lg(w_1)$ of roots $(\alpha_1,\ldots,\alpha_r)$ in $\Phi$ with $w_2=s_{\alpha_r}s_{\alpha_{r-1}}\ldots s_{\alpha_1}w_1$ and such that for each $1\leq i\leq r$, we have $\lg(s_{\alpha_i}\ldots s_{\alpha_1}w_1)=\lg(s_{\alpha_{i-1}}\ldots s_{\alpha_1}w_1)+1$ (hence automatically $s_{\alpha_i}\ldots s_{\alpha_1}w_1\succ s_{\alpha_{i-1}}\ldots s_{\alpha_1}w_1$). Then $w_2w_1^{-1}=s_{\alpha_r}s_{\alpha_{r-1}}\ldots s_{\alpha_1}$ is a decomposition of $w_2w_1^{-1}$ on $R\coloneqq\enstq{s_\alpha}{\alpha\in\Phi}$. Corollary \ref{corMinGenSysReflections}\ref{itemDwEqualsLength} then yields $d_{w_2w_1^{-1}} \leq r = \lg(w_2)-\lg(w_1)$. If there is equality, Corollary \ref{corMinGenSysReflections}\ref{itemReducedFactors} gives $\alpha_i\in\Phi_{w_2w_1^{-1}}$ for all $1\leq i\leq r$, which implies that $(w_1,w_2)$ is a good pair.
\end{proof}

Let $J\subseteq I$ be a set of simple roots. We write $S_J\coloneqq\enstq{s_\alpha}{\alpha\in J}$ for the set of simple reflections associated to the elements of $J$, and $W_J\coloneqq\engendre{s}{s\in S_J}$ for the subgroup of $W$ it spans. Note that $W_J=W(\Phi_J)$, where $\Phi_J$ is the root subsystem of $\Phi$ generated by $J$, and the Bruhat order on $W_J$ is the same as the restriction of the Bruhat order on $W$. We also write $W^J\coloneqq\enstq{w\in W}{ws\succ w\mathrel{}\forall s\in S_J}$. It is known that $W^J$ is the set of representatives of minimal length for the left cosets representatives $W/W_J$, and that for all $w\in W$, the decomposition $w=w^Jw_J$ with $w^J\in W^J$ and $w_J\in W_J$ satisfies $\lg(w)=\lg(w^J)+\lg(w_J)$ (see \cite[\S2.4]{coxeter}).

\begin{lem} \label{lemPiecesToWhole}
Let $J\subseteq I$ be a subset of the simple roots, let $u^J$ and $v^J$ be elements of $W^J$, let $w_J$ be an element of $W_J$ and let $s\in W_J$ be a reflection, \emph{i.e.}\ an element of $\bigcup_{w\in W_J}wS_Jw^{-1}$. Then $sw_J\succ w_J$ if and only if $u^Jsw_J(v^J)^{-1}\succ u^Jw_J(v^J)^{-1}$.
\end{lem}

\begin{proof}
Let $R\coloneqq\enstq{wSw^{-1}}{w\in W}$ be the set of reflections of $W$ and consider the map
\[
	\appl{\eta}{W\times R}{\{-1,1\}}{(w,t)}{
	\begin{cases}
	1 & \mathrm{if}\; tw\succ w \\
	-1 & \mathrm{if}\; tw \prec w
	\end{cases}
	}
\]
(see Theorem 1.4.3 of \cite{coxeter} and its proof). Our goal is then to prove that
\[
	\eta(u^Jw_J(v^J)^{-1},u^Js(u^J)^{-1})=\eta(w_J,s) .
\]

For any $w\in W$, $t\in R$ and $\varepsilon\in\{-1,1\}$, define
\[
	\pi_w(t,\varepsilon) \coloneqq (wtw^{-1},\varepsilon\eta(w^{-1},t)) \in R\times\{-1,1\} \,.
\]
This defines a map $\pi$ from $W$ to the finite group of bijections from $R \times\{-1,1\}$ to itself, and $\pi$ is a group morphism (\emph{loc.\,cit.}, Theorem 1.3.2(i)). Therefore, for any $v,w\in W$ and $t\in R$, we can write $\pi_{w^{-1}v^{-1}}(t,1)=\pi_{w^{-1}}(\pi_{v^{-1}}(t,1))$ and get the formula
\[
	\eta(vw,t) = \eta(v,t)\eta(w,v^{-1}tv) \,.
\]
First apply this formula to $v=u^J$, $w=w_J(v^J)^{-1}$ and $t=u^Js(u^J)^{-1}$:
\begin{equation} \label{eqEta}
	\eta(u^Jw_J(v^J)^{-1},u^Js(u^J)^{-1}) = \eta(u^J,u^Js(u^J)^{-1})\eta(w_J(v^J)^{-1},s) \,.
\end{equation}
Now, since $u^J\in W^J$ and $w_J^{-1},s\in W_J$ and using the fact that $w\mapsto w^{-1}$ preserves the length on $W$, we have
\begin{eqnarray*}
\lg(sw_J(v^J)^{-1})&=&\lg(v^Jw_J^{-1}s)=\lg(v^J)+\lg(w_J^{-1}s)=\lg(v^J)+\lg(sw_J)\\
\lg(w_J(v^J)^{-1})&=&\lg(v^Jw_J^{-1})=\lg(v^J)+\lg(w_J^{-1})=\lg(v^J)+\lg(w_J)
\end{eqnarray*}
from which we deduce $\eta(w_J(v^J)^{-1},s)=\eta(w_J,s)$. Similarly, $u^J\in W^J$ and $s\in W_J$, therefore
\[
	\lg(u^Js(u^J)^{-1}u^J)=\lg(u^Js)=\lg(u^J)+\lg(s)>\lg(u^J) \,,
\]
which means that $\eta(u^J,u^Js(u^J)^{-1})=1$. By \eqref{eqEta}, this completes the proof.
\end{proof}

\begin{lem} \label{lemConjugateToStandard}
Let $w\in W$. There is $u\in W$ and $J\subseteq I$ such that $\Phi_w=u(\Phi_J)$.
\end{lem}

\begin{proof}
Given any fixed basis $J'$ of $\Phi_w$, there is a basis $I'$ of $\Phi$ containing $J'$ (see Proposition 24 in \S1.8 of \cite[Chapter VI]{bourbakiLie456}). The sets $I$ and $I'$ are in the same $W$-orbit (see Remark 4 in \S1.5 of \emph{loc. cit.}), so we get $u\in W$ such that $I'=u(I)$ and $J\coloneqq u^{-1}(J')$ is the subset of $I$ we are looking for.
\end{proof}

The following proposition provides a direct way of seeing if a given pair is good.

\begin{prop} \label{propGoodPairsMinipermutations}
Let $(w_1,w_2)\in W^2$. Let $\Phi_{w_1w_2^{-1}}\subseteq\Phi$ be the minimal generating subsystem of $w_1w_2^{-1}$ and $J\subseteq I$ be such that $\Phi_{w_1w_2^{-1}}$ is in the $W$-orbit of the root subsystem $\Phi_J\subseteq\Phi$ (see Lemma \ref{lemConjugateToStandard}). Then we can write
\[
	w_1=u^Jw_{J,1}(v^J)^{-1} \,, \quad w_2=u^Jw_{J,2}(v^J)^{-1}
\]
for some $u^J,v^J\in W^J$ and $w_{J,1},w_{J,2}\in W_J$. Such a decomposition of $w_1$ and $w_2$, as well as the choice of $J$, may not be unique. Furthermore, for one such (or equivalently any such) decomposition of $w_1$ and $w_2$, we have $w_{J,1}\preceq w_{J,2}$ if and only if the pair $(w_1,w_2)$ is good.
\end{prop}

\begin{proof}
Take $u\in W$ such that $\Phi_{w_1w_2^{-1}}=u(\Phi_J)$. Let $u^J\in W^J$ be the representative of $uW_J$, then $\Phi_{w_1w_2^{-1}}=u^J(\Phi_J)$. We know from Proposition \ref{propSelfInWeylSubgroup} that $w_1w_2^{-1}\in W(\Phi_{w_1w_2^{-1}})=u^JW_J(u^J)^{-1}$. Therefore the cosets $w_1^{-1}u^JW_J$ and $w_2^{-1}u^JW_J$ are the same. Let $v^J\in W^J$ be the representative for that coset; we then have elements $w_{J,1}$ and $w_{J,2}$ of $W_J$ such that $w_1^{-1}u^J=v^Jw_{J,1}^{-1}$ and $w_2^{-1}u^J=v^Jw_{J,2}^{-1}$, \emph{i.e.\ }$w_1=u^Jw_{J,1}(v^J)^{-1}$, $w_2=u^Jw_{J,2}(v^J)^{-1}$.

Now assume that $w_{J,1}\preceq w_{J,2}$. Then, by the chain property in $W_J$ (see \cite[Thm.\ 2.2.6]{coxeter}) there is a sequence $(\alpha_1,\ldots,\alpha_r)$ of roots in $\Phi_J$ with $w_{J,2}=s_{\alpha_r}s_{\alpha_{r-1}}\ldots s_{\alpha_1}w_{J,1}$ such that $s_{\alpha_i}\ldots s_{\alpha_1}w_1\succ s_{\alpha_{i-1}}\ldots s_{\alpha_1}w_{J,1}$ for each $1\leq i\leq r$. Let $\beta_i=u^J(\alpha_i)\in u^J(\Phi_J)=\Phi_{w_1w_2^{-1}}$ for each $i$, then $s_{\beta_i}=u^Js_{\alpha_i}(u^J)^{-1}$ and we have
\[
	w_2=u^Jw_{J,2}(v^J)^{-1}=s_{\beta_r}\ldots s_{\beta_1}u^Jw_{J,1}(v^J)^{-1}=s_{\beta_r}\ldots s_{\beta_1}w_1 \,,
\]
and for each $i$ the relation by Lemma \ref{lemPiecesToWhole}:
\[
	s_{\beta_i}\ldots s_{\beta_1}w_1=u^Js_{\alpha_i}\ldots s_{\alpha_1}w_{J,1}(v^J)^{-1}\succ u^Js_{\alpha_{i-1}}\ldots s_{\alpha_1}w_{J,1}(v^J)^{-1}=s_{\beta_{i-1}}\ldots s_{\beta_1}w_1.
\]
Since this chain implies in particular that $w_1\preceq w_2$, we have shown all the conditions for the pair $(w_1,w_2)$ to be good. Conversely, going through all the steps backwards and using Lemma \ref{lemPiecesToWhole}, we see that if $(w_1,w_2)$ is a good pair then $w_{J,1}\preceq w_{J,2}$.
\end{proof}

\subsection{Systems of type $A$ and pattern avoidance} \label{ssecPatternAvoidance}

We provide a concrete description of minimal generating subsystems (Proposition \ref{propCycleDecompositionTypeA}) and of good pairs (Proposition \ref{propGoodPairsGLn} and Corollary \ref{corGoodPairsFlattening}) for root systems of type $A$.

Let \ $n>1$ \ an \ integer, \ $(e_i)_{1\leq i\leq n}$ \ the \ standard \ basis \ of \ $\Q^n$ \ and \ $\Phi=\enstq{e_i-e_j}{1\leq i,j\leq n\;,\;i\neq j}\subset \Q^n$ a root system of type $A_{n-1}$. The Weyl group of $\Phi$ is the symmetric group $\Scal_n$ which acts on $\Q^n$, and on $\Phi$, by $w(e_i)=e_{w(i)}$ for $w\in\Scal_n$ and $1\leq i\leq n$. We write any $w\in\Scal_n$ as $[w(1),\ldots,w(n)]$. Given a root $\alpha=e_i-e_j$ with $i\neq j$, the associated reflection $s_\alpha=(i,j)$ is given by $s_\alpha(i)=j$, $s_\alpha(j)=i$, $s_\alpha(k)=k$ $\forall \ k\neq i,j$. We set the standard basis $I\coloneqq\enstq{e_i-e_{i+1}}{1\leq i\leq n-1}$. It generates the set $\Phi_+\coloneqq\enstq{e_i-e_j}{1\leq i<j\leq n}$ of positive roots, and the set of simple reflections is $S=\enstq{[1\ldots(i-1)(i+1)i(i+2)\ldots n]}{1\leq i\leq n-1}$.

\begin{prop} \label{propCycleDecompositionTypeA}
Let $w\in\Scal_n$. Write $\coprod_{1\leq k\leq r}\Omega_k$ for the partition of $\{1,\ldots,n\}$ into orbits under the action of $w$. For $1\leq k\leq r$, let
\[
	\appl{f_k}{\Q^n}{\Q}{(x_i)_{1\leq i\leq n}}{\sum_{i\in\Omega_k}x_i} \,.
\]
Then $E_w=\bigcap_{1\leq k\leq r}\ker(f_k)$ and $\Phi_w=\bigcup_{1\leq k\leq r}\enstq{e_i-e_j}{i,j\in\Omega_k\;,\;i\neq j}$.
\end{prop}

\begin{proof}
The inclusion $E_w\subseteq\bigcap_{1\leq k\leq r}\ker(f_k)$ is straightforward. Conversely, Proposition \ref{propCycleDecompositionCentraliser} gives $\bigcup_{1\leq k\leq r}\enstq{e_i-e_j}{i,j\in\Omega_k\;,\;i\neq j}=\Phi(C_{\GL_n}((T^w)^\circ))=\Phi_w$, hence $\bigcap_{1\leq k\leq r}\ker(f_k)\subseteq E_w$ by taking the generated $\Q$-vector space.
\end{proof}

\begin{dfn} \label{dfnPartialOrder}
Let $d$ be an integer such that $1\leq d\leq n$. We put a partial order on the set of size $d$ subsets of $\{1,\ldots,n\}$ as follows: for any such subsets $A$ and $A'$, we say that $A\preceq A'$ if and only if $\abs{A\inter\{1,\ldots,m\}}\geq\abs{A'\inter\{1,\ldots,m\}}$ for all $1\leq m\leq n$. Equivalently, $A\preceq A'$ if and only if $\abs{A\inter\{m,\ldots,n\}}\leq\abs{A'\inter\{m,\ldots,n\}}$ for all $1\leq m\leq n$.
\end{dfn}

The relation $A\preceq A'$ can be seen concretely as follows: if $A=\{a_1<\ldots<a_d\}$ and $A'=\{a'_1<\ldots<a'_d\}$, then $A\preceq A'$ if and only if $a_k\leq a'_k$ for all $k\in\{1,\ldots,d\}$.

\begin{dfn} \label{dfnBox}
Let $w\in\Scal_n$. For $1\leq i,j\leq n$, we write
\[
	w[i,j] \coloneqq \Abs{ w(\{1,\ldots,i\})\inter\{j,\ldots,n\} } \,.
\]
For a subset $\Sigma$ of $\{1,\ldots,n\}$, we write
\[
	w[i,j]_\Sigma \coloneqq \Abs{ w(\{1,\ldots,i\})\inter\{j,\ldots,n\}\inter\Sigma } \,.
\]
\end{dfn}

\begin{ex} \label{exBruhatLineNotation}
For any $w,w'\in\Scal_n$, we have $w'\preceq w$ if and only if $w'[i,j]\leq w[i,j]$ for all $1\leq i,j\leq n$ (see for example \cite[Thm.\ 2.1.5]{coxeter}). Equivalently, $w'\preceq w$ if and only if $\{w(1),\ldots,w(d)\}\preceq\{w'(1),\ldots,w'(d)\}$ for all $1\leq d\leq n$.
\end{ex}

The following proposition is an application of Proposition \ref{propGoodPairsMinipermutations} to the type $A_{n-1}$.

\begin{prop} \label{propGoodPairsGLn}
Let $w_1$, $w_2$ be elements of $\Scal_n$. Then the pair $(w_1,w_2)$ is good if and only if, for any orbit $\Omega$ of $w_1w_2^{-1}$ under the action of $\Scal_n$ on $\{1,\ldots,n\}$ and any $1\leq d\leq n$, we have $\Omega\inter w_1(\{1,\ldots,d\})\succeq \Omega\inter w_2(\{1,\ldots,d\})$ (as subsets of $\{1,\ldots,n\}$ in the sense of Definition \ref{dfnPartialOrder}). Equivalently, the pair $(w_1,w_2)$ is good if and only if, for any orbit $\Omega$ of $w_1w_2^{-1}$ and any $1\leq i,j\leq n$, we have $w_1[i,j]_\Omega\leq w_2[i,j]_\Omega$.
\end{prop}

\begin{proof}
This is a consequence of Proposition \ref{propGoodPairsMinipermutations} and Example \ref{exBruhatLineNotation}.
\end{proof}

This result gives a way to visually recognise good and bad pairs through flattenings, which we now define.

\begin{dfn} \label{dfnFlattening}
Let $m,n$ be integers with $0<m\leq n$, and let $\Sigma=\{i_1<\ldots<i_m\}$ be a subset of $\{1,\ldots,n\}$. We define the \emph{flattening map} $\mathrm{fl}_\Sigma\colon\Scal_n\to\Scal_m$ as follows. For any $w\in\Scal_n$, we take $\mathrm{fl}_\Sigma(w)$ as the unique $f\in\Scal_m$ such that $(w(i_1).\ldots,w(i_m))$ is in the same relative order as $(f(1),\ldots,f(m))$. This means that $w\in\Scal_n$ flattens to $f\in\Scal_m$ through $\Sigma$ if and only if for any $1\leq k<l\leq m$, we have $w(i_k)<w(i_l)$ if and only if $f(k)<f(l)$.
\end{dfn}

\begin{cor} \label{corGoodPairsFlattening}
Let $w_1,w_2\in\Scal_n$. Then the pair $(w_1,w_2)$ is good if and only if for any orbit $\Omega$ of $w_1^{-1}w_2$, we have $\mathrm{fl}_\Omega(w_1)\preceq\mathrm{fl}_\Omega(w_2)$.
\end{cor}

\begin{proof}
This is a consequence of Proposition \ref{propGoodPairsGLn} and Example \ref{exBruhatLineNotation}).
\end{proof}

We now give a criterion that relates good pairs to pattern avoidance for root systems of type $A$.

\begin{dfn} \label{dfnPattern}
Let $m,n\leq 1$ be integers. For $w\in\Scal_n$ and $f\in\Scal_m$, we say that \emph{w has the pattern f} if and only if $w$ flattens to $f$ (through some $\Sigma\subseteq\{1,\ldots,n\}$ of size $m$), \emph{i.e.}\ there exists a flattening map $\mathrm{fl}_\Sigma\colon\Scal_n\to\Scal_m$ (see Definition \ref{dfnFlattening}) such that $f=\mathrm{fl}_\Sigma(w)$.
\end{dfn}

\begin{ex}
If $f=[3412]\in\Scal_4$ (resp.\ $f=[4231]$), the permutation $w\in\Scal_n$ has the pattern $f$ if and only if there exist $1\leq a<b<c<d\leq n$ such that $w(c)<w(d)<w(a)<w(b)$ (resp.\ $w(d)<w(b)<w(c)<w(a)$).
\end{ex}

\begin{thm} \label{thmBadPairPatterns}
Let $w\in\Scal_n$.
\begin{enumerate}[label=\emph{(\arabic*)},ref=(\arabic*)]
\item
There exists $w'\in\Scal_n$ such that $(w',w)$ is a bad pair if and only if $w$ has at least one of the four patterns $[4231]$, $[42513]$, $[35142]$ or $[351624]$.
\item
There exists $w''\in\Scal_n$ such that $(w,w'')$ is a bad pair if and only if $w$ has at least one of the four patterns $[1324]$, $[24153]$, $[31524]$ or $[426153]$.
\end{enumerate}
\end{thm}

\begin{proof}
The second statement is equivalent to the first by applying Proposition \ref{propGoodPairsInverted}. Also, if $w$ has one of the four patterns $[4231]$, $[42513]$, $[35142]$ and $[351624]$, then we can easily construct a $w'$ such that $(w',w)$ is a bad pair, using Corollary \ref{corGoodPairsFlattening} and the fact that $([1324],[4231]), ([13245],[42513]), ([12435],[35142]), ([124356],[351624])$ are bad pairs. We thus only have to prove that for any bad pair $(w',w)$, the permutation $w$ has one of the four patterns.

Let $(w',w)$ be such a bad pair. From Corollary \ref{corGoodPairsFlattening}, there exist an orbit $\Omega$ of $ww'^{-1}$ and $(i,j)\in\{1,\ldots,n\}^2$ such that $w[i,j]_\Omega<w'[i,j]_\Omega$; we fix one such $\Omega$ for the rest of the proof, and set $\Sigma=\{1,\ldots,n\}\setminus\Omega$. The strategy will be to find the patterns on the plot graph of $w_1$ using this fact. The different elements of the proof are represented in Figure \ref{figPatternsProof}.

The first step is to divide the graph into a few regions. First consider $i_1$ minimal with the property that there exists some $j$ such that $w[i_1,j]_\Omega<w'[i_1,j]_\Omega$. Then take $j_1$ maximal such that $w[i_1,j_1]_\Omega<w'[i_1,j_1]_\Omega$, and $j_2$ minimal such that $w[i_1,j_2]_\Omega<w'[i_1,j_2]_\Omega$. Finally take $i_2$ maximal such that $w[i_2,j_2]_\Omega<w'[i_2,j_2]_\Omega$.

The second step is to identify up to six points of interest on the graph of $w$. There exists $A=(i_A,j_A)$ such that $i_A\leq i_1$, $j_A\geq j_1$ and $w(i_A)=j_A\in\Sigma$. Indeed, $w[i_1,j_1]=w[i_1,j_1]_\Omega+w[i_1,j_1]_\Sigma$ and similarly with $w'$, hence $w[i_1,j_1]_\Sigma>w'[i_1,j_1]_\Sigma$ and in particular $w[i_1,j_1]_\Sigma\geq 1$. Similarly, there exists $F=(i_F,j_F)$ such that $i_F>i_2$, $j_F<j_2$ and $w(i_F)=j_F\in\Sigma$. Indeed, $w[i_2,j_2]_\Sigma>w'[i_2,j_2]_\Sigma$, hence there are strictly more (resp.\ strictly fewer) points $(i,j)\in\{1,\ldots,i_2\}\times\{1,\ldots,j_2-1\}$ (resp.\ $\in\{i_2+1,\ldots,n\}\times\{1,\ldots,j_2-1\}$) with $j=w(i)$ than with $j=w'(i)$. Also define $B=(i_B,j_B)$, $C=(i_C,j_C)$, $D=(i_D,j_D)$, $E=(i_E,j_E)$ respectively by $i_B=w^{-1}(j_2-1)$ and $j_B=j_2-1$; $i_C=i_1$ and $j_C=w(i_1)$; $i_D=i_2+1$ and $j_D=w(i_2+1)$; $i_E=w^{-1}(j_1)$ and $j_E=j_1$. Note that we indeed have $A,B,C,D,E,F\in\enstq{(i,j)}{w(i)=j}$.

By definition of $i_1,j_1,i_2,j_2$ have $w(i_1),j_1,w(i_2+1),j_2-1\in\Omega$: for instance, $w(i_1)\notin\Omega$ implies $w[i_1-1,j_1]_\Omega=w[i_1,j_1]_\Omega$ and $w'[i_1-1,j_1]_\Omega=w'[i_1,j_1]_\Omega$, which contradicts the minimality of $i_1$. We then deduce $i_B\leq i_1$, $j_C<j_2$, $j_D\geq j_2$, $i_E>i_1$ by definition of $j_2$, $j_1$, $i_2$ and $j_1$ respectively: for instance, $i_B>i_1$ implies $w[i_1,j_2-1]_\Omega=w[i_1,j_2]_\Omega<w'[i_1,j_2]_\Omega\leq w'[i_1,j_2-1]_\Omega$ which contradicts the minimality of $j_2$. Also, $i_A<i_1$ and $j_A>j_1$ since we already know $i_A\leq i_1$ and $j_A\geq j_1$ and $A\in w^{-1}(\Sigma)\times\Sigma$ while $(i,j)\in w^{-1}(\Omega)\times\Omega$. Similarly, $i_F>i_2+1$ and $j_F>j_2+1$.

Note that the sets of points $\{A\}$, $\{B,C\}$, $\{D,E\}$, $\{F\}$ lie in distinct quadrants delimited by $i_1$ and $j_2$. Hence $B=C$ and $D=E$ are the only equalities that may happen between those six points. Also note that since those points are all on the graph of $w$, unless two of those are equal, both of their coordinates are distinct.

\begin{figure}[h]
	\centering
	\begin{tikzpicture}[scale=0.3]
		\draw (9,0) node {$i_1$} ;
		\draw (14,0) node {$i_2$} ;
		\draw (0,10) node {$j_2$} ;
		\draw (0,16) node {$j_1$} ;
		\clip (0.6,0.6) rectangle (24.4,24.4) ;
		\draw[very thin, gray] (0,0) grid (25,25) ;
		\filldraw[gray, opacity=0.3] (8.6,0) rectangle (9.4,25) ;
		\filldraw[gray, opacity=0.3] (14.6,0) rectangle (15.4,25) ;
		\filldraw[gray, opacity=0.3] (0,8.6) rectangle (25,9.4) ;
		\filldraw[gray, opacity=0.3] (0,15.6) rectangle (25,16.4) ;
		\draw[thin] (9,0) -- (9,25) ;
		\draw[thin] (14,0) -- (14,25) ;
		\draw[thin] (0,10) -- (25,10) ;
		\draw[thin] (0,16) -- (25,16) ;
		\filldraw (4,19) circle (0.2) ;
		\draw (4,19) node[above left] {$A$} ;
		\filldraw (2,9) circle (0.2) ;
		\draw (2,9) node[below right] {$B_{iii}$} ;
		\filldraw (7,9) circle (0.2) ;
		\draw (6,9) node[below right] {$B_{ii}$} ;
		\draw[>=latex,<->] (2.2,9) -- (6.8,9) ;
		\filldraw (9,4) circle (0.2) ;
		\draw (9,4) node[below left] {$C$} ;
		\filldraw (15,13) circle (0.2) ;
		\draw (15,13) node[above right] {$D_2$} ;
		\filldraw (15,21) circle (0.2) ;
		\draw (15,21) node[above right] {$D_1$} ;
		\draw[>=latex,<->] (15,13.2) -- (15,20.8) ;
		\filldraw (12,16) circle (0.2) ;
		\draw (12,16) node[above right] {$E_1$} ;
		\filldraw (21,16) circle (0.2) ;
		\draw (21,16) node[above right] {$E_2$} ;
		\draw[>=latex,<->] (12.2,16) -- (20.8,16) ;
		\filldraw (18,2) circle (0.2) ;
		\draw (18,2) node[above right] {$F_i$} ;
		\filldraw (18,6) circle (0.2) ;
		\draw (18,6) node[above right] {$F_{iii}$} ;
		\draw[>=latex,<->] (18,2.2) -- (18,5.8) ;
	\end{tikzpicture}
	\caption{The black dots are on the graph $j=w(i)$; different possibilities for the position of a given point are connected by an arrow. Gray bands are either of $i$-coordinate in $w^{-1}(\Omega)$ or of $j$-coordinate in $\Omega$.}
	\label{figPatternsProof}
\end{figure}
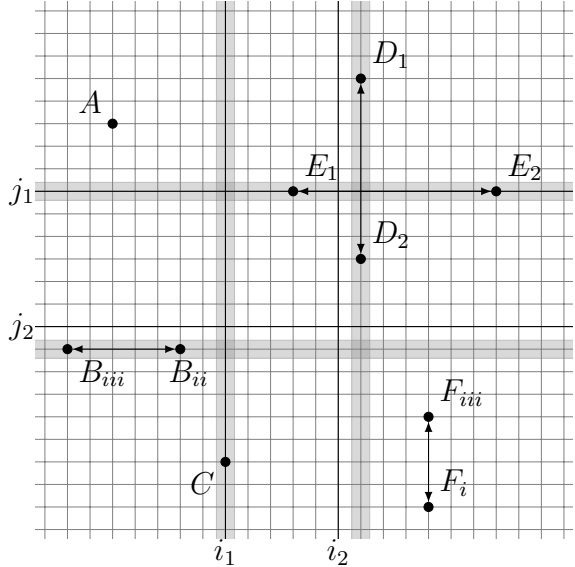

The third and final step is to distinguish cases according to the relative positions of the points.
\begin{description}[font=\normalfont\itshape]
\item[Case 1: $j_D>j_A$ and $i_E>i_F$]
Necessarily $D\neq E$ since $j_D>j_A=j_E$.

\item[Case 1.i: $j_F<j_C$]
$ACDFE$ gives the pattern $[42513]$; that is to say, $w$ flattens to $[42513]$ through $\{i_A,i_C,i_D,i_F,i_E\}$.

\item[Case 1.ii: $i_B>i_A$]
$ABDFE$ gives the pattern $[42513]$.

\item[Case 1.iii: $j_F>j_C$ and $i_B<i_A$]
In this case $B\neq C$ because $i_B<i_A<i_1=i_C$, and $BACDFE$ gives the pattern $[351624]$.

\item[Case 2: $j_D<j_A$ or $i_E<i_F$] Then, $D$ in the case $j_D<j_A$, or $E$ in the case $i_E<i_F$, have the same relative position with respect to the other four points: that is to say, between $C$ and $F$ along the $i$-axis and between $B$ and $A$ along the $j$-axis. We can therefore just assume that $j_D<j_A$.

\item[Case 2.i: $j_F<j_C$]
$ACDF$ gives the pattern $[4231]$.

\item[Case 2.ii: $i_B>i_A$]
$ABDF$ gives the pattern $[4231]$.

\item[Case 2.iii: $j_F>j_C$ and $i_B<i_A$]
As in case 1.iii, we have $B\neq C$; and here $BACDF$ gives the pattern $[35142]$. \qedhere

\end{description}
\end{proof}

\begin{cor} \label{corBadPairSingularSchubert}
Let $w,w'\in\Scal_n$ such that $(w',w)$ is a bad pair. Then the Schubert variety $\overline{BwB/B}$ is singular (where $B$ is the subgroup of upper triangular matrices and each element of $W=\Scal_n$ can be lifted to $G$).
\end{cor}

\begin{proof}
Since the patterns $[42513]$, $[35142]$ and $[351624]$ all contain the pattern $[3412]$, this is a direct consequence of Theorem \ref{thmBadPairPatterns} and of \cite[Thm.\ 1]{lakshmibaiSandhya}.
\end{proof}

\begin{rem}
Given a bad pair $(w',w)$ in $\Scal_n$, the point $w'B/B$ has no reason to be in the singular locus of $\overline{BwB/B}$. For instance, when $n=4$, the pair $([1324],[4231])$ is bad but the point $[1324]B/B$ is smooth in $\overline{B[4231]B/B}$; in fact the singular locus of $\overline{B[4231]B/B}$ is $\overline{B[2143]B/B}$ (see \cite[\S3]{lakshmibaiSandhya}).
\end{rem}

\section{A conjecture on a Schubert-like decomposition} \label{secConjs}

We recall Conjecture 2.3.7 of \cite{bhs3}, whose study is the main topic of this paper.

\subsection{A scheme related to the Springer resolution} \label{ssecSpringerScheme}

We recall the construction of a reduced scheme $X$ associated to the split reductive group $G$ and equipped with a stratification similar to the Schubert decomposition.

First recall the Schubert decomposition on the flag variety $G/B$. Let $w\in W$ and $\dot{w}\in N_G(T)(k)$ be any closed point lifting $w$. The left-action of $B$ on $G$ by multiplication induces an action on $G/B$, write $BwB/B$ for the orbit of $\dot{w}B/B$ under this action. Then $BwB/B$ is a smooth locally closed subvariety of $G/B$ called a \emph{Schubert cell}, and we have the decomposition
\begin{equation} \label{eqStratification}
	G/B=\coprod_{w\in W}BwB/B
\end{equation}
(see \cite[Thm.\ 21.73]{milne}). The cells $BwB/B$ are stable under the left action of $B$, thus their Zariski closure $\overline{BwB/B}$ too, and this action is transitive on $BwB/B$. Thus, a closed cell $\overline{BwB/B}$ intersects a cell $Bw'B/B$ if and only if $Bw'B/B\subseteq\overline{BwB/B}$, which happens if and only if $w'\preceq w$ for the Bruhat order (see \cite[Thm.\ 2.11]{schubertBGG}). Therefore \eqref{eqStratification} gives a good stratification of $G/B$, and we have a more general decomposition into locally closed strata for all $w\in W$
\[
	\overline{BwB/B}=\coprod_{w'\preceq w}Bw'B/B \,.
\]
One can get a similar stratification on $G/B\times G/B$: consider the isomorphism
\begin{equation} \label{equivalentCells}
	f \colon
	\begin{array}{ccc}
		G\times^BG/B &\isoto &G/B\times G/B \\
		(g_1,g_2B) &\longmapsto &(g_1B,g_1g_2B)
	\end{array}
\end{equation}
where $G\times^BG/B$ is the quotient of $G\times G/B$ by the right $B$-action defined by $(g_1,g_2B)b\coloneqq (g_1b,b^{-1}g_2B)$ (see \cite[\S3.7]{slodowy}). For $w\in W$, the right $B$-action leaves $G\times(BwB/B)$ invariant, so we can write
\[
	U_w \coloneqq f(G\times^B(BwB/B)) \subseteq G/B\times G/B \,.
\]
Equivalently, $U_w$ is the orbit of $(B/B,\dot{w}B/B)$ under the left $G$-action on $G/B\times G/B$ given by diagonal left multiplication. Then $G/B\times G/B=\coprod_{w\in W}U_w$ and for any $w\in W$, the decomposition into locally closed strata
\[
	\overline{U_w}=\coprod_{w'\preceq w}U_{w'}
\]
is a consequence of the analogous one in $G/B$.

As in \cite{bhs3}, let $X$ be the following algebraic variety:
\[
	X\coloneqq\enstq{(g_1B,g_2B,\psi)\in G/B\times G/B\times\gfrak}{\Ad(g_1^{-1})\psi\in\bfrak\,,\,\Ad(g_2^{-1})\psi\in\bfrak} \,.
\]
Let us write $\pi$ for the composition $X\inj G/B\times G/B\times\bfrak\surj G/B\times G/B$. For $w\in W$, define
\[
	V_w  \coloneqq \pi^{-1}(U_w) \subset X \,, \quad X_w \coloneqq \overline{V_w} \,.
\]
Note that $V_w$ is open in $X_w$ but not in $X$, since it is only locally closed. We also define two morphisms
\[
	\appl{\kappa_1}{X}{\tfrak}{(g_1B,g_2B,\psi)}{\overline{\Ad(g_1^{-1})\psi}} \,, \quad \appl{\kappa_2}{X}{\tfrak}{(g_1B,g_2B,\psi)}{\overline{\Ad(g_2^{-1})\psi}}
\]
where $\overline{\Ad(g_1^{-1})\psi}$ (resp.\ $\overline{\Ad(g_2^{-1})\psi}$) denotes the image of $\Ad(g_1^{-1})\psi$ (resp.\ $\Ad(g_2^{-1})\psi$) under the projection $\bfrak\surj\bfrak/\ufrak=\tfrak$.

We have the left action of $G$ on $G/B\times G/B$ by diagonal left multiplication, as well as the action of $G$ on $\gfrak$ by adjunction. This gives a diagonal action on $G/B\times G/B\times\gfrak$ by $g\cdot(g_1B,g_2B,\psi)=(gg_1B,gg_2B,\Ad(g)\psi)$. For this action, $X$ is stable and $\pi\colon X\to G/B\times G/B$ is $G$-equivariant. It therefore induces an action on $V_w$ for any given $w\in W$.

The cell decompositions $\overline{U_w}=\coprod_{w'\preceq w}U_{w'}$ for $w\in W$ imply that $X_w\inter V_{w'}=\emptyset$ unless $w'\preceq w$, and that (set theoretically)
\[
	X_w=\coprod_{w'\preceq w}X_w\inter V_{w'} \,.
\]

\subsection{Statement of the conjecture} \label{ssecMainResult}

\begin{conj}[Conjecture 2.3.7 of \cite{bhs3}] \label{conjLocalModel}
Let $w,w'\in W$ such that $w'\preceq w$. Then
\[
	X_w\inter V_{w'} = V_{w'}\inter{\kappa_1}^{-1}(\tfrak^{ww'^{-1}})
\]
as reduced closed subschemes of $X$.
\end{conj}

A consequence of this conjecture is the following proposition, that is mentioned in \cite[Remark 2.5.4]{bhs3} at least when $w=w_0$ is the element of maximal length in $W$. We provide here a proof for the general case.

\begin{prop} \label{conjImpliesConj3}
Let $w,w'\in W$ and $x=(\pi(x),0)\in G/B\times G/B\times\bfrak$ be a closed point of $X_w\inter V_{w'}$. Let $T_{X_w,x}$ (resp.\ $T_{\overline{U_w},\pi(x)}$) be the tangent space of $X_w$ at $x$ (resp.\ of $\overline{U_w}$ at $\pi(x)$) which is a vector space over the residue field $k(x)$ of $X_w$ at $x$ (resp.\ the residue field $k(\pi(x))$ of $\overline{U_w}$ at $\pi(x)$). If Conjecture \ref{conjLocalModel} is true for $(w',w)$, then 
\begin{equation} \label{prop2.5.3}
	\dim_{k(x)}T_{X_w,x} = \dim_{k(\pi(x))}T_{\overline{U_w},\pi(x)} + \dim_{k(x)}\tfrak^{ww'^{-1}}(k(x)) + \length(w'w_0) \,.
\end{equation}
\end{prop}

\begin{proof}
Since the action of $G$ on $U_{w'}$ is transitive and $\pi$ is $G$-equivariant, we can assume $\pi(x)=(B,w'B)$ (in particular the residue field $k(x)$ is $k$). Consider the composition $X_w\inj X\inj G/B\times G/B\times\gfrak$; it factors through a closed immersion $\iota\coloneqq X_w\inj\overline{U_w}\times\gfrak$. Because $\pi(x)=(B,w'B)$, the proof of \cite[Proposition 2.5.3]{bhs3} shows that the differential $\diff_x\iota\colon T_{X_w,x}\inj T_{\overline{U_w},\pi(x)}\oplus T_{\gfrak,0}=T_{\overline{U_w},\pi(x)}\oplus\gfrak(k)$ factors through an injection of $k$-vector spaces
\begin{equation} \label{injectionVectorSpaces}
	\diff_x\iota \colon T_{X_w,x} \inj T_{\overline{U_w},\pi(x)} \oplus \tfrak^{ww'^{-1}}(k) \oplus (\ufrak(k)\inter\Ad(w')\ufrak(k)) \,.
\end{equation}
Assuming Conjecture \ref{conjLocalModel}, we wish to prove the surjectivity of this map; comparing dimensions would then give the desired equality \eqref{prop2.5.3}.

Consider the closed immersion $G/B\times G/B\inj X$ given by $(g_1B,g_2B)\longmapsto(g_1B,g_2B,0)$. Its restriction to $\overline{U_w}$ factors through a closed immersion $\iota_1\coloneqq\overline{U_w}\inj X_w$. Since $\iota\circ\iota_1$ is the closed immersion $\overline{U_w}\to\overline{U_w}\times\gfrak$ given by $u\to(u,0)$, the composition $\diff_x\iota\circ\diff_{\pi(x)}\iota_1$ is the inclusion $T_{\overline{U_w},\pi(x)}\inj T_{\overline{U_w},\pi(x)}\oplus\gfrak(k)$. Hence, in \eqref{injectionVectorSpaces}, the first summand of the right-hand side is contained in the image of $\diff_x\iota$.

Now, we know from \cite[Proposition 2.2.1]{bhs3} that $\pi\colon V_{w'}\to U_{w'}$ is a geometric vector bundle, and the fiber of $\pi(x)$ in $X$ is
\begin{equation} \label{eqFibreSpringer}
	\pi^{-1}(\pi(x)) = (B,w'B) \times (\tfrak\oplus(\ufrak\inter\Ad(w')\ufrak)) \inj U_{w'}\times\gfrak \inj G/B\times G/B\times\gfrak \,.
\end{equation}
Take $Y\coloneqq(B,w'B) \times (\tfrak^{ww'^{-1}}\oplus(\ufrak\inter\Ad(w')\ufrak))$. The closed immersion of reduced schemes $Y\inj X$ factors through $Y\inj V_{w'}$ and $Y\inj\kappa_1^{-1}(\tfrak^{ww'^{-1}})$. This gives a closed immersion
\[
	Y \inj \left(V_{w'}\times_X{\kappa_1}^{-1}(\tfrak^{ww'^{-1}})\right)^\red = V_{w'}\inter{\kappa_1}^{-1}(\tfrak^{ww'^{-1}}) \,.
\]
Conjecture \ref{conjLocalModel} of \cite{bhs3} gives a closed immersion $V_{w'}\inter{\kappa_1}^{-1}(\tfrak^{ww'^{-1}})\inj X_w$. The composition $\iota_2\coloneqq Y\inj X_w$ is such that $\iota\circ\iota_2$ is the composition $Y\inj U_{w'}\times\gfrak\inj\overline{U_w}\times\gfrak$. Therefore, by definition of $Y$, the differential $\diff_x\iota\circ\diff_x\iota_2$ composed with the projection $T_{\overline{U_w}\times\gfrak,x}\surj T_{\gfrak,0}$ has image $\tfrak^{ww'^{-1}}\oplus(\ufrak\inter\Ad(w')\ufrak)$. This gives the second and third summands of \eqref{injectionVectorSpaces}.
\end{proof}

We now formulate a conjecture equivalent to Conjecture \ref{conjLocalModel} using a slightly simpler scheme than $X$. Recall from \cite[\S4.7]{slodowy} or \cite[\S VI.8]{kiehlWeissauer} \emph{Grothendieck's simultaneous resolution of singularities}
\begin{equation}
	\appl{q}{G\times^B\bfrak}{\gfrak}{(g,\psi)}{\Ad(g)\psi}
\end{equation}
where $G\times^B\bfrak$ is the quotient of $G\times\bfrak$ by the right $B$-action given by $(g,\psi)b\coloneqq(gb,\Ad(b^{-1})\psi)$. There is also a map $\widetilde\pi\colon G\times^B\bfrak\to G/B$ given by $\widetilde\pi(g,\psi)=gB$.

Define, for $w\in W$:
\[
	\widetilde{X}\coloneqq q^{-1}(\bfrak) \,,\quad \widetilde{V}_w\coloneqq \widetilde{X}\inter{\widetilde\pi}^{-1}(BwB/B) \,,\quad \widetilde{X}_w\coloneqq\overline{\widetilde{V}_w} \,.
\]
As for $X$, we have the decomposition $\widetilde{X}_w=\coprod_{w'\preceq w}\widetilde{X}_w\inter\widetilde{V}_{w'}$ for any $w\in W$, and $\widetilde{X}_w\inter\widetilde{V}_{w'}\neq\emptyset$ implies that $w'\preceq w$.

\begin{conj}[Breuil-Hellmann-Schraen] \label{conjLocalModelSimple}
Let $w,w'\in W$. If $w'\preceq w$, then the locally closed immersion $\widetilde{V}_{w'}\inter q^{-1}(\tfrak^{ww'^{-1}}+\ufrak)\inj\widetilde{X}$ factors into a locally closed immersion
\[
	\widetilde{V}_{w'}\inter q^{-1}(\tfrak^{ww'^{-1}}+\ufrak) \inj \widetilde{X}_w \inj\widetilde{X} \,.
\]
\end{conj}

\begin{prop} \label{conjImpliesConj2}
Conjecture \ref{conjLocalModel} and Conjecture \ref{conjLocalModelSimple} are equivalent.
\end{prop}

\begin{proof}
We have an isomorphism
\begin{equation} \label{eqIsoLocalModels}
	G\times^B\widetilde{X} \isoto X ,\quad (g_1,(g_2,\psi))\longmapsto(g_1B,g_1g_2B,\Ad(g_1g_2)\psi)
\end{equation}
where $B$ acts on the right on $G\times\widetilde{X}$ by $(g_1,(g_2,\psi))b\coloneqq (g_1b,(b^{-1}g_2,\psi))$. This isomorphism commutes with $\kappa_1$ (on the right) and with the composition of $(g_1,(g_2,\psi))\mapsto q(g_2,\psi)\in \bfrak$ with $\bfrak\to\bfrak/\ufrak=\tfrak$ (on the left), identifies $G\times^B\widetilde{V}_{w'}$ with $V_{w'}$ and identifies $G\times^B\widetilde{X}_w$ with $X_w$. Thus everything translates from $\widetilde{X}$ to $X$. Therefore Conjecture \ref{conjLocalModelSimple} is equivalent to having a locally closed immersion $V_{w'}\inter\kappa_1^{-1}(\tfrak^{ww'^{-1}})\inj X_w\inter V_{w'}$, and \cite[Lemma 2.3.4]{bhs3} shows that it has to be an isomorphism.
\end{proof}

Using the formulation of Conjecture \ref{conjLocalModelSimple}, we prove in \S\ref{secProofGoodPairs} that Conjecture \ref{conjLocalModel} is true in the case of good pairs (in the sense of Definition \ref{dfnGoodPair}), and we exhibit in \S\ref{secEquations} counter-examples showing that it is false in general (for bad pairs).

\subsection{Application to tangent spaces on the trianguline variety} \label{ssecTrianguline}

We recall the setting in which the variety $X$ was originally introduced by Breuil-Hellmann-Schraen \cite{bhs3} who were motivated by the study of the \emph{trianguline variety} from $p$-adic Hodge theory.

Let $p$ be a prime number and $K,L$ be finite extensions of $\Q_p$ such that $\Sigma\coloneqq\Hom_{\Q_p}(K,L)$ has $[K:\Q_p]$ elements. Let $K_0$ be the maximal unramified extension of $\Q_p$ inside $K$. We write $\G_K$ for the absolute Galois group of $K$ and $k_L$ for the residual field of $L$. We fix an integer $n\in\Z_{>0}$ and a continuous representation $\bar{r}\colon\G_K\to\GL_n(k_L)$. We set $G$ to be the split reductive group $G\coloneqq\left(\Res_{K/\Q_p}G'_K\right)\times_{\Q_p}L\iso\prod_{\tau\in\Sigma}G'_L$ over $L$, $T\coloneqq\prod_{\tau\in\Sigma}T'_L$ and $B\coloneqq\prod_{\tau\in\Sigma}B'_L$, where $G'$ is the split reductive group $\GL_{n,\Q_p}$ over $\Q_p$, $T'$ is the maximal torus in $G'$ of diagonal matrices and $B'$ is the Borel subgroup of $G'$ of upper-triangular matrices. Note that the Weyl group of $G$ is $(\Scal_n)^\Sigma$.

Let $\Rr$ be the framed deformation ring of $\bar{r}$ constructed by Mazur \cite[§1.2]{mazur} and Kisin \cite{kisin09}. Taking its formal spectrum $\Spf(\Rr)$ and then applying Berthelot's rigidification functor \cite[§0.2]{berthelot} gives a rigid analytic space $\Xfrakr\coloneqq\Spf(\Rr)^\rig$ over $L$. Let $\Tcal_L\coloneqq\widehat{K^\times}\times_{\Q_p}L$ be the rigid analytic space over $L$ parametrising the continuous characters of $K^\times$. For any $\kbold=(k_\tau)\in\Z^\Sigma$, we write $z^\kbold\in\Tcal_L(L)$ for the character $z\longmapsto\prod_{\tau\in\Sigma}\tau(z)^{k_\tau}$, and $\abs{z}_K\in\Tcal_L(L)$ for the $p$-adic absolute value normalised by $\abs{p}_K=p^{-[K:\Q_p]}$. A character $\underline{\delta}=(\delta_i)\in\Tcal_L^n$ is called regular if, for all $1\leq i\neq j\leq n$, the character $\delta_i\delta_j^{-1}$ is not of the form $z^{-\kbold}$ or $N_{K/\Q_p}(z)z^{\kbold}\abs{z}_K$ for some $\kbold\in\Z_{\geq0}^\Sigma$, where $N_{K/\Q_p}$ is the norm map. We write $\Tnreg\subset\Tcal_L^n$ for the subset of regular characters, which is Zariski-open.

\begin{dfn} \label{dfnTrianguline}
Let $\Utri$ be the set of points $x=(r,\underline{\delta})\in\Xfrakr\times\Tnreg$ such that the representation $r\colon\G_K\to\GL_n(k(x))$ is trianguline in the sense of Colmez with parameter $\underline{\delta}\colon(K^\times)^n\to k(x)^\times$ (see \emph{e.g.\ }\cite[Definition 3.3.8]{bhs3}). The \emph{trianguline variety} $\Xtri$ is defined as the Zariski-closure of $\Utri$ in $\Xfrakr\times\Tcal_L^n$.
\end{dfn}

The geometry of $\Utri$ is well understood: $\Utri$ is a Zariski-dense and Zariski-open subset of $\Xtri$ which is smooth of dimension $n^2+[K:\Q_p]\frac{n(n+1)}{2}$; see \cite[Thm.\ 2.6]{bhs1}. In particular, $\Xtri$ is also equidimensional of dimension $n^2+[K:\Q_p]\frac{n(n+1)}{2}$.

Now let $x=(r,\underline{\delta})\in\Xtri$ be a crystalline point, \emph{i.e.\ }a point such that $r$ is a crystalline representation in the sense of Fontaine \cite{fontaine94III}. Recall that there is a module $\DdR(r)$ over $K\tens_{\Q_p}L\simeq\prod_{\tau\in\Sigma}L$ canonically associated to $r$, equipped with a (canonical) decreasing filtration $\Fil^\bullet\DdR(r)$; for each $\tau\in\Sigma$, there are exactly $n$ integers $h$ such that $\Fil^{-h}\DdR(r)\tens_{K,\tau}L\neq\Fil^{-h+1}\DdR(r)\tens_{K,\tau}L$ (with multiplicities), called the $\tau$-Hodge-Tate weights of $r$ and denoted $h_{\tau,1}\leq\ldots\leq h_{\tau,n}$. Also recall that there is a $K_0\tens L$-module $\Dcris(r)$ canonically associated to $r$, equipped with a canonical endomorphism $\varphi$ which is semilinear with respect to the Frobenius of $K_0$. We assume that $r$ is generic in the sense of \cite{bhs3}, which means that for each $\tau\in\Sigma$, the $\tau$-Hodge-Tate weights of $r$ are all distinct, and that $\varphi/\varphi'\notin\{1,p^{[K_0:\Q_p]}\}$ for any two eigenvalues $\varphi,\varphi'$ of the linearised Frobenius $\Phi\coloneqq\varphi^{[K_0:\Q_p]}$ on $\Dcris(r)$ (with multiplicities).

From \cite[Lemma 2.1]{bhs2}, we know that $\underline{\delta}$ is determined by the data of (i) an ordering $\underline{\varphi}=(\varphi_1,\ldots,\varphi_n)$ of the eigenvalues of $\Phi$ and (ii) for each $\tau\in\Sigma$, a permutation $w_\tau$ in $\Scal_n$. More precisely, one can write $\underline{\delta}=z^{w(\hbold)}\nr(\underline{\varphi})$, where the character $\nr(\underline{\varphi})=(\nr(\varphi_i))_i$ of $(K^\times)^n$ is defined by $\nr(\varphi_i)|_{\intring_K^\times}=1$ and $\nr(\varphi_i)(\varpi)=\varphi_i$ for a uniformiser $\varpi$ of $K$, and where $w\coloneqq(w_\tau)\in(\Scal_n)^\Sigma$ and $w(\hbold)\coloneqq\big(h_{\tau,w_\tau^{-1}(i)}\big)_{\tau,i}\in(\Z^n)^\Sigma$. Moreover, there is a unique $w_\sat\in(\Scal_n)^\Sigma$ such that the point $x_\sat\coloneqq\left(r,z^{w_\sat(\hbold)}\nr(\underline{\varphi})\right)\in\Xfrakr\times\Tcal_L^n$ lies in $\Utri$.

Breuil-Hellmann-Schraen have proved that $X_w$ at a certain point $x_\mathrm{pdR}\in X_w\inter V_{w_\sat}$ is a \emph{local model} of $\Xtri$ at the point $x$, see \cite[Thm.\ 1.6]{bhs3} for more details. As a consequence, they find that the pair $(w_\sat,w)\in\left((\Scal_n)^\Sigma\right)^2$ attached to $x=(r,\underline{\delta})\in\Xtri$ as above satisfies $w_\sat\preceq w$ for the product Bruhat order on $(\Scal_n)^\Sigma$ (\emph{loc.\ cit.}, Theorem 1.8). They also find an upper bound for $\dim T_{\Xtri,x}$ (\emph{loc.\ cit.}, Proposition 4.1.5(ii)). Conjecture \ref{conjLocalModel} has the following consequence:

\begin{thm} \label{thmTrianguline}
If Conjecture \ref{conjLocalModel} is true for the pair $(w_\sat\preceq w)$, then
\begin{equation} \label{eqDimensionEspaceTangent}
	\dim_{k(x)}T_{\Xtri,x} = \dim\Xtri - d_{ww_\sat^{-1}} + \dim_{\Q_p}T_{\overline{(BwB/B)},w_\sat B} - \length(w_\sat)
\end{equation}
where $k(x)$ is the residue field of $\Xtri$ at $x$.
\end{thm}

\begin{proof}
From \eqref{equivalentCells} and the definition of $U_w$ in \S\ref{ssecSpringerScheme}, one derives $\dim_{\Q_p}T_{\overline{(BwB/B)},w_\sat B}=\dim_{\Q_p}T_{\overline{U_w},(B,w_\sat B)}-[K:\Q_p]\frac{n(n-1)}{2}=\dim_{\Q_p}T_{\overline{U_w},\pi(x_\mathrm{pdR})}-[K:\Q_p]\frac{n(n-1)}{2}$. Hence \eqref{eqDimensionEspaceTangent} is a consequence of Proposition \ref{conjImpliesConj3} and the proof of \cite[Proposition 4.1.5]{bhs3}.
\end{proof}

We prove in \S\ref{secProofGoodPairs} below that the hypothesis of Theorem \ref{thmTrianguline} (and hence \eqref{eqDimensionEspaceTangent}) is always satisfied if $(w_\sat,w)$ is a good pair in $(\Scal_n)^\Sigma=\Phi(G,T)$ in the sense of Definition \ref{dfnGoodPair}. The following result gives a very partial converse.

\begin{prop} \label{propSmoothnessXtri}
If \eqref{eqDimensionEspaceTangent} is true and $\Xtri$ is smooth at $x$, then $(w_\sat,w)$ is a good pair in $(\Scal_n)^\Sigma$ and $\overline{(BwB/B)}$ is smooth at $w_\sat B$.
\end{prop}

\begin{proof}
Since $\dim{\overline{(BwB/B)}}=\lg(w)$, this is a direct consequence of Proposition \ref{propEqualityImpliesGoodPair}.
\end{proof}

\section{The case of good pairs} \label{secProofGoodPairs}

This section is devoted to proving the following theorem.

\numberwithin{thm}{section}
\begin{thm} \label{thmConjIsTrue}
Let $(w',w)$ be a good pair in $W$ in the sense of Definition \ref{dfnGoodPair}. Then Conjecture \ref{conjLocalModel} and Conjecture \ref{conjLocalModelSimple} are true for the pair $(w',w)$, \emph{i.e.\ }the locally closed immersion $\widetilde{V}_{w'}\inj\widetilde{X}$ induces
\begin{equation}
	\widetilde{V}_{w'}\inter q^{-1}(\tfrak^{ww'^{-1}}+\ufrak) \inj \widetilde{X}_w \,.
\end{equation}
\end{thm}
\numberwithin{thm}{subsection}

\subsection{Preliminaries: Levi subgroups and the adjoint action} \label{ssecLeviAdjoint}

We establish a few technical lemmas about Levi subgroups and the derived adjoint action, which are used in the proof of Theorem \ref{thmConjIsTrue}.

Let $J\subseteq I$ be a subset of simple roots of $G$, write $\Phi_J\subseteq\Phi$ for the root subsystem generated by $J$ and denote by $T_J\coloneqq \left(\bigcap_{\alpha\in\Phi_J}(\ker\alpha)\right)^\circ\subseteq T$ the largest subtorus of $T$ such that $\alpha(T_J)=1$ for all $\alpha\in\Phi_J$ (or equivalently $\alpha\in J$). It is indeed a torus, as it is a reduced (the characteristic is $0$) and connected diagonalisable group; see the discussion in \cite[Notation 12.29]{milne}.

Let $L\coloneqq C_G(T_J)$ be the standard Levi subgroup of $G$ associated to $J$. According to \cite[Proposition 21.90]{milne}, $L$ is a connected reductive group with $T$ as a maximal torus and root system $\Phi_J$. Therefore $L=\engendre{T,U_\alpha}{\alpha\in\Phi_J}$, \emph{i.e.\ }$L$ is the subgroup of $G$ generated by $T$ and the $U_\alpha$ for $\alpha\in\Phi_J$. Write $U_L\coloneqq\engendre{U_\alpha}{\alpha\in\Phi_J\inter\Phi_+}$ for the unipotent radical of $L\inter B$ and $\tfrak_J$, $\lfrak$, $\ufrak_L$ for the Lie algebras of $T_J$, $L$, $U_L$ respectively.

For a root $\alpha\colon T\to\Gm$, consider its differential $\diff_e\alpha\colon\tfrak\to k$; we can see it as $\mathrm{Lie}(\alpha)$, where $\mathrm{Lie}$ is the Lie algebra functor from algebraic groups over $k$ to Lie algebras over $k$. Since $\mathrm{Lie}$ commutes with finite limits, and in particular with kernels and fibred products (see \cite[10.14]{milne}), we have
\[
	\tfrak_J = \bigcap_{\alpha\in\Phi_J}(\ker\diff_e\alpha) \,.
\]

\begin{prop} \label{propGenericTorus}
There is a Zariski-dense open subscheme $\tfrak_J^\gen$ of the affine scheme $\tfrak_J$ such that, for any closed point $t\in\tfrak_J^\gen$:
\begin{enumerate}[label=\emph{(\arabic*)},ref=(\arabic*)]
\item \label{itemGenericTorus}
$L = C_G(t)^\circ \coloneqq \enstq{g\in G}{\Ad(g)t=t}^\circ$;
\item \label{itemGenericTorusLie}
$\lfrak = \zfrak_\gfrak(t) \coloneqq \enstq{x\in\gfrak}{[x,t]=0}$.
\end{enumerate}
\end{prop}

\begin{proof}
For any closed point $t\in\tfrak$, the centraliser $C_G(t)$ contains $T$ and is normalised by $T$, therefore the same is true of the identity component $C_G(t)^\circ$ since $T$ is connected. By \cite[21.66]{milne}, we then have
\begin{equation} \label{eqCentraliser}
	C_G(t)^\circ = \engendre{T,U_\alpha}{\alpha\in\Phi,\,U_\alpha\subseteq C_G(t)^\circ} = \engendre{T,U_\alpha}{\alpha\in\Phi,\,U_\alpha\subseteq C_G(t)}
\end{equation}
since each $U_\alpha$ is connected.

Let $\alpha\in\Phi$ be a root and $t\in\tfrak(\kalg)$ be a $\kalg$-point such that $U_\alpha\subseteq C_G(t)$. Then the morphism
\[
	\phi\colon U_\alpha\to\gfrak \,, \quad u\longmapsto\Ad(u)t
\]
is constant, so its differential
\[
	\diff_e\colon\ufrak_\alpha\to\gfrak \,, \quad u_\alpha\longmapsto\ad(u_\alpha)t
\]
is zero. But for any $u_\alpha\in\ufrak_\alpha(\kalg)$, we have $\ad(u_\alpha)t=[u_\alpha,t]=-\diff_e\alpha(t)u_\alpha$. This shows that $\diff_e\alpha(t)=0$ and hence $t\in\ker\diff_e\alpha$.

Not let $\alpha\in\Phi\setminus\Phi_J$. Since $\Phi_J$ is the root system of $L$, necessarily $U_\alpha\not\subseteq L=C_G(T_J)$. Therefore there exist $u_\alpha\in U_\alpha(\kalg)$ and $t\in T_J(\kalg)$ such that $u_\alpha t\neq tu_\alpha$, which means $u_\alpha\neq tu_\alpha t^{-1}$. This implies $\alpha(t)\neq1$ (see \cite[21.19]{milne}). Hence $\ker(\alpha)\inter T_J\neq T_J$. Applying the $\mathrm{Lie}$ functor (which preserves finite limits), we get $\ker(\diff_e\alpha)\inter\tfrak_J\neq\tfrak_J$, so $\tfrak_J\setminus(\ker(\diff_e\alpha)\inter\tfrak_J)$ is non-empty hence dense open in $\tfrak_J$. We thus define
\[
	\tfrak_J^\gen \coloneqq \tfrak_J\setminus\bigg(\bigcup_{\alpha\in\Phi\setminus\Phi_J}\ker(\diff_e\alpha)\inter\tfrak_J\bigg) = \bigcap_{\alpha\in\Phi\setminus\Phi_J}\left(\tfrak_J\setminus\left(\ker(\diff_e\alpha)\inter\tfrak_J\right)\right)
\]
which is dense open in $\tfrak_J$.

Now, fix a closed point $t\in\tfrak_J^\gen$. For any root $\alpha\in\Phi$, the previous discussion shows that $U_\alpha\subseteq C_G(t)$ implies $\alpha\in\Phi_J$. By \eqref{eqCentraliser}, we have
\[
	C_G(t) \subseteq \engendre{T,U_\alpha}{\alpha\in\Phi_J} = L \,.
\]
The converse $L=C_G(T_J)\subseteq C_G(t)$ is true for any $t\in\tfrak_J$, so \ref{itemGenericTorus} is proved.

Statement \ref{itemGenericTorusLie} is a consequence of \ref{itemGenericTorus} and \cite[Proposition 17.76]{milne}.
\end{proof}

The following lemma explains the relevance of Proposition \ref{propGenericTorus}.

\begin{lem} \label{lemDenseIsEnough}
Let $\tfrak'\subseteq\tfrak^{ww'^{-1}}$ be a Zariski-dense subscheme. Then the algebraic variety (over $k$)
\[
	\widetilde{V}(w',\tfrak') \coloneqq \widetilde{V}_{w'}\inter q^{-1}(\tfrak'+\ufrak)
\]
is a dense subvariety of
\[
	\widetilde{V}(w',\tfrak^{ww'^{-1}}) \coloneqq \widetilde{V}_{w'}\inter q^{-1}(\tfrak^{ww'^{-1}}+\ufrak) \,.
\]
\end{lem}

\begin{proof}
Recall the morphism $\widetilde\pi\colon G\times^B\bfrak\to G/B$, which restricts to two morphisms
\begin{equation*}
	\widetilde\pi_1 \colon q^{-1}(\tfrak^{ww'^{-1}}+\ufrak) \to G/B\ ,\quad
	\widetilde\pi_1' \colon  q^{-1}(\tfrak'+\ufrak) \to G/B \,.
\end{equation*}
Then $\widetilde{V}(w',\tfrak^{ww'^{-1}})$ (resp. $\widetilde{V}(w',\tfrak')$) is just the fibre of $\widetilde\pi_1$ (resp. $\widetilde\pi_1'$) above $Bw'B/B$. It is thus enough to check density on fibres above closed points of $Bw'B/B$: let $gB\in Bw'B/B\subset G/B$ with $g\in G$, we want to show that $\widetilde\pi_1'^{-1}(gB)$ is dense in $\widetilde\pi_1^{-1}(gB)$. Now the left-action of $B$ on $G\times\bfrak$ given by $b(g,\psi)=(bg,\psi)$ induces a left-action on $G\times^B\bfrak$. This action stabilises $q^{-1}(\tfrak^{ww'^{-1}}+\ufrak)$ and $q^{-1}(\tfrak'+\ufrak)$ because $\Ad(b)$ stabilises $t+\ufrak$ for any $b\in B(\kalg)$ and $t\in\tfrak(\kalg)$. Moreover, this action induces an action on the set of fibres of $\widetilde\pi_1$ which is continuous in the sense that, for any $b\in B$ and $gB\in G/B$, we have an isomorphism of algebraic varieties $\widetilde\pi_1^{-1}(gB)\to\widetilde\pi_1^{-1}(bgB),\,x\longmapsto bx$; and similarly for $\widetilde\pi_1'$. Also, this action on the fibres is obviously transitive on the subset of fibres at points of $Bw'B/B$.

Therefore it suffices to consider the fibre at $w'B\in G/B$: we want to check that $\widetilde\pi_1'^{-1}(ww'B)$ is dense in $\widetilde\pi_1^{-1}(w'B)$. Combining the formula \eqref{eqFibreSpringer} from \cite[Proposition 2.2.1]{bhs3} with the isomorphism \eqref{eqIsoLocalModels} gives
\[
	\widetilde\pi_1'^{-1}(ww'B) = \enstq{(w',\psi)\in G\times^B\bfrak}{\psi\in\tfrak'\oplus(\ufrak\inter\Ad(w')\ufrak)}
\]
and
\[
	\widetilde\pi_1^{-1}(ww'B) = \enstq{(w',\psi)\in G\times^B\bfrak}{\psi\in\tfrak^{ww'^{-1}}\oplus(\ufrak\inter\Ad(w')\ufrak)} \,,
\]
from which we easily deduce the desired density statement.
\end{proof}

We introduce other objects associated to $J$. Let $P\coloneqq P_J=\engendre{U_\alpha}{\alpha\in\Phi_J\cup\Phi_+}$ be the parabolic subgroup associated to $J$ and $U_P\coloneqq\engendre{U_\alpha}{\alpha\in\Phi_+\setminus\Phi_J}$ be the unipotent radical of $P$. Write $W_J\coloneqq W(\Phi_J)\subseteq W$ for the Weyl group of $\Phi_J$; therefore $W_J=N_L(T)/T$ is also the Weyl group of $L$. Write $W^J\subset W$ for the set of $w\in W$ such that $ws_\alpha\succ w$ for all $\alpha\in\Phi_J$, \emph{i.e.\ }$w(\alpha)\in\Phi_+$ for all $\alpha\in\Phi_J\inter\Phi_+$. The reductive group $L$ is the Levi factor of $P$, and we have the decomposition $P=LU_P$ (see \cite[Thm.\ 21.91]{milne}). Also recall from \S\ref{ssecBonnePaireDef} the decomposition $W=W^JW_J$.

\begin{lem} \label{lemNoNegativeComponent}
Let $l\in L$ and $w^J\in W^J$ be such that $w^Jl(w^J)^{-1}\in B$. Then $l\in L\inter B$.
\end{lem}

\begin{proof}
From the Bruhat decomposition in $L$, let $w_J\in W_J$ be such that $l\in(L\inter B)w_J(L\inter B)$. The hypothesis $w^J\in W^J$ can be reformulated as $w^J(L\inter B)(w^J)^{-1}\subseteq B$. Thus we have $w^Jl(w^J)^{-1}\in Bw^Jw_J(w^J)^{-1}B$. The Bruhat decomposition in $G$ and our hypothesis then give $w^Jw_J(w^J)^{-1}=1$, \emph{i.e.\ }$w_J=1$, hence $l\in B$.
\end{proof}

\begin{lem} \label{lem2Breuil}
Let $t\in\tfrak_J^\gen(\kalg)$ be a $\kalg$-point. Let $w^J\in W^J$, $u_P\in U_P(\kalg)$, $l\in L(\kalg)$ and $x\in\ufrak_L(\kalg)$ be such that $\Ad(w^Ju_Pl)(t+x)\in\bfrak(\kalg)$. Then $\Ad(l)x\in\ufrak_L(\kalg)$.
\end{lem}

\begin{proof}
Since $U_P$ is normal in $P$, the action of $L$ on $P$ by conjugation (resp.\ on the Lie algebra of $P$ by adjunction) keeps $U_P$ (resp.\ $\ufrak_P$) stable. Also, the action of $U_P$ by conjugation on $B$ induces a trivial action on $B/U_P$, hence on $\bfrak/\ufrak_P(\kalg)$ as well. Therefore, setting $u_P'\coloneqq l^{-1}u_Pl\in U_P(\kalg)$, we have using $t+x\in\bfrak(\kalg)$:
\[
	y \coloneqq \Ad(u_P')(t+x)-(t+x) \in\ufrak_P(\kalg) \,.
\]
Hence, setting $z\coloneqq\Ad(l)y\in\ufrak_P(\kalg)$, we have using $\Ad(l)t=t$:
\[
	\Ad(w^Ju_Pl)(t+x)=\Ad(w^Jlu_P')(t+x)=\Ad(w^J)(t+\Ad(l)x+z)
\]
which is in $\bfrak(\kalg)=\tfrak(\kalg)\oplus\bigoplus_{\alpha\in\Phi_+}\ufrak_\alpha(\kalg)$ by assumption.

Now we have $\Ad(l)x\in\bigoplus_{\alpha\in\Phi_J}\ufrak_\alpha(\kalg)$ and $z\in\ufrak_P(\kalg)\subset\bigoplus_{\alpha\in\Phi\setminus\Phi_J}\ufrak_\alpha(\kalg)$, therefore $\Ad(w^J)\Ad(l)x$ and $\Ad(w^J)z$ have disjoint components in the sum $\bigoplus_{\alpha\in\Phi}\ufrak_\alpha(\kalg)$. Thus they each cannot have a non-zero component in $\ufrak_\alpha(\kalg)$ for $\alpha\notin\Phi_+$, so
\[
	\Ad(w^J)\Ad(l)x\in\ufrak(\kalg) \,.
\]
However, $w^J\in W^J$ sends negative roots in $\Phi_J$ to negative roots in $\Phi$. Because $\Ad(l)x\in\bigoplus_{\alpha\in\Phi_J}\ufrak_\alpha(\kalg)$, this necessarily implies $\Ad(l)x\in\bigoplus_{\alpha\in\Phi_J\inter\Phi_+}\ufrak_\alpha(\kalg)=\ufrak_L(\kalg)$.
\end{proof}

\begin{lem} \label{lem3Breuil}
With the setting of Lemma \ref{lem2Breuil}, we have $w^Ju_P(w^J)^{-1}\in U(\kalg)$.
\end{lem}

\begin{proof}
We know that $\Ad(l)x\in\ufrak_L(\kalg)$ and that $\Ad(l)t=t$, so replacing $x$ by $\Ad(l)x$ if necessary we can assume that $\Ad(w^Ju_P)(t+x)\in\bfrak(\kalg)$.

Partition $\Phi_+\setminus(\Phi_J\inter\Phi_+)$ into the two sets
\begin{align*}
	\Phi_1 &\coloneqq \enstq{\alpha\in\Phi_+\setminus(\Phi_J\inter\Phi_+)}{w^J(\alpha)>0} \,, \\
	\Phi_2 &\coloneqq \enstq{\alpha\in\Phi_+\setminus(\Phi_J\inter\Phi_+)}{w^J(\alpha)<0}
\end{align*}
and put a total order $\prec$ on $\Phi_2$ such that $\alpha\preceq\beta$ (with $\alpha,\beta\in\Phi_2$) whenever $\alpha>\beta$ (recall that by definition $\alpha>\beta$ if and only if $\alpha-\beta$ is in the $\Z_{\geq0}$-span of $I$). Then, for any total order on $\Phi_1$, there is an isomorphism of affine schemes
\begin{equation*}
	\begin{array}{ccc}
		\left(\prod_{\alpha\in\Phi_1}U_\alpha\right)\times\left(\prod_{\alpha\in\Phi_2}U_\alpha\right) &\isoto &U_P \\
		\left(u_{\alpha_1},\ldots,u_{\alpha_{\abs{\Phi_1}+\abs{\Phi_2}}}\right) &\longmapsto &u_{\alpha_1}\ldots u_{\alpha_{\abs{\Phi_1}+\abs{\Phi_2}}}
	\end{array} \,;
\end{equation*}
this can be seen by applying \cite[Thm.\ 21.68]{milne} to $U$ and to $U_L$, and using the Levi decomposition $U=U_LU_P$. Let $u_{P,1}$ be the component of $u_P$ in $\prod_{\alpha\in\Phi_1}U_\alpha(\kalg)$; by definition of $\Phi_1$, we have $w^Ju_{P,1}(w^J)^{-1}\in U(\kalg)$. In particular,
\[
	\Ad\left(w^Ju_{P,1}^{-1}u_P\right)(t+x) = \Ad\left(\left(w^Ju_{P,1}(w^J)^{-1}\right)^{-1}\right) \Ad\left(w^Ju_P\right)(t+x)\in\bfrak(\kalg) \,.
\]
This means that replacing $u_P$ by $u_{P,1}^{-1}u_P$ changes neither our hypothesis nor our aim.

Therefore, we can assume that $u_P=\prod_{\alpha\in\Phi_2}u_\alpha$ for some uniquely determined $u_\alpha\in U_\alpha(\kalg)$. Write also $x=\sum_{\beta\in\Phi_J\inter\Phi_+}x_\beta$, for uniquely determined $x_\beta\in\ufrak_\beta(\kalg)$. Fix some $\alpha\in\Phi_2$. For any root $\beta\in\Phi_J\inter\Phi_+$, we have $\alpha+\beta\neq 0$ (because both are positive roots), so \cite[Proposition 8.2.3]{springer} implies that
\[
	\Ad(u_\alpha)x_\beta-x_\beta \in \bigoplus_{i,j>0}\ufrak_{i\alpha+j\beta}(\kalg) \,.
\]
This is indeed the case with $\alpha=\beta$ as well since $\Ad(u_\alpha)x_\alpha-x_\alpha=0$. As $i\alpha+j\beta>\alpha$ for $i,j>0$, we can sum over the $x_\beta$ and get
\begin{equation} \label{eqRootsAugment}
	\Ad(u_\alpha)x-x \in \bigoplus_{\gamma>\alpha}\ufrak_\gamma(\kalg) \,.
\end{equation}
Now suppose that $u_P\neq 1$; write $u_P=u_{\alpha_r}u_{\alpha_{r-1}}\ldots u_{\alpha_0}$ for uniquely determined $r\geq 0$ and $u_{\alpha_i}\neq 1$ in $\ufrak_{\alpha_i}$. Remember that the ordering on $\Phi_2$ forbids (in particular) to have $\alpha_0>\alpha_i$ for $i>0$. Applying \eqref{eqRootsAugment}, we get, for each $1\leq i\leq r$
\[
	\Ad(u_{\alpha_i})\Ad(u_{\alpha_{i-1}}\ldots u_{\alpha_0})x - \Ad(u_{\alpha_{i-1}}\ldots u_{\alpha_0})x \in \bigoplus_{\gamma>\alpha_i}u_\gamma(\kalg)
\]
and summing over $i$ yields
\begin{equation} \label{eqNoAlphaComponent}
	\Ad(u_P)x-x \in \bigoplus_{\gamma>\alpha_r}\ufrak_\gamma(\kalg)\oplus\ldots\oplus\bigoplus_{\gamma>\alpha_0}\ufrak_\gamma(\kalg) \subseteq \bigoplus_{\gamma\neq\alpha_0}\ufrak_\gamma(\kalg) \,.
\end{equation}
Since $(u_{\alpha_0}t'u_{\alpha_0}^{-1})t'^{-1}=u_{\alpha_0}(t'u_{\alpha_0}^{-1}t'^{-1})\in U_{\alpha_0}(\kalg)$ for $t'\in T(\kalg)$, we have $\Ad(u_{\alpha_0})t-t\in\ufrak_{\alpha_0}(\kalg)$, which is non-zero because $u_{\alpha_0}\notin\zfrak_\gfrak(t)(\kalg)$ by Proposition \ref{propGenericTorus}\ref{itemGenericTorusLie}. Applying again \eqref{eqRootsAugment} repeatedly and summing as in \eqref{eqNoAlphaComponent} gives
\begin{equation} \label{eqOneAlphaComponent}
	\Ad(u_P)t-t \in \left(\ufrak_{\alpha_0}(\kalg)\setminus\{0\}\right) \oplus \bigoplus_{\gamma\neq\alpha_0}\ufrak_\gamma(\kalg) \,.
\end{equation}
As $x\in\ufrak_L(\kalg)$, we have in particular obviously $t+x\in\tfrak(\kalg)\oplus\bigoplus_{\gamma\neq\alpha_0}\ufrak_\gamma(\kalg)$. Combining with \eqref{eqNoAlphaComponent} and \eqref{eqOneAlphaComponent}, we see that $\Ad(u_P)(t+x)$ has a non-zero component in $\ufrak_{\alpha_0}(\kalg)$. Therefore, $\Ad(w^J)(u_P)(t+x)$ has a non-zero component in $\ufrak_{w^J(\alpha_0)}(\kalg)$. This contradicts the hypothesis that $\Ad(w^J)(u_P)(t+x)\in\bfrak(\kalg)$ because $\alpha_0\in\Phi_2$ means that $w^J(\alpha_0)<0$. As a conclusion, $u_P=1$ so our claim is true.
\end{proof}

The two lemmas \ref{lem2Breuil} and \ref{lem3Breuil} can be nicely summarised as:

\begin{lem} \label{joliCorollaire}
Let $t\in\tfrak_J^\gen(\kalg)$, $w^J\in W^J$, $u_P\in U_P(\kalg)$, $l\in L(\kalg)$ and $x\in\ufrak_L(\kalg)$. Then $\Ad(w^Ju_Pl)(t+x)\in\bfrak(\kalg)$ if and only if $w^Ju_P(w^J)^{-1}\in U(\kalg)$ and $\Ad(l)x\in\ufrak_L(\kalg)$.
\end{lem}

\begin{proof}
The ``only'' way is the combination of the two previous lemmas. For the ``if'' way, assume that $w^Ju_P(w^J)^{-1}\in U(\kalg)$ and $\Ad(l)x\in\ufrak_L(\kalg)$. Then since $w^J\in W^J$, we have $\Ad(w^J)\Ad(l)x\in\bfrak(\kalg)$, so $\Ad(w^Jl)(t+x)=\Ad(w^J)(t+\Ad(l)x)\in\bfrak(\kalg)$, hence $\Ad(w^Ju_P(w^J)^{-1})\Ad(w^Jl)(t+x)\in\bfrak(\kalg)$ which is what we want to prove.
\end{proof}

\subsection{Proof of Theorem \ref{thmConjIsTrue}} \label{ssecExtensionLevi}

This part contains the heart of the proof of Theorem \ref{thmConjIsTrue}. Before beginning the proof, we make the following two observations about base change to the algebraic closure $\kalg$.
\begin{enumerate}[1.]
\item
If $S$ is a scheme over $k$, then $(S_\kalg)^\red=(S^\red)_\kalg$; indeed $(S^\red)_\kalg$ is reduced by \cite[\href{https://stacks.math.columbia.edu/tag/020I}{Lemma 020I}]{stacks}, hence it satisfies the universal property defining $(S_\kalg)^\red$. Therefore our definition of intersections of varieties and inverse image of morphism of varieties commute with base change. In particular, starting from the reductive group $G_\kalg$ over $\kalg$ with Borel subgroup $B_\kalg$ and respective Lie algebras $\gfrak_\kalg$ and $\bfrak_\kalg$, one can define the morphism $q_\kalg\colon G_\kalg\times^{B_\kalg}\bfrak_\kalg\to\gfrak_\kalg$ and the varieties $\widetilde{X}_\kalg$, $\widetilde{V}_{w,\kalg}$ and $\widetilde{X}_{w,\kalg}$ for $w\in W$, similarly as for the split reductive group $G$ over $k$. We then easily see that these constructions commute with base change, and that there is a natural isomorphism for any $w'\in W$
\[
	\widetilde{V}_{w',\kalg}\inter q_\kalg^{-1}(\tfrak^{ww'^{-1}}_\kalg+\ufrak_\kalg) \isoto \left(\widetilde{V}_{w'}\inter q^{-1}(\tfrak^{ww'^{-1}}+\ufrak)\right)\times_k\kalg .
\]
\item
If $Z\inj S$ is a closed immersion of $k$-schemes, then a morphism $S'\to S$ of $k$-schemes factors through $Z$ if and only if the base change $S'_\kalg\to S_\kalg$ factors through $Z_\kalg$. Indeed, it suffices to check this for affine schemes, \emph{i.e.\ }to check that if $I$ is an ideal of a $k$-algebra $A$, then any morphism of $k$-algebras $A\to B$ factors through $A/I$ if and only if the extension of scalars $A\tens_k\kalg\to B\tens_k\kalg$ factors through $(A/I)\tens_k\kalg=(A\tens_k\kalg)/(I\tens_k\kalg)$.
\end{enumerate}

Therefore we can, and do, assume that $k$ is algebraically closed for the rest of \S\ref{secProofGoodPairs}. In particular, the set of closed points of any $k$-scheme $S$ is in natural bijection with the set $S(k)$ of $k$-points of $S$.

We use the following notation. Let $L$ be the Levi subgroup of a standard parabolic subgroup $P$ of $G$ (\emph{i.e.\ }such that $B\subseteq P$); there exists $J\subseteq I$ such that $P=P_J$. Since $L$ is a connected reductive closed subgroup of $G$ and $L\inter B$ is a Borel subgroup of $L$, one can define $q_L\colon L\times^{L\inter B}\lfrak\inter\bfrak\to\lfrak$ and $\widetilde{X}_L$ in the same way as for $G$, as well as $\widetilde{V}_{L,w}$ and $\widetilde{X}_{L,w}$ for any $w\in W_J$. These objects are always assumed to be defined for $G$ by default when $L$ is not written in subscript.

\subsubsection*{Preview of the proof}

Theorem \ref{thmConjIsTrue} can be stated as follows: for all good pairs $(w',w)$ in $W$ and closed points $t\in\tfrak^{ww'^{-1}}$, we have the inclusion
\begin{equation} \label{eqInclusionForFiber}
	\widetilde{V}_{w'}\inter q^{-1}(t+\ufrak) \inj \widetilde{X}_w \,.
\end{equation}
The proof of \eqref{eqInclusionForFiber} will proceed in four steps.

The first step is a result of \cite{bhs3} stating that \eqref{eqInclusionForFiber} is true when $t=0$, regardless of whether $(w',w)$ is a good pair or not.  We then apply this to a conjugate $L$ of the Levi subgroup $C_G((T^{ww'^{-1}})^\circ)$ of $G$, to get the inclusion $\widetilde{V}_{L,w'_L}\inter q_L^{-1}(\ufrak_L)\inj\widetilde{X}_{L,w_L}$ for any $w'_L,w_L$ in the Weyl group of $L$. It is of particular interest because we can then translate by any suitable closed point $t\in\tfrak^{ww'^{-1}}$ when such a $t$ is centralised by $L$; therefore we have, in fact, $\widetilde{V}_{L,w'_L}\inter q_L^{-1}(t+\ufrak_L)\inj\widetilde{X}_{L,w_L}$ for such a $t$.

The second step establishes a connection between $X_L$ and $X$. More precisely, we construct an isomorphism between $\widetilde{V}_{L,w_L}\inter q_L^{-1}(\ufrak_L)$ (or equivalently, by the previous paragraph, $\widetilde{V}_{L,w_L}\inter q_L^{-1}(t'+\ufrak_L)$ for any $t'$ centralised by $L$ and conjugate to some $t\in\tfrak^{ww'^{-1}}$) and a closed subvariety of $\widetilde{V}_w\inter q^{-1}(t+\ufrak)$ for some $w_L$ in the Weyl group of $L$ and $w\in W$. We momentarily call $\widetilde{V}'_w$ this subvariety of $\widetilde{V}_w\inter q^{-1}(t+\ufrak)$.

The third step bridges the last gap from $\widetilde{V}'_w$ to $\widetilde{V}_w\inter q^{-1}(t+\ufrak)$: an action of $U$ on $\widetilde{V}_w\inter q^{-1}(t+\ufrak)$ is defined such that all the orbits intersect $\widetilde{V}'_w$.

The fourth and last step articulates the previous steps into a proof: the good pair condition on some given $(w'\preceq w)$ allows us to find a standard Levi subgroup $L$ of $G$ and an ordered pair $(w'_L\preceq w_L)$ of elements in the Weyl group of $L$, such that the second and third steps can reduce the problem to the situation of the first.

\subsubsection*{Step 1}

\begin{thm}[Beilinson-Bernstein, Ginsburg] \label{thmInclusionFiber0}
Let $w,w'\in W$ such that $w'\preceq w$. Then the locally closed immersion $\widetilde{V}_{w'}\inj\widetilde{X}$ induces $\widetilde{V}_{w'}\inter q^{-1}(\ufrak) \inj \widetilde{X}_w$.
\end{thm}

\begin{proof}
It follows from the proof of Proposition \ref{conjImpliesConj2} that the theorem is equivalent to $V_{w'}\inter\kappa_1^{-1}(0)\inj X_w$ for all $w'\preceq w$. This statement is a consequence of the Beilinson-Bernstein correspondence (see \cite[\S11]{DModules}) and a result of Ginsburg \cite{ginsburg}; the whole argument is laid out in \cite[\S2.4]{bhs3} (see in particular the proof of \cite[Thm.\ 2.4.7]{bhs3}).
\end{proof}

\subsubsection*{Step 2}

\begin{prop} \label{propLeviToG}
Let $t\in\tfrak^\gen_{J}$ be a closed point, and let $u^J,v^J\in W^J$.
\begin{enumerate}[label=\emph{(\arabic*)},ref=(\arabic*)]
\item \label{itemLeviImmersion}
There is a morphism $\iota_{t,u^J,v^J}$ given by
\[
	\appl{\iota_{t,u^J,v^J}} {L\times^{L\inter B}\lfrak\inter\bfrak}{G\times^B\bfrak} {(l,\psi_L)} {\left(u^Jl(v^J)^{-1},\Ad(v^J)(t+\psi_L)\right)}
\]
and it is a closed immersion.
\item \label{itemLeviIsomorphism}
The closed immersion $\iota_{t,u^J,v^J}$ induces, for any $w_J\in W_J$, an isomorphism
\[
	\iota_{t,u^J,w_J,v^J} \colon \widetilde{V}_{L,w_J}\inter q_L^{-1}(\ufrak_L) \isoto \widetilde{V}_{u^Jw_J(v^J)^{-1}}\inter q^{-1}\left(\Ad(u^J)(t+\ufrak_L)\right) \,.
\]
\end{enumerate}
\end{prop}

\begin{proof}
We first prove statement \ref{itemLeviImmersion}. We begin by checking that $\iota_{t,u^J,v^J}$ is well defined. Let $b\in L\inter B$. Since $L\inter B=T\prod_{\alpha_\in\Phi_J\inter\Phi_+}U_\alpha$ and $v^J\in W^J$, we have
\[
	b'\coloneqq v^Jb(v^J)^{-1}\in B \,.
\]
Thus, for any $l\in L$ and $\psi_L\in\lfrak\inter\bfrak$, we have
\[
	\left(u^Jlb(v^J)^{-1},\Ad(v^J)(t+\Ad(b)\psi_L)\right) = \left(u^Jl(v^J)^{-1}b',\Ad(b')\Ad(v^J)(t+\psi_L)\right)
\]
(note that $b$ fixes $t$ as $b\in L$), which shows that $\iota_{t,u^J,v^J}$ is well defined.

Now we check that $\iota_{t,u^J,v^J}$ is a closed immersion. First consider the morphism
\begin{equation} \label{eqMapForImmersion}
	\begin{array}{ccc}
		L\times\lfrak\inter\bfrak &\to &(u^JL(v^J)^{-1}B)/B\times\lfrak \\
		(l,\psi_L) &\longmapsto &(u^Jl(v^J)^{-1}B,\Ad(l)\psi_L)
	\end{array} \,.
\end{equation}
From $u^Jlb(v^J)^{-1}=u^Jl(v^J)^{-1}b'$, we know that for any $b\in L\inter B$ and $(l,\psi_L)\in L\times\lfrak\inter\bfrak$, the images of $(l,\psi_L)$ and $(lb,\Ad(b^{-1})\psi_L)$ under \eqref{eqMapForImmersion} are the same. Conversely, assume that $(l,\psi_L)$ and $(l',\psi_L')$ have the same image. This means that
\[
	\Ad(l)\psi_L=\Ad(l')\psi_L'
\]
and that there is $b\in B$ such that
\[
	u^Jl(v^J)^{-1}=u^Jl'(v^J)^{-1}b \,.
\]
Write $b'=l'^{-1}l$. Then $v^Jb'(v^J)^{-1}=b$, so $b'\in L\inter B$ according to Lemma \ref{lemNoNegativeComponent}. We also have $l=l'b'$, so $\psi_L=\Ad({b'}^{-1})\psi_L'$ and $(l,\psi_L)=(l',\psi_L')$ in $L\times^{L\inter B}\lfrak\inter\bfrak$. Therefore, by definition of the quotient space $L\times^{L\inter B}\lfrak\inter\bfrak$, the morphism \eqref{eqMapForImmersion} factorises as
\[
	f_L \colon
	\begin{array}{ccc}
		L\times^{L\inter B}\lfrak\inter\bfrak &\to &\left(u^JL(v^J)^{-1}B\right)/B\times\lfrak \\
		\left(l,\psi_L\right) &\longmapsto &\left(u^Jl(v^J)^{-1}B,\Ad(l)\psi_L\right) \,.
	\end{array}
\]
Together with the morphism
\[
	f_G \colon
	\begin{array}{ccc}
		G\times^B\bfrak &\to &G/B\times\gfrak \\
		(g,\psi) &\longmapsto &\left(gB,\Ad(g)\psi\right) \,,
	\end{array}
\]
the morphism $\iota_{t,u^J,v^J}$ and the closed immersion
\[
	\iota \colon
	\begin{array}{ccc}
		(u^JL(v^J)^{-1}B)/B\times\lfrak &\inj &G/B\times\gfrak \\
		(gB,\psi_L) &\longmapsto &\left(gB,\Ad(u^J)(t+\psi_L)\right) ,
	\end{array}
\]
we get a commutative diagram
\begin{equation} \label{eqDiagramImmersion}
	\begin{tikzcd}
		L\times^{L\inter B}\lfrak\inter\bfrak \arrow[d,"f_L"] \arrow[r,"\iota_{t,u^J,v^J}"] & G\times^B\bfrak \arrow[d,"f_G"] \\
		(u^JL(v^J)^{-1}B)/B\times\lfrak \arrow[r,"\iota",hook] & G/B\times\gfrak \,.
	\end{tikzcd}
\end{equation}
This diagram is a cartesian square. Indeed, let $(g,\psi)\in G\times^B\bfrak$ be such that $(gB,\Ad(g)\psi)$ is in the image of $\iota$. Then, modifying $(g,\psi)\in G\times\bfrak$ if necessary (so that its image in $G\times^B\bfrak$ is the same), there is $l\in L$ such that $g=u^Jl(v^J)^{-1}$, and $\psi_L\in\lfrak$ such that $\Ad(g)\psi=\Ad(u^Jl(v^J)^{-1})\psi=\Ad(u^J)(t+\psi_L)$. From this last equality we get
\[
	\psi = \Ad(v^Jl^{-1})(t+\psi_L) = \Ad(v^J)(t+\Ad(l^{-1})\psi_L)
\]
with $\Ad(l^{-1})\psi_L\in\lfrak=\tfrak\oplus\bigoplus_{\alpha\in\Phi_J}\ufrak_\alpha$. Furthermore, $v^J\in W^J$ implies that
\[
	\Ad(v^J)\bigg(\bigoplus_{\alpha\in\Phi_J\inter\Phi_-}\ufrak_\alpha\bigg) \subseteq \bigoplus_{\alpha\in\Phi_-}\ufrak_\alpha \,,
\]
hence, as $\Ad(v^J)\Ad(l^{-1})\psi_L=\psi-\Ad(v^Jl^{-1})t\in\bfrak$, we necessarily have
\[
	\Ad(l^{-1})\psi_L \in \tfrak\oplus\bigoplus_{\alpha\in\Phi_J\inter\Phi_+}\ufrak_\alpha = \lfrak\inter\bfrak \,.
\]
This gives $(g,\psi)=\iota_{t,u^J,v^J}\big(l,\Ad(l^{-1})\psi_L\big)$ as we wanted. Since the diagram \eqref{eqDiagramImmersion} is cartesian and $\iota$ is a closed immersion, so is $\iota_{t,u^J,v^J}$ by base change.

As for statement \ref{itemLeviIsomorphism}, we had already established the inclusion $u^J(L\inter B)(u^J)^{-1}\subseteq B$ as a consequence of $u^J\in W^J$, and similarly for $v^J$. From this we easily see that $\iota_{t,u^J,v^J}$ restricts to a closed immersion of algebraic varieties
\[
	\widetilde{V}_{L,w_J}\inter q_L^{-1}(\ufrak_L) \inj \widetilde{V}_{u^Jw_J(v^J)^{-1}}\inter q^{-1}\left(\Ad(u^J)(t+\ufrak_L)\right)
\]
which we want to be an isomorphism. For this, let $(g,\psi)\in q^{-1}\left(\Ad(u^J)(t+\ufrak_L)\right)$. We know that
\[
	G = \coprod_{\substack{w^J\in W^J\\w_J\in W_J}}Bw^Jw_JB = \coprod_{w^J\in W^J}Bw^JP = \coprod_{w^J\in W^J}Bw^JU_PL \,,
\]
where the first equality comes from the Bruhat decomposition of $G$, the second from the Bruhat decomposition of $P=P_J$ and the third from the Levi decomposition of $P$ (see \cite[Thm.\ 21.91]{milne} for the last two). Hence, there are $b\in B$, $w^J\in W^J$, $u_P\in U_P$ and $l\in L$ such that $g^{-1}u^J=bw^Ju_Pl$. Changing the lift of $(g,\psi)$ in $G\times\bfrak$ if necessary, we can assume $b=1$, and by assumption there is $x\in\ufrak_L$ such that
\[
	(g,\psi) = \left(u^Jl^{-1}u_P^{-1}(w^J)^{-1},\Ad(w^Ju_Pl)(t+x)\right) \,.
\]
Since $\psi\in\bfrak$, Lemma \ref{joliCorollaire} implies that $w^Ju_P(w^J)^{-1}\in U\subset B$ and that
\[
	\psi_L\coloneqq\Ad(l)x\in\ufrak_L \,,
\]
so we actually have in $G\times^B\bfrak$
\begin{equation} \label{eqIsInImage}
	(g,\psi) = \left(u^Jl^{-1}(w^J)^{-1},\Ad(w^Jl)(t+x)\right) = \left(u^Jl^{-1}(w^J)^{-1},\Ad(w^J)(t+\psi_L)\right) \,.
\end{equation}
Finally, $l^{-1}\in (L\inter B)w_J'(L\inter B)$ for some $w_J'\in W_J$, which gives as usual $g=u^Jl^{-1}(w^J)^{-1}\in Bu^Jw_J'(w^J)^{-1}B$. If we add the additional condition that $(g,\psi)\in\widetilde{V}_{u^Jw_J(v^J)^{-1}}$, or equivalently that $g\in Bu^Jw_J(v^J)^{-1}B$, then we have $w_J'(w^J)^{-1}=w_J(v^J)^{-1}$. This forces $w_J'=w_J$ and $w^J=v^J$, thus \eqref{eqIsInImage} translates to
\[
	(g,\psi) = \iota_{t,u^J,v^J}(l^{-1},\psi_L)
\]
with $(l^{-1},\psi_L)\in\widetilde{V}_{L,w_J}$. This finishes the proof of Proposition \ref{propLeviToG}.
\end{proof}

\subsubsection*{Step 3}

The proof of the following lemma uses standard results about Jordan decompositions, which can be found in Appendix \ref{appendixJordan} below.

\begin{lem} \label{lemJordan}
Let $t\in\tfrak_J^\gen$ be a closed point and let $u^J\in W^J$. Then there is a surjective morphism given by
\begin{equation*}
	f \colon
	\begin{array}{ccc}
		U\times\left(\Ad(u^J)(t+\ufrak_L)\right) &\surj &\Ad(u^J)t+\ufrak \\
		(u,x) &\longmapsto &\Ad(u)x \,.
	\end{array}
\end{equation*}
\end{lem}

\begin{proof}
The existence of the map $f$ is a direct consequence of the fact that $\Ad(u)$ sends any $t'\in\tfrak$ to $t'+\ufrak$ when $u\in U$. Thus we only have to prove the surjectivity

Let $y\in\Ad(u^J)t+\ufrak$, and consider its Jordan decomposition $y=y_\semisimple+y_\nilp$ in $\gfrak$ (see Theorem \ref{thmJordanDecomposition}). From Proposition \ref{propSsDiagConjugates}, there is $g\in G$ such that
\[
	y_\semisimple = \Ad(g)\Ad(u^J)t \,.
\]
We use the decomposition
\[
	G = \coprod_{w^J\in W^J}Bw^JP = \coprod_{w^J\in W^J}UTw^JP = \coprod_{w^J\in W^J}Uw^JU_PL \,.
\]
Let $u\in U$, $w^J\in W^J$ and $u_P\in U_P$ such that $gu^J\in uw^Ju_PL$, from the fact that $t$ is centralised by $L$ we get:
\[
	y_\semisimple = \Ad(uw^Ju_P)t \,.
\]

Now, from Proposition \ref{propSsDiagConjugates}\ref{itemComponentsInBorel}, we have
\begin{equation} \label{eqYsemisimple}
	y_\nilp\in\ufrak, \quad y_\semisimple\in y+\ufrak\subseteq\bfrak \,.
\end{equation}
Hence, $\Ad(w^Ju_P)t=\Ad(u^{-1})y_\semisimple\in\bfrak$, so Lemma \ref{lem3Breuil} implies that $w^Ju_P(w^J)^{-1}\in U$. Set $u'\coloneqq uw^Ju_P(w^J)^{-1}\in U$. We can now write
\begin{equation} \label{eqYsemisimpleFormula}
	y_\semisimple = \Ad(u'w^J)t \in \Ad(w^J)t+\ufrak \,.
\end{equation}
Also, \eqref{eqYsemisimple} and the assumption on $y$ give $y_\semisimple\in \Ad(u^J)t+\ufrak$. Together with \eqref{eqYsemisimpleFormula}, we deduce that 
\[
	\Ad(w^J)t=\Ad(u^J)t \,,
\]
hence $(u^J)^{-1}w^J\in C_G(t)$. If $(u^J)^{-1}w^J\notin W_J$, then there exists $\alpha\in\Phi_J$ such that $\beta\coloneqq(u^J)^{-1}w^J(\alpha)\notin\Phi_J$. Since $U_\alpha\subseteq C_G(t)$, we have
\[
	U_\beta = (u^J)^{-1}w^JU_\alpha(w^J)^{-1}u^J \subseteq C_G(t)
\]
so $U_\beta\subseteq C_G(t)^\circ=L$ because $U_\beta$ is connected, which is a contradiction as $\beta$ is not a root of $L$. Hence $(u^J)^{-1}w^J\in W_J$, from which we deduce $w^J=u^J$ as $W^J$ is a set of representatives of $W/W_J$.

Moreover, the commutation $[y_\semisimple,y_\nilp]=0$ together with $y_\semisimple=\Ad(u'w^J)t$ (see \eqref{eqYsemisimpleFormula}) and Proposition \ref{propGenericTorus}\ref{itemGenericTorusLie} imply $y_\nilp\in\Ad(u'w^J)\lfrak$. Since $y_\nilp\in\ufrak$ and $w^J\in W^J$, we necessarily get $y_\nilp\in\Ad(u'w^J)\ufrak_L$. Thus,
\[
	y = y_\semisimple+y_\nilp \in \Ad(u'w^J)(t+\ufrak_L) = \Ad(u')\Ad(u^J)(t+\ufrak_L)
\]
which concludes the proof.
\end{proof}

\begin{prop} \label{propLastStepLevi}
Let $t\in\tfrak_J^\gen$ be a closed point, let $u^J,v^J\in W^J$ and let $w_J\in W_J$. Write $w\coloneqq u^Jw_J(v^J)^{-1}$. Then there is a surjective morphism of algebraic varieties
\[
	\mu_{t,u^J,w_J,v^J} \colon
	\begin{array}{ccc}
		U\times\left(\widetilde{V}_{L,w_J}\inter q_L^{-1}(\ufrak_L)\right) &\surj &\widetilde{V}_w\inter q^{-1}\left(\Ad(u^J)t+\ufrak\right) \\
		\left(u,(l,\psi_L)\right) &\longmapsto &\left(uu^Jl(v^J)^{-1}B,\Ad(v^J)(t+\psi_L)\right) .
	\end{array}
\]
\end{prop}

\begin{proof}
Remember that the surjectivity of a morphism of algebraic varieties can be checked on closed points (\cite[Exercise 7.4.C]{vakil}). The morphism of algebraic varieties
\begin{equation} \label{eqUnipotentGivesAll}
	\begin{array}{ccc}
		 U\times q^{-1}(\Ad(u^J)(t+\ufrak_L)) &\to &q^{-1}(\Ad(u^J)t+\ufrak) \\
		(u,(g,x)) &\longmapsto &(ug,x)
	\end{array}
\end{equation}
is surjective as a consequence of Lemma \ref{lemJordan}. Since $U\subseteq B$, for any $u\in U$ and $g\in G$, we have $g\in BwB$ if and only if $ug\in BwB$. Therefore \eqref{eqUnipotentGivesAll} induces a surjective morphism of algebraic varieties
\[
	\begin{array}{ccc}
		 U\times \left(\widetilde{V}_w\inter q^{-1}\left(\Ad(u^J)(t+\ufrak_L)\right)\right) &\to &\widetilde{V}_w\inter q^{-1}\left(\Ad(u^J)t+\ufrak\right) \\
		\left(u,(g,x)\right) &\longmapsto &(ug,x)
	\end{array} .
\]
Composing with the bijective morphism $\id\times\iota_{t,u^J,w_J,v^J}$ of Proposition \ref{propLeviToG}\ref{itemLeviIsomorphism} yields the desired morphism $\mu_{t,u^J,w_J,v^J}$. This concludes the proof.
\end{proof}

\subsubsection*{Step 4}

We now prove Theorem \ref{thmConjIsTrue}.

Let $\Phi$ be the root system of $G$ with basis $I$ and $\Phi_{ww'^{-1}}\subseteq\Phi$ the minimal generating subsystem of $ww'^{-1}$ in the sense of Definition \ref{dfnMinGenSys}. According to Proposition \ref{propCycleDecompositionCentraliser}, it is the root system of the Levi subgroup $C_G((T^{ww'^{-1}})^\circ)$. Proposition \ref{propGoodPairsMinipermutations} gives a subset $J$ of the basis $I$ of roots, and elements $u^J,v^J\in W^J$, $w'_J,w_J\in W$ satisfying
\[
	w = u^Jw_J(v^J)^{-1} \,, \quad w' = u^Jw'_J(v^J)^{-1} \,, \quad w'_J\preceq w_J
\]
and such that $u^JL(u^J)^{-1}=C_G((T^{ww'^{-1}})^\circ)$, where $L\coloneqq L_J$ is the standard Levi subgroup associated to $J$. Then $L=C_G((T^{w_J(w'_J)^{-1}})^\circ)$, with $(T^{w_J(w'_J)^{-1}})^\circ=T_J$ and $\tfrak^{w_J(w'_J)^{-1}}=\tfrak_J$ (using the notation of \S\ref{ssecLeviAdjoint}).

The variety $\tfrak^\gen_J$ is Zariski-dense in $\tfrak_J=\tfrak^{w_J(w'_J)^{-1}}$, therefore $\Ad(u^J)\tfrak^\gen_J$ is Zariski-dense in $\Ad(u^J)\tfrak_J=\tfrak^{ww'^{-1}}$. According to Lemma \ref{lemDenseIsEnough}, we thus only need to prove that $\widetilde{V}_{w'}\inj\widetilde{X}$ induces
\[
	\widetilde{V}_{w'}\inter q^{-1}(\Ad(u^J)t+\ufrak) \inj \widetilde{X}_w
\]
for all closed points $t\in\tfrak^\gen_J$. In the following, fix one such $t$.

The closed immersion $\iota_{t,u^J,v^J}$ of Proposition \ref{propLeviToG} restricts to a closed immersion
\[
	\iota_{t,u^J,v^J} \colon \widetilde{X}_L=q_L^{-1}(\lfrak\inter\bfrak) \inj \widetilde{X}=q^{-1}(\bfrak) \,.
\]
Indeed, one has to check that $\Ad(u^J)\Ad(l)\psi_L\in\bfrak$ for all $(l,\psi_L)\in L\times(\lfrak\inter\bfrak)$ such that $\Ad(l)\psi_L\in\lfrak\inter\bfrak$, which is true because $u^J\in W^J$ means that $\Ad(u^J)(\lfrak\inter\bfrak)\subseteq\bfrak$. In the same way, we can further restrict to
\[
	\iota_{t,u^J,v^J} \colon \widetilde{V}_{L,w_J} \inj \widetilde{V}_w
\]
because $u^J\in W^J$ means that $u^J(L\inter B)(u^J)^{-1}\subseteq B$ and similarly for $v^J$. Taking the Zariski closures gives a closed immersion $\iota_{t,u^J,v^J}\colon\widetilde{X}_{L,w_J}\inj\widetilde{X}_w$. There is also a morphism
\[
	\applo{U\times\widetilde{X}_w}{\widetilde{X}_w}{(u,(g,x))}{(ug,x)}
\]
which, when composed with $\id\times\iota_{t,u^J,v^J}\colon U\times\widetilde{X}_{L,w_J}\to U\times\widetilde{X}_w$, gives a morphism
\[
	\bar{\mu}_{t,u^J,w_J,v^J} \colon U\times\widetilde{X}_{L,w_J} \to \widetilde{X}_w
\]
which can be defined by a similar formula to the one for the surjective morphism $\mu_{t,u^J,w_J,v^J}$ of Proposition \ref{propLastStepLevi}.

We also have an obvious locally closed immersion
\[
	\varphi_G \colon \widetilde{V}_{w'}\inter q^{-1}(\Ad(u^J)t+\ufrak) \inj \widetilde{X}
\]
as well as a less obvious locally closed immersion
\[
	\varphi_L \colon \widetilde{V}_{L,w'_J}\inter q_L^{-1}(\ufrak_L) \inj \widetilde{X}_{L,w_J}
\]
given by theorem \ref{thmInclusionFiber0} applied to the reductive group $L$ together with the good pair condition $w'_J\preceq w_J$.

All these morphisms fit into a commutative diagram of algebraic varieties over $k$
\[
	\begin{tikzcd}
		U\times\left(\widetilde{V}_{L,w'_J}\inter q_L^{-1}(\ufrak_L)\right) \arrow[r,"\id\times\varphi_L"] \arrow[dd,"\mu_{t,u^J,w'_J,v^J}",two heads] & U\times\widetilde{X}_{L,w_J} \arrow[d,"\bar{\mu}_{t,u^J,w_J,v^J}"] \\
		 & \widetilde{X}_w \arrow[d,hook] \\
		\widetilde{V}_{w'}\inter q^{-1}(\Ad(u^J)t+\ufrak) \arrow[r,"\varphi_G"] & \widetilde{X}
	\end{tikzcd}
\]
where the surjectivity of the left map $\mu_{t,u^J,w'_J,v^J}$ (Proposition \ref{propLastStepLevi}) proves that the image of the bottom map $\varphi_G$ lands in $\widetilde{X}_w$. This finishes the proof of Theorem \ref{thmConjIsTrue}.

\section{Counter-examples for bad pairs when $G=\mathrm{GL}_n$} \label{secEquations}

In \S\ref{secProofGoodPairs}, we proved Conjecture \ref{conjLocalModel} in the case of good pairs. The present section explores the case of bad pairs for the split reductive group $G=\GL_{n/k}$ where $n\geq 1$ and $k$ is any field of characteristic $0$. From now on $G=\GL_{n/k}$, $B$ is the subgroup of upper triangular matrices in $G$, $T$ is the subgroup of diagonal matrices in $G$, $U$ is the subgroup of upper triangular matrices in $G$ whose diagonal coordinates are $1$, and $W=\Scal_n$ is the Weyl group of $G$.

\subsection{Grassmannians, the flag variety and Schubert cells} \label{ssecFlagEquations}

We recall a few well-know facts about the flag variety and its decomposition in Schubert cells from the computational point of view. We also set the notation for the rest of the section.

\begin{dfn}
For $1\leq d\leq n$, the \emph{Grassmannian} over $k$ of $d$-planes in the $n$-space is the set of linear subspaces of dimension $d$ in $k^n$. It is noted $\Gr_{d,n/k}$.
\end{dfn}

For $V\in\Gr_{d,n/k}$, let $(v_1,\ldots,v_d)$ be a basis of $V$. Then $v_1\wedge\ldots\wedge v_d$ is a non-zero element of $\bigwedge^d k^n$. Choosing a different basis would change this alternate product by a multiplicative factor (the determinant of one basis in the other), therefore the line spanned by $v_1\wedge\ldots\wedge v_d$ in $\bigwedge^d k^n$ is a well-defined element of $\Proj(\bigwedge^d k^n)$. One can also recover $V$ from this line by the property that a vector $v$ is in $V$ if and only if $v\wedge v_1\wedge\ldots\wedge v_d = 0$ in $\bigwedge^{d+1}k^n$.

\begin{dfn} \label{dfnPluckerEmbedding}
The \emph{Plücker embedding} is the map 
\[
	\applo{\Gr_{d,n/k}}{\Proj\Big(\bigwedge^d k^n\Big)}{\left\langle v_1,\ldots,v_d\right\rangle}{k\cdot(v_1\wedge\ldots\wedge v_d)}
\]
where $(v_1,\ldots,v_d)$ is any linearly independent family in $k^n$.
\end{dfn}

\begin{rem}
Proposition \ref{pluckerRelations} below together with the discussion above Definition \ref{dfnPluckerEmbedding} show that the Grassmannian is a projective variety and the Plücker embedding a closed immersion.
\end{rem}

We write any point $x\in\Proj(\bigwedge^d k^n)$ as
\[
	x=[x_{i_1\ldots i_d}]_{1\leq i_1<\ldots<i_d\leq n}
\]
using the projective coordinates along the standard basis $(e_{i_1}\wedge\ldots\wedge e_{i_d})_{1\leq i_1<\ldots<i_d\leq n}$ of $\bigwedge^d k^n$, where $(e_i)_{1\leq i\leq n}$ is the standard basis of $k^n$. It is sometimes convenient to use the variables $x_{i_1\ldots i_d}$ for any sequence $(i_1,\ldots,i_d)$ of integers between $1$ and $n$; they satisfy the relations $x_{i_{\sigma(1)}\ldots i_{\sigma(d)}}=\sgn(\sigma)x_{i_1\ldots i_d}$ for any $\sigma\in\Scal_d$ and $x_{i_1,\ldots,i_d}=0$ whenever there are $1\leq k<l\leq d$ such that $i_k=i_l$.

\begin{dfn}
The \emph{Plücker coordinates} of a point in $\Gr_{d,n/k}$ are the projective coordinates $[x_{i_1\ldots i_d}]_{1\leq i_1<\ldots<i_d\leq n}$ of its image under the Plücker embedding.
\end{dfn}

Let $(v_1,\ldots,v_d)$ be a linearly independent family in $k^n$, and for any $1\leq j\leq d$ write the coordinates of $v_j$ along the standard basis as $(v_{ij})_{1\leq i\leq n}$. Then the Plücker coordinate $x_{i_1\ldots i_d}$ of the span $\langle v_1,\ldots,v_d\rangle$ can be computed as the rank $d$ minor of the matrix $(v_{ij})_{1\leq i\leq n,1\leq j\leq n}$ along the rows $i_1$,\ldots,$i_d$ and columns $1,\ldots,d$.

\begin{prop} \label{pluckerRelations}
The Plücker coordinates satisfy the following quadratic relations, called the \emph{Plücker relations in dimension $d$}:
\begin{equation}
	\sum_{k=1}^{d+1} (-1)^k x_{i_1\ldots i_{d-1}j_k}x_{j_1\ldots\widehat{j_k}\ldots j_{d+1}} = 0
\end{equation}
where $(i_1,\ldots,i_{d-1})$ and $(j_1,\ldots,j_{d+1})$ are any increasing sequences of indices and the notation $\widehat{j}$ indicates that the index $j$ is omitted. These relations define precisely the points in $\Proj(\bigwedge^d k^n)$ that arise from the Grassmannian.
\end{prop}

\begin{proof}
See for example \cite[Thm.\ 1]{kleimanLaksov} or \cite[\S9.1 Lemma 1]{youngTableaux}.
\end{proof}

In the same vein as Proposition \ref{pluckerRelations}, one can express the property that a subspace of $k^n$ is included into another.

\begin{prop} \label{incidenceRelations}
Let $1\leq d<d'\leq n$ be two integers. The set of pairs $(V,V')\in\Gr_{d,n/k}\times\Gr_{d',n/k}$ such that $V\subset V'$ is identified with the set of points
\[
	\left([x_{i_1,\ldots,i_d}]_{1\leq i_1<\ldots<i_d\leq n},[x'_{i_1,\ldots,i_{d'}}]_{1\leq i_1<\ldots<i_{d'}\leq n}\right)\in\Proj\Big(\bigwedge^d k^n\Big)\times\Proj\Big(\bigwedge^{d'}k^n\Big)
\]
which satisfy both sets of Plücker relations and the following additional quadratic relations:
\begin{equation}
	\sum_{k=1}^{d'+1} (-1)^k x_{i_1\ldots i_{d-1}j_k}x'_{j_1\ldots\widehat{j_k}\ldots j_{d'+1}} = 0
\end{equation}
where $(i_1,\ldots,i_{d-1})$ and $(j_1,\ldots,j_{d'+1})$ are any increasing sequences of indices. These relations are called the \emph{incidence relations in dimensions $(d,d')$}.
\end{prop}

\begin{proof}
See \cite[\S9.1 Lemma 2]{youngTableaux}.
\end{proof}

\begin{dfn}
A \emph{flag} in $k^n$ is a strictly increasing sequence $(V_1,\ldots,V_{n-1})$ of proper subspaces of $k^n$. The \emph{flag variety over $k$} is the set of flags and it is noted $\flag_{n/k}$.
\end{dfn}

There is an obvious map $\flag_{n/k}\inj\prod_{d=1}^{n-1}\Gr_{d,n/k}$ which, when composed with the product of the Plücker embeddings, gives the embedding
\begin{equation} \label{eqFlagEmbedsInProj}
	\flag_{n/k} \inj \prod_{d=1}^{n-1}\Proj\Big(\bigwedge^d k^n\Big) \,.
\end{equation}
This realises the flag variety as the projective subvariety of $\prod_{d=1}^{n-1}\Proj(\bigwedge^d k^n)$ defined by the Plücker relations and the incidence relations (note that the Segre embedding \cite[\href{https://stacks.math.columbia.edu/tag/01WD}{Lemma 01WD}]{stacks} indeed makes $\prod_{d=1}^{n-1}\Proj(\bigwedge^d k^n)$ projective).

We identify $\flag_{n/k}$ with the algebraic group $G/B$ in the standard way. As in the first paragraph of \S\ref{ssecSpringerScheme}, we have the decomposition $G/B=\coprod_{w\in \Scal_n}BwB/B$. The following is well-known (see for example \cite[\S10.5 Exercise 11]{youngTableaux}):

\begin{prop} \label{propCellEquations}
The cell $BwB/B$ is the set of points in $G/B$ that satisfy the following equations on Plücker coordinates, for all $d\in\{1,\ldots,n-1\}$:
\begin{equation}
	x_{w(1)\ldots w(d)} \neq 0 \quad\text{and}\quad  x_{i_1\ldots i_d}=0 \quad\text{if}\quad \{i_1,\ldots,i_d\}\not\preceq\{w(1),\ldots,w(d)\}
\end{equation}
where $\preceq$ is the order defined in Definition \ref{dfnPartialOrder}.
\end{prop}

\subsection{Explicit equations for \texorpdfstring{$\widetilde{X}$}{X tilde} and \texorpdfstring{$\widetilde{V}_w$}{Vw tilde}} \label{ssecExplicitEquations}

We give a description of the algebraic varieties $\widetilde{X}$ and $\widetilde{V}_w$ of \S\ref{ssecMainResult} as locally closed subvarieties of a projective space defined by explicit equations.

Consider the morphism
\[
	\applo{G\times\bfrak}{G/B\times\gfrak}{(g,\psi)}{(gB,\Ad(g)\psi) .}
\]
It induces an injection $G\times^B\bfrak\inj G/B\times\gfrak$, which restricts to the isomorphism
\begin{equation} \label{defXtilde}
	\widetilde{X}\isoto\enstq{(gB,\psi)\in G/B\times\bfrak}{\Ad(g^{-1})\psi\in\bfrak}
\end{equation}
between $\widetilde{X}$ and a closed subvariety of $G/B\times\bfrak$. From now on, we use this isomorphism \eqref{defXtilde} as the definition of $\widetilde{X}$. This way, the maps $\widetilde{\pi}\colon\widetilde{X}\to G/B$ and $q\colon\widetilde{X}\to\bfrak$ also introduced in \S\ref{ssecMainResult}, are the projections on the first and second coordinate respectively. For $w\in\Scal_n$ we have the cell
\[
	\widetilde{V}_w = \enstq{(gB,\psi)\in\widetilde{X}}{g\in BwB/B}
\]
and its Zariski-closure $\widetilde{X}_w$, which give the cell decomposition $\widetilde{X}=\coprod_{w\in\Scal_n}\widetilde{V}_w$.

A point $(gB,\psi)\in \widetilde{X}$ will be identified by the projective coordinates $x_{i_1\ldots i_d}$ of the flag $gB$ seen inside $\prod_{d=1}^{n-1}\Proj(\bigwedge^d k^n)$ by \eqref{eqFlagEmbedsInProj}, and by the affine coordinates $(u_{ij})_{i\leq j}$ given by the matrix coefficients of $\psi$. We sometimes write $t_i$ instead of $u_{ii}$. We view $\psi$ as the matrix of an endomorphism $u$ of $k^n$ written relatively to the standard basis ($u$ is upper triangular). We view $g$ as a flag $(V_1,\ldots,V_{n-1})$, so that $V_d$ is the span of the first $d$ columns of $g$. 

\begin{lem} \label{lemEqStableFlagIso}
Suppose that $u$ is an isomorphism. Then $\Ad(g^{-1})\psi\in\bfrak$ if and only if, for each $1\leq d\leq n-1$, the point $x_d=\left(x_{i_1\ldots i_d}\right)_{i_1<\ldots<i_d} \in k^{\binom{n}{d}}$ is colinear to
\[
	u(x_d) \coloneqq \left(\sum_{j_1<\ldots<j_d}\Delta^{j_1\ldots j_d}_{i_1\ldots i_d}(u)x_{j_1\ldots j_d}\right)_{i_1<\ldots<i_d} \in k^{\binom{n}{d}}
\]
where $\Delta^{j_1\ldots j_d}_{i_1\ldots i_d}(u)$ is the minor of $u$ associated to rows $i_1,\ldots,i_d$ and columns $j_1,\ldots,j_d$. This amounts to $\binom{n}{d}\left(\binom{n}{d}-1\right)/2$ equations.
\end{lem}

\begin{proof}
$\Ad(g^{-1})\psi$ is the matrix of $u$ relative to the basis given by the columns of $g$. It is upper triangular if and only if each $V_d$ is stable by $u$. Since $u$ is an isomorphism, $u(V_d)$ is of dimension $d$ and we can consider its Plücker coordinates: we specifically want that $u(V_d)$ and $V_d$ have the same image in $\bigwedge^d k^n$ under the Plücker embedding. The image of $V_d$ is the line spanned by the vector
\[
	x_d=\sum_{i_1<\ldots<i_d}x_{i_1\ldots i_d}e_{i_1}\wedge\ldots\wedge e_{i_d} \quad \in \bigwedge^d k^n
\]
and a quick calculation shows that the image of $u(V_d)$ is the line spanned by
\[
	u(x_d) = \sum_{i_1<\ldots<i_d} \left(\sum_{j_1<\ldots<j_d}\Delta^{j_1\ldots j_d}_{i_1\ldots i_d}(u)x_{j_1\ldots j_d}\right) e_{i_1}\wedge\ldots\wedge e_{i_d} \quad \in \bigwedge^d k^n \,. \qedhere
\]
\end{proof}

\begin{rem}
Note that, since $u$ is upper triangular, we have $\Delta^{j_1\ldots j_d}_{i_1\ldots i_d}(u)=0$ unless $\{i_1,\ldots,i_d\}\preceq\{j_1,\ldots,j_d\}$ for the partial order of Definition \ref{dfnPartialOrder}. Also note that $\Delta^{i_1\ldots i_d}_{i_1\ldots i_d}(u)=t_{i_1}\ldots t_{i_d}$.
\end{rem}

\begin{lem} \label{lemEqStableFlag}
Let $A$ be the polynomial ring in the coordinates of $(gB/B,\psi)$, that is $A=\Z[x_{i_1\ldots i_d},u_{kl},t_m]_{1\leq d<n,\:1\leq i_1<\ldots<i_d\leq n,\:1\leq k<l\leq n,\:1\leq m\leq n}$. Let $\lambda$ be an indeterminate variable.
\begin{enumerate}[label=\emph{(\arabic*)},ref=(\arabic*)]
\item \label{itemColinearEquations}
Keep the notation of Lemma \ref{lemEqStableFlagIso}. The condition
\begin{multline} \label{eqColinearity}
	x_d=\left(x_{i_1\ldots i_d}\right)_{i_1<\ldots<i_d} \quad\text{is colinear to} \\
	(u+\lambda\id)(x_d)=\left(\sum_{j_1<\ldots<j_d}\Delta^{j_1\ldots j_d}_{i_1\ldots i_d}(u+\lambda\id)x_{j_1\ldots j_d}\right)_{i_1<\ldots<i_d}
\end{multline}
translates into $d \binom{n}{d}\left(\binom{n}{d}-1\right)/2$ equations in $A$.
\item \label{itemXtildeEquations}
The collection of all these equations in $A$ for $1\leq d\leq n-1$ defines $\widetilde{X}$ inside $G/B\times\bfrak$.
\end{enumerate}
\end{lem}

\begin{proof}
For all $i_1<\ldots<i_d$ and $j_1<\ldots<j_d$, the degree in $\lambda$ of $\Delta^{j_1\ldots j_d}_{i_1\ldots i_d}(u+\lambda\id)$ is at most $d-1$, unless $(j_1,\ldots,j_d)=(i_1,\ldots,i_d)$ where the degree is $d$ with leading term $\lambda^d$. Hence the degree of $\sum_{j_1<\ldots<j_d}\Delta^{j_1\ldots j_d}_{i_1\ldots i_d}(u+\lambda\id)x_{j_1\ldots j_d}$ in $\lambda$ is $d$, with leading term $x_{i_1\ldots i_d}\lambda^d$. Therefore the $\binom{n}{d}\left(\binom{n}{d}-1\right)/2$ ``cross-product'' equations in $A[\lambda]$ that express the condition \eqref{eqColinearity} are of degree $d-1$ in $\lambda$. This proves \ref{itemColinearEquations}.

As for \ref{itemXtildeEquations}, the equations defining $\widetilde{X}$ are the ones that express the condition $\Ad(g^{-1})\psi\in\bfrak$. They are equivalent to saying that
\[
	\Ad(g^{-1})(\psi+\lambda\id)\in\bfrak
\]
for all $\lambda\in k$. Enlarging $k$ if necessary, we can assume that it is algebraically closed, in particular infinite. Therefore there is a $\lambda\in k$ such that $\psi+\lambda\id$ is invertible. Applying Lemma \ref{lemEqStableFlagIso}, our condition is equivalent to \eqref{eqColinearity} for any such $\lambda$. Note that \eqref{eqColinearity} still holds true for any value of $\lambda$. Indeed, if $(u+\lambda)(V_d)$ is of dimension $<d$, then $(u+\lambda)(\sum x_{i_1\ldots i_d}e_{i_1}\wedge\ldots\wedge e_{i_d})=0$, \emph{i.e.\ }the Plücker coordinates of $(u+\lambda)(V_d)$ are zero. In particular their vector is colinear to the vector of the Plücker coordinates of $V_d$.

In the end, $\Ad(g^{-1})\psi\in\bfrak$ if and only if \eqref{eqColinearity} is true for all $\lambda\in k$. As $k$ is algebraically closed, the corresponding polynomials in $A[\lambda]$ are identically zero.
\end{proof}

\begin{prop} \label{allEquations}
Let $w\in\Scal_n$. Fix an integer $1\leq d\leq n-1$, let $A$ be the ring $\Z[x_{i_1\ldots i_d},u_{kl},t_m]_{1\leq i_1<\ldots<i_d\leq n,\:1\leq k<l\leq n,\:1\leq m\leq n}$, let $\lambda$ be an indeterminate variable. For a sequence $1\leq i_1<\ldots<i_d\leq n$, let $P_{w,i_1\ldots i_d}\in A[\lambda]$ be the polynomial
\[
	P_{w,i_1\ldots i_d}(\lambda) \coloneqq \sum_{j_1<\ldots<j_d}\Delta^{j_1\ldots j_d}_{i_1\ldots i_d}(u+\lambda\id)x_{j_1\ldots j_d} - \left(\prod_{k=1}^d(t_{w(k)}+\lambda)\right)x_{i_1\ldots i_d}
\]
and for $0\leq s\leq d-1$, let $P_{w,i_1\ldots i_d,s}\in A$ be the unique homogeneous polynomials such that
\[
	P_{w,i_1\ldots i_d}(\lambda) = \sum_{s=0}^{d-1} P_{w,i_1\ldots i_d,s}\lambda^s \,.
\]
Then, using the Plücker coordinates $(x_{i_1\ldots i_d})_{1\leq d<n,\:1\leq i_1<\ldots<i_d\leq n}$ on $G/B$ and the affine coordinates $(u_{kl},t_m)_{1\leq k<l\leq n,\:1\leq m\leq n}$ on $\bfrak=\ufrak\oplus\tfrak$, the subvariety $\widetilde{V}_w$ of $G/B\times \bfrak$ is defined by:
\begin{itemize}[--]
\item
the Plücker relations in dimension $d$, given in Proposition \ref{pluckerRelations}, for each $1\leq d\leq n-1$
\item
the incidence relations in dimension $(d,d+1)$, given in Proposition \ref{incidenceRelations}, for each $1\leq d\leq n-2$
\item
the equations for $(x_{i_1\ldots i_d})_{1\leq d<n,\:1\leq i_1<\ldots<i_d\leq n}$ to be in the cell $BwB/B$ of $G/B$, given in Proposition \ref{propCellEquations}
\item
the equations $P_{w,i_1\ldots i_d,s}=0$ for all $1\leq d\leq n-1$, $1\leq i_1<\ldots<i_d\leq n$ and $0\leq s\leq d-1$.
\end{itemize}
\end{prop}

\begin{proof}
Lemma \ref{lemEqStableFlag} gives a desciption of $\widetilde{X}$ as a subvariety of $G/B\times\bfrak$ defined by $\sum_{d=1}^{n-1} d\binom{n}{d}\left(\binom{n}{d}-1\right)/2$ equations (which are not necessarily independent). To find the equations for $\widetilde{V}_w$, we add the ones given by Proposition \ref{propCellEquations}.

Now, $x_{j_1\ldots j_d}=0$ for $\{j_1,\ldots,j_d\}\not\preceq\{w(1),\ldots,w(d)\}$ and $\Delta^{j_1\ldots j_d}_{w(1)\ldots w(d)}(u+\lambda\id)=0$ for $\{j_1,\ldots,j_d\}\not\succeq\{w(1),\ldots,w(d)\}$. Hence the $(w(1),\ldots,w(d))$-th term of the vector $(u+\lambda\id)(x)$, as expressed in \eqref{eqColinearity}, is
\[
	\left(\prod_{k=1}^d(t_{w(k)}+\lambda)\right)x_{w(1)\ldots w(d)} \,.
\]
Knowing that $x_{w(1)\ldots w(d)}\neq0$ and that $x_{j_1\ldots j_d}=0$ for $\{j_1,\ldots,j_d\}\not\preceq\{w(1),\ldots,w(d)\}$, the colinearity condition \eqref{eqColinearity} is reduced to the following equations:
\[
	\left(\prod_{k=1}^d(t_{w(k)}+\lambda)\right)x_{i_1\ldots i_d} = \sum_{j_1<\ldots<j_d}\Delta^{j_1\ldots j_d}_{i_1\ldots i_d}(u+\lambda\id)x_{j_1\ldots j_d}
\]
where $i_1<\ldots<i_d$ is any strictly increasing sequence of indices. These are the announced equations.
\end{proof}

\begin{rem} \label{remSimplifiedEquations}
We have also seen that for each $w\in\Scal_n$, $1\leq d\leq n-1$ and $1\leq i_1<\ldots<i_d\leq n$, the polynomial $P_{w,i_1\ldots i_d}$ is a sum of terms
\[
	a_{j_1\ldots j_d}x_{j_1\ldots j_d}, \quad \{j_1,\ldots,j_d\}\succeq\{i_1,\ldots,i_d\} \,,
\]
where $a_{j_1\ldots j_d}$ is an element of $\Z[u_{kl},t_m,\lambda]_{1\leq k<l\leq n, 1\leq m\leq n}$. Also, the particular ``first term coefficient'' $a_{i_1\ldots i_d}$ is equal to $\prod_{k=1}^d(t_{i_k}+\lambda)-\prod_{k=1}^d(t_{w(k)}+\lambda)$. In particular, for $\{i_1,\ldots,i_d\}\succeq\{w(1),\ldots,w(d)\}$ the polynomial $P_{w,i_1\ldots i_d}$ is in the ideal generated by the $x_{j_1\ldots j_d}$ for $\{j_1,\ldots,j_d\}\succ\{w(1),\ldots,w(d)\}$, which are all zero on the cell $BwB/B$ of $G/B$ by Proposition \ref{propCellEquations}. Thus the associated equation is tautological on $\widetilde{V}_w$.
\end{rem}

\subsection{Taking the Zariski closure} \label{ssecZariskiClosureEquations}

For $w\in\Scal_n$, let $\widetilde{X}'_w$ denote the subvariety of $\prod_{d=1}^{n-1}\Proj(\bigwedge^d k^n)\times\mathbb{A}^{n(n+1)/2}$, with coordinate ring $A\coloneqq\Z[x_{j_1\ldots j_d},u_{kl},t_m]$, defined by all the equations listed in Proposition \ref{allEquations} except for the inequations $x_{w(1)\ldots w(d)}\neq0$ for $1\leq d\leq n-1$. It is therefore the subvariety of $G/B\times\bfrak$ that satisfies the equations $P_{w,i_1\ldots i_d,s}=0$ for all $1\leq d\leq n-1$, $1\leq i_1<\ldots<i_d\leq n$ and $0\leq s\leq d-1$. In particular, note that $\widetilde{X}'_w$ contains $\widetilde{V}_w$ by Proposition \ref{allEquations} and is Zariski-closed. Therefore $\widetilde{X}_w\subseteq \widetilde{X}'_w$.

\begin{prop} \label{conjImpliesConj1}
Let $w,w'\in\Scal_n$ such that $w'\preceq w$. Then, as algebraic varieties over $k$,
\[
	\widetilde{X}'_w\inter \widetilde{V}_{w'} = \widetilde{V}_{w'}\inter q^{-1}(\tfrak^{ww'^{-1}}+\ufrak) \,.
\]
\end{prop}

\begin{proof}
Let $(gB/B,\psi)\in G/B\times\bfrak$ be a point in the intersection $\widetilde{X}'_w\inter \widetilde{V}_{w'}$. Its coordinates $(x_{j_1\ldots j_d},u_{kl},t_m)$ satisfy in particular $P_{w,i_1\ldots i_d}(\lambda)=0$ in $k[\lambda]$ for any $\{i_1,\ldots,i_d\}$ and $x_{i_1,\ldots,i_d}=0$ if $\{i_1,\ldots,i_d\}\not\preceq\{w'(1),\ldots,w'(d)\}$. Under these latter conditions, almost all the terms in $P_{w,w'(1)\ldots w'(d)}(\lambda)$ become zero, and the equation $P_{w,w'(1)\ldots w'(d)}(\lambda)=0$ becomes
\[
	\left(\prod_{k=1}^d(t_{w'(k)}+\lambda)-\prod_{k=1}^d(t_{w(k)}+\lambda)\right)x_{w'(1)\ldots w'(d)} = 0
\]
on $\widetilde{X}'_w$. Since we also have $x_{w'(1)\ldots w'(d)}\neq 0$, we then must have $\prod_{k=1}^d(t_{w'(k)}+\lambda)=\prod_{k=1}^d(t_{w(k)}+\lambda)$ as polynomials in $k[\lambda]$.

Also, we have by definition, for all $\{i_1,\ldots,i_d\}$,
\[
	P_{w,i_1\ldots i_d}(\lambda)-P_{w',i_1\ldots i_d}(\lambda) = \prod_{k=1}^d(t_{w(k)}+\lambda)-\prod_{k=1}^d(t_{w'(k)}+\lambda)
\]
in $A[\lambda]$. As a consequence, $P_{w,i_1\ldots i_d,s}$ and $P_{w',i_1\ldots i_d,s}$ are equal for all $s$ and $\{i_1,\ldots,i_d\}$ on $\widetilde{X}'_w$. Therefore $\widetilde{X}'_w\inter \widetilde{V}_{w'}$ is the subvariety of $\widetilde{V}_{w'}$ where $\prod_{k=1}^d(t_{w'(k)}+\lambda)=\prod_{k=1}^d(t_{w(k)}+\lambda)$ for all $1\leq d\leq n-1$. This last equality of polynomials in $k[\lambda]$ means that the multisets $\{t_{w'(1)},\ldots,t_{w'(d)}\}$ and $\{t_{w(1)},\ldots,t_{w(d)}\}$ are equal for all $1\leq d\leq n-1$. A recursion on $d$ shows that this amounts to saying $t_{w'(k)}=t_{w(k)}$ for all $1\leq k\leq d$. For $d=n-1$, this is equivalent to $(gB/B,\psi)\in q^{-1}(\tfrak^{ww'^{-1}}+\ufrak)$.
\end{proof}

Replacing $\widetilde{X}'_w$ by $\widetilde{X}_w$ in Proposition \ref{conjImpliesConj1} would yield Conjecture \ref{conjLocalModelSimple}. However we shall see in Theorem \ref{thmConjIsFalse} that this is not always true.

The strategy to find counterexamples to Conjecture \ref{conjLocalModelSimple} is to use the inequations $x_{w(1)\ldots w(d)}\neq0$ satisfied on $\widetilde{V}_w$ to simplify other equations satisfied on $\widetilde{V}_w$, by ``removing'' non-zero factors. We then get an equation that is satisfied on $\widetilde{V}_w$, hence on $\widetilde{X}_w$ as well, but not necessarily on $\tilde{X}'_w$. This is the object of the following two lemmas. Once this is done, this equation, when considered on $\widetilde{V}_{w'}$, gives a condition that turns out to be strictly more restrictive that simply being in $q^{-1}(\tfrak^{ww'^{-1}}+\ufrak)$.

\begin{lem} \label{lemIncidenceSimplified}
Let $w\in\Scal_n$, $q\in\{1,\ldots,n-2\}$ and $b\in\{w(1),\ldots,w(q+1)\}$; write $\{w(1),\ldots,w(q+1)\}=\{i_1,\ldots,i_q,b\}$. Let $\{j_1,\ldots,j_q\}$ be a subset of $\{1,\ldots,n\}$ which does not contain $b$. Assume that for any $1\leq l\leq q$ such that $j_l<b$, we have $j_l\in\{w(1),\ldots,w(q+1)\}$. Then, for any $a\in\{1,\ldots,n\}$, the equation
\[
	x_{i_1\ldots i_qa}x_{j_1\ldots j_qb} = \pm x_{i_1\ldots i_qb}x_{j_1\ldots j_qa}
\]
is satisfied on $\widetilde{V}_w$ for a certain choice of sign.
\end{lem}

\begin{proof}
Consider the Plücker relation in dimension $q+1$ with indices $\{i_1,\ldots,i_q\}$ and $\{j_1,\ldots,j_q,a,b\}$ (see Proposition \ref{incidenceRelations}), which is satisfied on $\widetilde{V}_w$. Up to signs, it is of the form
\begin{equation} \label{eqIncidenceToSimplify}
	\pm x_{i_1\ldots i_qa}x_{j_1\ldots j_qb} \pm x_{i_1\ldots i_qb}x_{j_1\ldots j_qa} + \sum_{1\leq l\leq q} \pm x_{i_1\ldots i_qk_l}x_{j_1\ldots\hat{j_l}\ldots j_qab} = 0 \,.
\end{equation}
Let $1\leq l\leq q$. If $j_l<b$, then our hypothesis implies that $x_{i_1\ldots i_qj_l}=0$ because the indices are redundant. Else we have $j_l>b$, so that $\{i_1,\ldots,i_q,j_l\}\succ\{i_1,\ldots,i_q,b\}=\{w(1),\ldots,w(q+1)\}$, hence $x_{i_1\ldots i_qj_l}=0$ on $\widetilde{V}_w$ by virtue of Proposition \ref{propCellEquations}. Eliminating all those null variables in \eqref{eqIncidenceToSimplify} gives the desired equation.
\end{proof}

\begin{lem} \label{lemAdditionalEquation}
Keep the notation and assumption of Lemma \ref{lemIncidenceSimplified}. Assume further that $\{j_1,\ldots,j_q,a\}\preceq\{i_1,\ldots,i_q,b\}$. Then the equation
\begin{equation}
	\sum_{\{j_1,\ldots,j_q\}\preceq\{k_1,\ldots,k_q\}\preceq\{i_1,\ldots,i_q\}}\pm\Delta^{k_1\ldots k_qb}_{j_1\ldots j_qb}(u+\lambda\id)x_{k_1\ldots k_qa} - \left(\prod_{i=1}^{p+1}(t_{w(i)}+\lambda)\right)x_{j_1\ldots j_qa} = 0
\end{equation}
is satisfied on $\widetilde{V}_w$ for a certain choice of sign.
\end{lem}

\begin{proof}
We can assume that the sequences $(i_l)_{1\leq l\leq q}$ and $(j_l)_{1\leq l\leq q}$ are in increasing order. Let $1\leq r\leq q$ be such that $j_r<b<j_{r+1}$ (setting $j_{q+1}=\infty$ if necessary). Then $\{j_1,\ldots,j_q,a\}$ has either $r$ or $r+1$ elements smaller than or equal to $b$, depending on wether $a>b$ or $a\leq b$. Because $\{j_1,\ldots,j_q,a\}\preceq\{i_1,\ldots,i_q,b\}$, this implies that $\{i_1,\ldots,i_q,b\}$ has at most $r+1$ elements smaller than or equal to $b$. However the hypothesis of Lemma \ref{lemIncidenceSimplified} implies that $\{j_1,\ldots,j_r,b\}\subseteq\{i_1,\ldots,i_q,b\}$. This forces $j_1$,\ldots,$j_r$ to be the $r$ smallest elements of $\{i_1,\ldots,i_q\}$, hence $i_l=j_l$ for all $l\leq r$.

In particular, we have $\{j_1,\ldots,j_q,b\}\preceq\{i_1,\ldots,i_q,b\}$ and any set of integers that sits between those two is of the form $\{k_1,\ldots,k_q,b\}$ with $k_l=i_l$ for $l\leq r$ and $k_l>b$ for $l>r$. We can therefore apply Lemma \ref{lemIncidenceSimplified} replacing $\{j_1,\ldots,j_q\}$ by $\{k_1,\ldots,k_q\}$; this yields, up to sign,
\begin{equation} \label{eqIncidenceSimplified}
	x_{i_1\ldots i_qa}x_{k_1\ldots k_qb} = \pm x_{i_1\ldots i_qb}x_{k_1\ldots k_qa} \,.
\end{equation}

Now, keeping the notation of Proposition \ref{allEquations}, we multiply by $x_{i_1\ldots i_qa}$ the equation $P_{w,j_1\ldots j_qb}(\lambda)=0$ (in $A[\lambda]$). From the discussion of Remark \ref{remSimplifiedEquations}, the resulting equation only has terms in $x_{k_1\ldots k_qb}x_{i_1\ldots i_qa}$ for $\{j_1,\ldots,j_q\}\preceq\{k_1,\ldots,k_q\}$. Also, because of Proposition \ref{propCellEquations}, the $x_{k_1\ldots k_qb}$ are non-zero on $\widetilde{V}_w$ only for $\{k_1,\ldots,k_q\}\preceq\{i_1,\ldots,i_q\}$. With all these simplifications, we get on $\widetilde{V}_w$ the equation
\begin{multline*}
	\sum_{\{j_1,\ldots,j_q\}\preceq\{k_1,\ldots,k_q\}\preceq\{i_1,\ldots,i_q\}}\Delta^{k_1\ldots k_qb}_{j_1\ldots j_qb}(u+\lambda\id)x_{k_1\ldots k_qb}x_{i_1\ldots i_qa} \\
	- \left(\prod_{i=1}^{q+1}(t_{w(i)}+\lambda)\right)x_{j_1\ldots j_qb}x_{i_1\ldots i_qa} = 0 \,.
\end{multline*}
Using the equations \eqref{eqIncidenceSimplified}, we get
\begin{multline*}
	\sum_{\{j_1,\ldots,j_q\}\preceq\{k_1,\ldots,k_q\}\preceq\{i_1,\ldots,i_q\}}\pm\Delta^{k_1\ldots k_qb}_{j_1\ldots j_qb}(u+\lambda\id)x_{k_1\ldots k_qa}x_{i_1\ldots i_qb} \\
	- \left(\prod_{i=1}^{q+1}(t_{w(i)}+\lambda)\right)x_{j_1\ldots j_qa}x_{i_1\ldots i_qb} = 0
\end{multline*}
which, after factorisation and simplification by $x_{i_1\ldots i_qb}$ (the simplification is valid because $x_{i_1\ldots i_qb}=x_{w(1)\ldots w(q+1)}\neq 0$ on $\widetilde{V}_w$), gives the equation we wanted.
\end{proof}

\begin{prop} \label{propAdditionalEq}
Let $w, w'\in\Scal_n$ with $w'\preceq w$. Assume that there exist integers $a<b$ and $q\in\{1,\ldots,n-2\}$ such that:
\begin{itemize}[--]
\item
$a\notin\{w(1),\ldots,w(q+1)\}$ and $a\in\{w'(1),\ldots,w'(q+1)\}$,
\item
$b\in\{w(1),\ldots,w(q+1)\}$ and $b\notin\{w'(1),\ldots,w'(q+1)\}$,
\item
for any $i\in\{w'(1),\ldots,w'(q+1)\}$ such that $i\neq a $ and $i<b$, we have $i\in\{w(1),\ldots,w(q+1)\}$.
\end{itemize}

Then the equation $t_a=t_b$ is satisfied on $\widetilde{X}_w\inter \widetilde{V}_{w'}$.
\end{prop}

\begin{ex}
Take $n=4$, $w=[4231]$ and $w'=[1324]$. Then $q=1$ and $(a,b)=(1,2)$ meets the condition of Proposition \ref{propAdditionalEq}, hence $\widetilde{X}_w\inter \widetilde{V}_{w'}$ satisfies $t_1=t_2$. This is stronger than for $\widetilde{X}'_w\inter \widetilde{V}_{w'}$, which by Proposition \ref{conjImpliesConj1} only satisfies $t_1=t_4$ and $t_2=t_3$ (the equations defining $\tfrak^{ww'^{-1}}$). One could also take $(a,b)=(3,4)$; this gives the equation $t_3=t_4$ on $\widetilde{X}_w\inter \widetilde{V}_{w'}$, which is equivalent to the already found $t_1=t_2$, since $t_1=t_4$ and $t_2=t_3$ are true.
\end{ex}

\begin{proof}
Write $\{j_1,\ldots,j_q\}=\{w'(1),\ldots,w'(q+1)\}\setminus\{a\}$; then the hypotheses of Lemma \ref{lemIncidenceSimplified} are satisfied. From Example \ref{exBruhatLineNotation} applied to $w'\preceq w$, we also have $\{j_1,\ldots,j_q,a\}\preceq\{i_1,\ldots,i_q,b\}$. We can therefore apply Lemma \ref{lemAdditionalEquation} to get
\[
	\sum_{\{j_1,\ldots,j_q\}\preceq\{k_1,\ldots,k_q\}\preceq\{i_1,\ldots,i_q\}}\pm\Delta^{k_1\ldots k_qb}_{j_1\ldots j_qb}(u+\lambda\id)x_{k_1\ldots k_qa} - \left(\prod_{i=1}^{p+1}(t_{w(i)}+\lambda)\right)x_{j_1\ldots j_qa} = 0
\]
which is satisfied on $\widetilde{V}_w$, hence on the Zariski closure $\widetilde{X}_w$ as well.

Now consider this equation on $\widetilde{X}_w\inter \widetilde{V}_{w'}$. From Proposition \ref{propCellEquations}, all variables $x_{k_1\ldots k_qa}$ are zero on $\widetilde{V}_{w'}$ for $\{k_1,\ldots,k_q\}\succ\{j_1,\ldots,j_q\}$, so it reduces to one term:
\[
	\left( \Delta^{j_1\ldots j_qb}_{j_1\ldots j_qb}(u+\lambda\id) - \prod_{i=1}^{p+1}(t_{w(i)}+\lambda) \right) x_{j_1\ldots j_qa} = 0 \,.
\]
Again on $\widetilde{V}_{w'}$ we can simplify by $x_{j_1\ldots j_qa}=x_{w'(1)\ldots w'(q+1)}\neq 0$. Moreover, $u$ being upper triangular gives $\Delta^{j_1\ldots j_qb}_{j_1\ldots j_qb}(u+\lambda\id)=(t_b+\lambda)\prod_{l=1}^p(t_{j_l}+\lambda)=\frac{t_b+\lambda}{t_a+\lambda}\prod_{i=1}^{p+1}(t_{w'(i)}+\lambda)$, so the equation becomes
\[
	\frac{t_b+\lambda}{t_a+\lambda}\prod_{i=1}^{p+1}(t_{w'(i)}+\lambda) = \prod_{i=1}^{p+1}(t_{w(i)}+\lambda) \,.
\]

We also know that $\prod_{i=1}^{p+1}(t_{w'(i)}+\lambda)=\prod_{i=1}^{p+1}(t_{w(i)}+\lambda)$ from the proof of Proposition \ref{conjImpliesConj1}; so we finally get $t_a+\lambda=t_b+\lambda$, which implies $t_a=t_b$.
\end{proof}

\begin{rem} \label{remVariationAdditionalEq}
The same result holds when the last condition on $a,b$ in Proposition \ref{propAdditionalEq} is replaced by the following variant: for any $j\in\{w(1),\ldots,w(q+1)\}\setminus\{b\}$ such that $j>a$, we have $j\in\{w'(1),\ldots,w'(q+1)\}$. The proof carries out in almost identical fashion, with the difference that, in the proof of Lemma \ref{lemIncidenceSimplified}, instead of using the Plücker relation with indices $\{i_1,\ldots,i_q\}$ and $\{k_1,\ldots,k_q,a,b\}$ for equation \eqref{eqIncidenceToSimplify}, we use the Plücker relation with indices $\{k_1,\ldots,k_q\}$ and $\{i_1,\ldots,i_q,a,b\}$.
\end{rem}

\begin{thm} \label{thmConjIsFalse}
Let $w,w'\in\Scal$ with $w'\preceq w$. Assume that there exist $a$, $b$ as in Proposition \ref{propAdditionalEq} such that $a$ and $b$ are not in the same orbit under $ww'^{-1}$. Then Conjecture \ref{conjLocalModel} is false for the pair $(w,w')$, that is:
\[
	\widetilde{V}_{w'}\inter q^{-1}(\tfrak^{ww'^{-1}}+\ufrak) \not\subseteq \widetilde{X}_w \,.
\]
\end{thm}

\begin{proof}
Consider any point $x=(x_{i_1\ldots i_d},t_i,u_{ij})\in\prod_{d=1}^{n-1}\Proj(\bigwedge^d k^n)\times\mathbb{A}^{n(n+1)/2}$ where, for $1\leq d\leq n-1$ and $1\leq i_1<\ldots<i_d<n$, $x_{i_1\ldots i_d}\neq 0$ if and only if $\{i_1,\ldots,i_d\}=\{w'(1),\ldots,w'(d)\}$, and where $u_{ij}=0$ for all $1\leq i<j\leq n$. It is easily checked that $x\in \widetilde{V}_{w'}$, either from the equations of Proposition \ref{allEquations}, or from the fact that such a point $x$ corresponds to $(gB/B,\psi)\in G/B\times\bfrak$ where $g$ is an invertible matrix with zeroes everywhere but in positions $(w'(j),j)$ for $1\leq j\leq n$, and where $\psi\in\tfrak$.

When $a$ and $b$ are not in the same orbit under $ww'^{-1}$, it is possible to set $\psi=(t_1,\ldots,t_n)\in\tfrak^{ww'^{-1}}$ with $t_a\neq t_b$. This gives $x\in q^{-1}(\tfrak^{ww'^{-1}}+\ufrak)$, while Proposition \ref{propAdditionalEq} forces $x\notin \widetilde{X}_w\inter \widetilde{V}_{w'}$, so $x\notin\widetilde{X}_w$.
\end{proof}

\begin{rem} \label{remConjIsFalseDim}
In the situation of Theorem \ref{thmConjIsFalse}, we actually have a strict inclusion
\[
	\widetilde{X}_w\inter\widetilde{V}_{w'} \subsetneq \widetilde{V}_{w'}\inter q^{-1}(\tfrak^{ww'^{-1}}+\ufrak)
\]
since the (non strict) inclusion is already known (see Proposition \ref{conjImpliesConj1} or the proof of Proposition \ref{conjImpliesConj2}). It is also known that $\widetilde{V}_{w'}\inter q^{-1}(\tfrak^{ww'^{-1}}+\ufrak)$ is irreducible: by an argument similar to the proof of Proposition \ref{conjImpliesConj2}, this statement is equivalent to the analogous statement with $V_{w'}$ (defined in \S\ref{ssecSpringerScheme}) in place of $\widetilde{V}_{w'}$, which is \cite[Lemma 2.3.5]{bhs3}. Therefore $\widetilde{X}_w\inter\widetilde{V}_{w'}$ is a union of irreducible component which have dimension strictly smaller than the dimension of $\widetilde{V}_{w'}\inter q^{-1}(\tfrak^{ww'^{-1}}+\ufrak)$, thus
\[
	\dim\left(\widetilde{X}_w\inter\widetilde{V}_{w'}\right) < \dim\left(\widetilde{V}_{w'}\inter q^{-1}(\tfrak^{ww'^{-1}}+\ufrak)\right) \,.
\]
\end{rem}

\begin{rem} \label{remConcreteCounterExamples}
The easiest example of $w,w'$ meeting the conditions of Theorem \ref{thmConjIsFalse} is when there exist $1\leq i<j<k<l\leq n$ and $1\leq a<b<c<d\leq n$ such that:
\begin{enumerate}
\item
$(w(i),w(j),w(k),w(l))=(d,b,c,a)$.
\item
$(w'(i),w'(j),w'(k),w'(l))=(a,c,b,d)$.
\item
for all $m\in\{1,\ldots,n\}\setminus\{i,j,k,l\}$, $w(m)=w'(m)$.
\end{enumerate}
 The smallest such example is when $(w,w')$ is the first bad pair for $n=4$ (see Theorem \ref{thmBadPairPatterns}). In fact, all cases of interest (\emph{e.g.\ }when $(w,w')$ is a bad pair, see Lemma \ref{lemAdditionalEqBadPairs} below) for $n\leq 5$ meet the conditions of Theorem \ref{thmConjIsFalse}.

For $n=6$, a few cases arise where neither those conditions nor the variant of Remark \ref{remVariationAdditionalEq} are met, for example when $w=(16)(25)$ and $w'=(34)$ (using the cycles notation for elements of the symmetric group $\Scal_6$). In this specific case, the approach of the proof of Theorem \ref{thmConjIsFalse} is not adapted. Indeed, the only equations defining the Schubert cell $BwB/B$ in $G/B$ are $x_6\neq0$, $x_{56}\neq0$, $x_{356}\neq0$, $x_{3456}\neq0$, $x_{23456}\neq0$ and $x_{456}=0$. Because $x_{456}$ is the only zero variable, there are too few Plücker relations and incidence relations that have only two nonzero terms on $BwB/B$ (in the same way equation \eqref{eqIncidenceToSimplify} gives Lemma \ref{lemIncidenceSimplified}). Furthermore those few relations cannot be used to get new equations satisfied on $\widetilde{V}_w$ (and hence $\widetilde{X}_w$) that are nontrivial on $\widetilde{V}_{w'}$ (the same way the equation $t_a=t_b$ of Proposition \ref{propAdditionalEq} is), because all Plücker coordinates but $x_1,x_{12},x_{123},x_{124},x_{1234},x_{12345}$ are zero on $\widetilde{V}_{w'}$. The status of Conjecture \ref{conjLocalModel} remains unknown for this pair.
\end{rem}

Finally, the following lemma confirms that Theorem \ref{thmConjIsFalse} and Theorem \ref{thmConjIsTrue} are consistent with each other.

\begin{lem} \label{lemAdditionalEqBadPairs}
Let $(w',w)\in W$ satisfy the conditions of Theorem \ref{thmConjIsFalse}. Then it is a bad pair in the sense of \S\ref{secGoodPairs}. The same conclusion holds with the conditions of Remark \ref{remVariationAdditionalEq}.
\end{lem}

\begin{proof}
If the pair $(w,w')$ satisfies the conditions of Theorem \ref{thmConjIsFalse}, including the one that states that $a$ and $b$ are not in the same orbit under $ww'^{-1}$, then it is a bad pair in the sense of \S\ref{secGoodPairs}. Indeed, let $\Omega$ the orbit of $b$ under $ww'^{-1}$. By assumption, $a\notin\Omega$, so the last condition of Proposition \ref{propAdditionalEq} implies that
\[
	\left(\Omega\inter w(\{1,\ldots,q+1\})\right)\inter\{1,\ldots,b\} \supsetneq \left(\Omega\inter w'(\{1,\ldots,q+1\})\right)\inter\{1,\ldots,b\}
\]
with the inequality stemming from the presence of $b$ in the left-hand side and not in the right-hand side. Therefore $\Omega\inter w(\{1,\ldots,q+1\})\succeq \Omega\inter w'(\{1,\ldots,q+1\})$ is false (see the second characterisation in Definition \ref{dfnPartialOrder}), so the pair $(w,w')$ is bad according to Proposition \ref{propGoodPairsGLn}.

In the situation of Remark \ref{remVariationAdditionalEq}, take $\Omega$ to be the orbit of $a$ under $ww'^{-1}$. Again, $b\notin \Omega$, so the condition spelled out in Remark \ref{remVariationAdditionalEq} implies that
\[
	\left(\Omega\inter w(\{1,\ldots,q+1\})\right)\inter\{a,\ldots,n\} \subsetneq \left(\Omega\inter w'(\{1,\ldots,q+1\})\right)\inter\{a,\ldots,n\}
\]
from which we get $\abs{(\Omega\inter w(\{1,\ldots,q+1\}))\inter\{a,\ldots,n\}}<\abs{(\Omega\inter w'(\{1,\ldots,q+1\}))\inter\{a,\ldots,n\}}$, or equivalently $\abs{(\Omega\inter w(\{1,\ldots,q+1\}))\inter\{1,\ldots,a-1\}}>\abs{(\Omega\inter w'(\{1,\ldots,q+1\}))\inter\{1,\ldots,a-1\}}$. As before, this means that $\Omega\inter w(\{1,\ldots,q+1\})\succeq \Omega\inter w'(\{1,\ldots,q+1\})$ is false and hence $(w,w')$ is a bad pair.
\end{proof}

\appendix

\numberwithin{thm}{section}

\section{Jordan decomposition and orbits in $\mathfrak{g}$} \label{appendixJordan}

We recall here a few well-known results about Jordan decompositions in Lie algebras. Note that Proposition \ref{propSsDiagConjugates} below is needed in \S\ref{ssecExtensionLevi}. We thank Anne Moreau for its proof, which we couldn't find precisely in the literature.

In the following, $k$ is an algbraically closed field of characteristic $0$ (so that $k$-points and closed points on algebraic varieties over $k$ are the same). For instance, there is no distinction between viewing a Lie algebra as a $k$-scheme (as we do in \S\ref{secProofGoodPairs}) or as a $k$-vector space.

\begin{dfn}
Let $G$ be an algebraic group over $k$ with Lie algebra $\gfrak$. A closed point  $x\in\gfrak$ is called \emph{semisimple} (resp.\ \emph{nilpotent}) if the following property holds. For any morphism of algebraic groups $\phi\colon G\to\GL(V)$ inducing $\diff\phi\colon\gfrak\to\gl(V)$, where $V$ is a finite-dimensional vector space over $k$, the endomorphism $\diff\phi(x)\in\gl(V)\iso\End(V)$ is diagonalisable (resp.\ nilpotent).
\end{dfn}

\begin{thm}[Jordan decomposition] \label{thmJordanDecomposition}
\begin{enumerate}[label=\emph{(\arabic*)},ref=(\arabic*)]
\item \label{itemJordanDecomposition}
Let $G$ be an algebraic group over $k$ with Lie algebra $\gfrak$. For any closed point $x\in\gfrak$, there is a unique couple $(x_\semisimple,x_\nilp)\in\gfrak^2$ with $x_\semisimple$ semisimple and $x_\nilp$ nilpotent such that $[x_\semisimple,x_\nilp]=0$ and $x=x_\semisimple+x_\nilp$. We call this the \emph{Jordan decomposition} of $x$.
\item \label{itemJordanMorphism}
The Jordan decomposition is functorial: for any morphism $\phi\colon G\to H$ of algebraic groups over $k$ inducing $\diff\phi\colon\gfrak\to\hfrak$ and any closed point $x\in\gfrak$, we have
\[
	\diff\phi(x_\semisimple)=(\diff\phi(x))_\semisimple,\quad \diff\phi(x_\nilp)=(\diff\phi(x))_\nilp \,.
\]
\end{enumerate}
\end{thm}

\begin{proof}
The existence of a functorial Jordan decomposition is a combination of 23.6.2, 23.6.3 and 23.6.4 of \cite{tauvelYu}. The uniqueness property in \ref{itemJordanDecomposition} is given by the fact that for any algebraic group $G$, there exists a finite-dimensional vector space $V$ with a morphism of algebraic groups $\phi\colon G\to\GL(V)$ such that $\diff\phi\colon\gfrak\to\gl(V)$ is injective (see 22.1.5 and 24.4.1 of \emph{loc.\ cit.}).
\end{proof}

\begin{prop} \label{propJordanReductive}
Let $G$ be a connected reductive group over $k$, let $\gfrak$ be the Lie algebra of $G$ and $\zfrak$ be the centre of $\gfrak$. There is a decomposition of $k$-vector spaces
\begin{equation} \label{eqReductiveLie}
	\gfrak = [\gfrak,\gfrak]\oplus\zfrak
\end{equation}
and for any closed point $x\in\gfrak$, the following conditions are equivalent:
\begin{enumerate}[label=\emph{(\roman*)},ref=(\roman*)]
\item \label{itemSemisimple}
$x$ is semisimple (resp.\ nilpotent).
\item \label{itemAdSemisimple}
$\ad(x)\in\gl(\gfrak)$ is semisimple (resp.\ $\ad(x)\in\gl(\gfrak)$ is nilpotent and the component of $x$ in $\zfrak$ in \eqref{eqReductiveLie} is $0$).
\item \label{itemSigmaSemisimple}
For any finite-dimensional representation $\sigma\colon\gfrak\to\gl(V)$, the endomorphism $\sigma(x)\in\gl(V)\iso\End(V)$ is diagonalisable (resp.\ nilpotent).
\end{enumerate}
\end{prop}

\begin{proof}
The decomposition \eqref{eqReductiveLie} is 27.2.2 and 20.5.5 of \cite{tauvelYu}. Note that $\zfrak=0$ if and only if $G$ is semisimple (see 20.1.2 and 20.5.4 of \emph{loc.\ cit.}). The implication \ref{itemSigmaSemisimple}$\implies$\ref{itemSemisimple} is a consequence of the definitions. We now prove \ref{itemSemisimple}$\implies$\ref{itemAdSemisimple}. Since $\ad\colon\gfrak\to\gl(\gfrak)$ is the differential of $\Ad\colon G\to\GL(\gfrak)$ (see 23.5.5 of \emph{loc.\ cit.}), we only need to prove the following statement (the rest follows from the definitions): if $x\in\gfrak$ decomposes as $x=y+z$ via \eqref{eqReductiveLie} and $x$ is nilpotent, then $z=0$. Since $[\gfrak,\gfrak]$ is the Lie algebra of the derived group $(G,G)$ (see 27.2.3 of \emph{loc.\ cit.}), we can consider the Jordan decomposition $y=y_\semisimple+y_\nilp$ of $y$ in $[\gfrak,\gfrak]$. Applying Theorem \ref{thmJordanDecomposition}\ref{itemJordanMorphism} to the inclusion $(G,G)\inj G$, we see that $y_\semisimple$ (resp.\ $y_\nilp$) is semisimple (resp.\ nilpotent) in $\gfrak$. As $z$ is semisimple in $\gfrak$ (see 27.2.2(ii) of \emph{loc.\ cit.}) and commutes with everything, we see that $y_\semisimple+z$ is semisimple and $x=(y_\semisimple+z)+y_\nilp$ is the Jordan decomposition of $x$. By uniqueness, $y_\semisimple+z=0$ hence $z=0$. This finishes the proof of \ref{itemSemisimple}$\implies$\ref{itemAdSemisimple}. Finally, \ref{itemAdSemisimple}$\implies$\ref{itemSigmaSemisimple} follows from 20.5.7 and 20.5.8 of \emph{loc.\ cit.}\ (note that their definition of semisimple and nilpotent elements differs from ours).
\end{proof}

\begin{lem} \label{lemCartanConjugates}
Let $G$ be a semisimple algebraic group over $k$, with semisimple Lie algebra $\gfrak$. Let $\hfrak,\hfrak'$ be Cartan subalgebras of $\gfrak$. Then there exists a closed point $g\in G$ such that $\hfrak'=\Ad_G(g)\hfrak$.
\end{lem}

\begin{proof}
We proceed exactly as in \cite[29.2.3(ii)]{tauvelYu}. We first recall the notation of 19.1.1 of \emph{loc.\ cit}. For a closed point $x\in\gfrak$, we write
\[
	\gfrak^0(x) \coloneqq \cap_{n\in\Z_{\geq0}} \ker((\ad{x})^n)
\]
and we say that $x$ is \emph{generic} if $\dim{\gfrak^0(x)}$ (which is the multiplicity of $0$ as a root of the characteristic polynomial of $\ad(x)$) is minimal among all $x\in\gfrak$. Call $\gfrak_\gen$ the set of generic elements; it is a Zariski-open subset of $\gfrak$ since the coefficients of the characteristic polynomial are a polynomial expression in $x\in\gfrak$. By 29.2.3(i) of \emph{loc.~cit.\ }a Lie subalgebra of $\gfrak$ is Cartan if and only if it is of the form $\gfrak^0(x)$ for some $x\in\gfrak_\gen$. Hence choose $x\in\gfrak_\gen$ such that $\hfrak=\gfrak^0(x)$. Consider the morphism
\[
	\appl{\theta}{G\times\hfrak}{\gfrak}{(g,y)}{\Ad(g)y}
\]
whose differential at the point $(e,x)\in G\times\hfrak$ ($e$ is the neutral element in $G$) is
\[
	\appl{\diff\theta_{(e,x)}}{\gfrak\times\hfrak}{\gfrak}{(z,y)}{[z,x]+y}
\]
(see 29.1.4(i) and 23.5.5 of \emph{loc.\ cit.}). By definition, $\hfrak=\gfrak^0(x)$ is a subspace of $\gfrak$ stable by $\ad(x)$, and $\ad(x)$ induces an automorphism on $\gfrak/\gfrak^{0}(x)=\gfrak/\hfrak$. Therefore, as in 29.2.1 of \emph{loc.\ cit.}, $\diff\theta_{(e,x)}$ is surjective, so $\theta$ is dominant by 16.5.7 of \emph{loc.\ cit.}, and from 15.4.2 of \emph{loc.\ cit.}, the image $\Ad(G)\hfrak$ contains an open dense subset of $\gfrak$. Since the same can be said of $\Ad(G)\hfrak'$, we have in $\gfrak$
\[
	\Ad(G)\hfrak\inter\Ad(G)\hfrak'\inter\gfrak_\gen \neq \emptyset \,.
\]
Therefore there is $t\in\gfrak_\gen$ and $g,g'\in G$ such that $\Ad(g)t\in\hfrak$ and $\Ad(g')t\in\hfrak'$. Since $\Ad(g)t$ and $\Ad(g')t$ are generic as well, 19.8.5 of \emph{loc.\ cit.\ }gives
\[
	\hfrak=\gfrak^0(\Ad(g)t) \,, \quad \hfrak'=\gfrak^0(\Ad(g')t)
\]
hence $\hfrak'=\Ad(g'g^{-1})\hfrak$.
\end{proof}

For any vector space $V$ over any field, write $V^*$ for its dual, $S(V)$ for the symmetric algebra of $V$ and $S^n(V)$ for the graded piece of $S(V)$ of degree $n$ (where $n\in\Z_{\geq0}$). Recall that if $V$ is finite dimensional, $S(V^*)$ can be interpreted as the algebra of polynomial functions on $V$.

Let $\gfrak$ be a semisimple Lie algebra, the action of $\gfrak$ on itself by adjunction induces an action on $S(\gfrak^*)$. We can then consider the algebra $S(\gfrak^*)^\gfrak$ of invariants under this action. Similarly, given a Cartan subalgebra $\hfrak$ of $\gfrak$, the Weyl group $W$ of $(\gfrak,\hfrak)$ acts on $\hfrak$ (see \cite[\S19.8]{tauvelYu}), so we can consider $S(\hfrak^*)^W$ as well.

\begin{thm} \label{thmSymmetricAlgebras}
Let $\gfrak$ be a semisimple Lie algebra over $k$, $\hfrak$ be a Cartan subalgebra of $\gfrak$ and $W$ be the Weyl group of the root system associated to $(\gfrak,\hfrak)$.
\begin{enumerate}[label=\emph{(\arabic*)},ref=(\arabic*)]
\item
The restriction morphism $S(\gfrak^*)\to S(\hfrak^*)$ induces a surjection $S(\gfrak^*)^\gfrak\surj S(\hfrak^*)^W$.
\item
For $n\in\Z_{\geq0}$, the vector space $S^n(\gfrak^*)$ is spanned by functions of the form $x\longmapsto\tr(\sigma(x)^n)$, where $\sigma$ is a finite-dimensional representation of $\gfrak$ over $k$.
\end{enumerate}
\end{thm}

\begin{proof}
This is \cite[31.2.6]{tauvelYu}.
\end{proof}

\begin{lem} \label{lemSameOrbitW}
Let $\gfrak$ be a semisimple Lie algebra over $k$, $\hfrak$ be a Cartan subalgebra of $\gfrak$ and $W$ be the Weyl group of the root system associated with $(\gfrak,\hfrak)$. Let $x,y\in\hfrak$ be closed points such that $f(x)=f(y)$ for all $f\in S(\hfrak^*)^W$. Then $x$ and $y$ are in the same orbit for the action of $W$ on $\hfrak$.
\end{lem}

\begin{proof}
This is \cite[34.2.1]{tauvelYu}.
\end{proof}

\begin{prop} \label{propSameOrbitCriterion}
Let $G$ be a semisimple algebraic group over $k$ with Lie algebra $\gfrak$, let $x,y\in\gfrak$ be semisimple elements. Assume that $\tr(\sigma(x)^n)=\tr(\sigma(y)^n)$ for all finite-dimensional representation $\sigma$ of $\gfrak$ and all $n\in\Z_{\geq0}$. Then $x$ and $y$ are in the same orbit for the action of $G$ on $\gfrak$ by adjunction.
\end{prop}

\begin{proof}
We know from \cite[20.6.3(ii)]{tauvelYu} that $x$ and $y$ are each contained in a Cartan subalgebra of $\gfrak$. According to Lemma \ref{lemCartanConjugates}, we can assume that $x$ and $y$ lie in the same Cartan subalgebra $\hfrak$. Let $W$ be the Weyl group associated to the root system of $(\gfrak,\hfrak)$. It follows from Theorem \ref{thmSymmetricAlgebras} and Lemma \ref{lemSameOrbitW} that $x$ and $y$ are in the same orbit for the action of $W$ on $\hfrak$.

There is a maximal torus $T$ of $G$ such that $\hfrak$ is the Lie algebra of $T$ (see 29.2.5 of \emph{loc.\ cit.\ }and \cite[26.2.A]{humphreys}). The action of the Weyl group $W(G,T)=N_G(T)/T$ on $T$ induces, through differentiation, an action of $\hfrak$; this identifies $W(G,T)$ with $W=W(\gfrak,\hfrak)$. Hence $W$ is isomorphic to $N_G(T)/T$, and the corresponding action of $w\in N_G(T)/T$ on $\hfrak$ is given by the adjunction by a lifting of $w$ in $N_G(T)$. Therefore there is an element $n\in N_G(T)\subset G$ such that $x=\Ad_G(n)y$.
\end{proof}

\begin{lem} \label{lemEqualTraces}
Let $\gfrak$ be a semisimple Lie algebra over $k$, let $\sigma\colon\gfrak\to\gl(V)$ be a finite-dimensional representation and let $n\in\Z_{\geq0}$.
\begin{enumerate}[label=\emph{(\arabic*)},ref=(\arabic*)]
\item \label{itemSemisimpleTrace}
Let $x\in\gfrak$ be a closed point with Jordan decomposition $x=x_\semisimple+x_\nilp$. Then $\tr(\sigma(x)^n)=\tr(\sigma(x_\semisimple)^n)$.
\item \label{itemDiagonalTrace}
Let $\bfrak$ be a Borel subalgebra of $\gfrak$, let $x,t\in\bfrak$ be closed points such that $x-t$ is nilpotent. Then $\tr(\sigma(x)^n)=\tr(\sigma(t)^n)$.
\end{enumerate}
\end{lem}

\begin{proof}
The assertion \ref{itemSemisimpleTrace} is \cite[34.2.3]{tauvelYu}. As for \ref{itemDiagonalTrace}, note that by semisimplicity $\gfrak=[\gfrak,\gfrak]$, so $\sigma(\bfrak)$ is a solvable subalgebra of $\slfrak(V)=[\gl(V),\gl(V)]$. Hence $\sigma(\bfrak)$ is contained in a Borel subalgebra $\bfrak'$ of $\slfrak(V)$. Using 29.4.7 of \emph{loc.\ cit.}, one can find a basis of $V$ such that $\bfrak'$ is the set of upper triangular matrices in $\slfrak(V)$. Since $x-t$ is nilpotent, by Proposition \ref{propJordanReductive} the matrix $\sigma(x)-\sigma(t)$ is nilpotent and in $\bfrak'$ hence is strictly upper triangular. Therefore $\sigma(x)$ and $\sigma(t)$ are upper triangular with the same diagonal in $\bfrak'$. Thus so are $\sigma(x)^n$ and $\sigma(t)^n$, which therefore have the same trace on $V$.
\end{proof}

\begin{prop} \label{propSsDiagConjugates}
Let $G$ be a connected reductive group over $k$ with a maximal torus $T$ and a Borel subgroup $B$ containing $T$, let $\gfrak$, $\tfrak$ and $\bfrak$ be their respective Lie algebra. Let $x\in\bfrak$  be a closed point with Jordan decomposition $x=x_\semisimple+x_\nilp$.
\begin{enumerate}[label=\emph{(\arabic*)},ref=(\arabic*)]
\item \label{itemComponentsInBorel}
The components $x_\semisimple$ and $x_\nilp$ are in $\bfrak$.
\item \label{itemSameOrbit}
Let $t\in\tfrak$ and $u\in\bfrak$ such that $u$ is nilpotent and $x=t+u\ \in \bfrak$. Then there exists $g\in(G,G)$ such that $x_\semisimple=\Ad_G(g)t$.
\end{enumerate}
\end{prop}

\begin{proof}
The inclusion $\bfrak\subset\gfrak$ is the differential of the inclusion $B\subset G$, hence by Theorem \ref{thmJordanDecomposition}\ref{itemJordanMorphism}, the Jordan decomposition of $x$ viewed as an element of $\bfrak$ is the same as that of $x$ viewed as an element of $\gfrak$. This proves \ref{itemComponentsInBorel}.

For \ref{itemSameOrbit}, the decomposition \eqref{eqReductiveLie} gives $x=y+z$ with $y\in[\gfrak,\gfrak]$ and $z\in\zfrak$. Proposition \ref{propJordanReductive} gives $u\in[\gfrak,\gfrak]$, hence $t=s+z$ for some semisimple element $s\in[\gfrak,\gfrak]$. Since $\zfrak$ is the Lie algebra of the center $Z(G)$ of $G$ (see \cite[24.4.2(ii)]{tauvelYu}) and $Z(G)\subset T\subset B$ (see \cite[26.2.A]{humphreys}), we have $\zfrak\subset\tfrak\subset\bfrak$, hence $y=x-z\in\bfrak$ and $s=t-z\in\bfrak$, thus $y,s\in\bfrak'\coloneqq\bfrak\inter[\gfrak,\gfrak]$.

Note that $\bfrak=\bfrak'\oplus\zfrak$ since $\zfrak\subset\bfrak$. Let $\bfrak_1'$ be a solvable Lie subalgebra of $[\gfrak,\gfrak]$ containing $\bfrak'$ and $\bfrak_1\coloneqq\bfrak_1'\oplus\zfrak\subseteq\gfrak$. Then $\bfrak_1$ is a Lie subalgebra of $\gfrak$ containing $\bfrak$ and it is solvable since $[\bfrak_1,\bfrak_1]=[\bfrak_1',\bfrak_1']$. Therefore $\bfrak_1=\bfrak$, and $\bfrak_1'=\bfrak'$. This shows that  $\bfrak'$ is a Borel subalgebra of $[\gfrak,\gfrak]$. Consequently, since $y=x-z=t+u-z=s+u$, we can apply Lemma \ref{lemEqualTraces} and Proposition \ref{propSameOrbitCriterion} to the semisimple group $(G,G)$ and the elements $y$, $s$ of its Lie algebra $[\gfrak,\gfrak]$ (see \cite[27.2.3]{tauvelYu}). This gives $g\in (G,G)\subset G$ such that
\[
	y_\semisimple=\Ad_{(G,G)}(g)s=\Ad_G(g)s \,.
\]
From the proof of Proposition \ref{propJordanReductive}, we know that $x_\semisimple=y_\semisimple+z$. Since the conjugation by $g$ is the identity on $Z(G)$, its differential $\Ad_G(g)$ is trivial on $\zfrak$, so
\[
	x_\semisimple=y_\semisimple+z=\Ad_G(g)(s+z)=\Ad_G(g)t
\]
which finishes the proof.
\end{proof}

\bibliographystyle{amsplain}
\bibliography{biblio_trianguline_springer}

\end{document}